\documentclass[12 pt]{article}
\pdfoutput=1 

\usepackage{geometry}
\usepackage[utf8]{inputenc}
\usepackage[T1]{fontenc}
\usepackage{tgschola} 
\usepackage[english]{babel}

\usepackage{amsmath}
\usepackage{amssymb}
\usepackage{amsthm}
\usepackage{mathtools}
\usepackage{mathrsfs}
\usepackage{dsfont}
\usepackage{stmaryrd}

\usepackage{csquotes} 
\usepackage{xcolor} 
\usepackage{soul} 
\usepackage{framed} 

\usepackage{enumitem} 

\usepackage{graphicx} 
\usepackage{float}
\usepackage{caption}
\usepackage{tikz}
\usetikzlibrary{shapes.geometric}
\usetikzlibrary{arrows}

\usepackage{tikz-cd}

\usepackage{pgfplots}
\pgfplotsset{compat=1.18}

\usepackage[backend=biber,
            style=alphabetic,
            maxnames=10,
            minnames=2,
            maxalphanames=5,
            minalphanames=2,
            backref=true,
            backrefstyle=none]{biblatex}

\usepackage{microtype} 

\usepackage[colorinlistoftodos]{todonotes} 
\usepackage{pdflscape} 

\usepackage{hyperref}
\hypersetup{
    breaklinks=true,
    colorlinks=true,
    linkcolor=cyan,     
    urlcolor={blue!60},
    citecolor=teal,
}

\usepackage[hyphenbreaks]{breakurl} 

\usepackage{adjustbox} 

\usepackage{totcount} 

\usepackage{upgreek} 

\usepackage{fancyhdr} 

\numberwithin{equation}{section}

\theoremstyle{plain} 

\newtheorem{prop}[equation]{Proposition}
\newtheorem{cor}[equation]{Corollary}
\newtheorem{thm}[equation]{Theorem}
\newtheorem{lemma}[equation]{Lemma}
\newtheorem{rmk}[equation]{Remark}

\newtheorem*{thm*}{Theorem}

\theoremstyle{definition}

\definecolor{lightblueforque}{HTML}{8fcaff}

\newtheorem{protoerr}{ERROR}

\newtheorem{protoque}[protoerr]{QUESTION}

\newtheorem{prototbe}[protoerr]{To be expanded}

\newtheorem{protoysr}[protoerr]{You should remember}

\definecolor{colarr2}{rgb}{1.0,0,0}
\definecolor{colarr3}{rgb}{0.5,0,0.5}
\definecolor{colarr4}{rgb}{0,0,1.0}
\definecolor{colarr5}{rgb}{1,0.6,0}
\definecolor{colarr6}{rgb}{0,1.0,0}
\definecolor{colarr7}{rgb}{0.5,0.5,0}
\definecolor{colarr8}{rgb}{0,0.5,0.5}
\definecolor{colarr12}{RGB}{255,102,102}
\definecolor{colarr18}{RGB}{150,39,240}

\definecolor{faintcolarr2}{rgb}{1.0,0.5,0.3}
\definecolor{faintcolarr4}{rgb}{0.5,0.5,1.0}
\definecolor{faintcolarr8}{rgb}{0.2,1.0,1.0}
\definecolor{faintcolarr16}{rgb}{0.2,1.0,0.2}

\definecolor{colnewp}{rgb}{0.4,0,0.7}

\definecolor{colsigma}{rgb}{0.1,1.0,0.3}

\definecolor{weight0}{HTML}{14b992}
\definecolor{weight1}{HTML}{6BCE0E}
\definecolor{weight2}{HTML}{f58424}
\definecolor{weight3}{HTML}{c1c530}
\definecolor{weight4}{HTML}{f2b0f8}
\definecolor{weight5}{HTML}{0987f5}
\definecolor{weight6}{HTML}{0a734b}
\definecolor{weight7}{HTML}{b1fae4}
\definecolor{weight8}{HTML}{1de6e3}
\definecolor{weight9}{HTML}{a0c089}
\definecolor{weight10}{HTML}{0987f5}
\definecolor{weight12}{HTML}{0a734b}
\definecolor{weight14}{HTML}{b1fae4}
\definecolor{weight16}{HTML}{1de6e3}
\definecolor{weight18}{HTML}{a0c089}
\definecolor{weight20}{HTML}{170de7}

\definecolor{colarrow}{HTML}{ADD8E6}

\definecolor{almostwhite}{RGB}{248,248,248}

\definecolor{color1}{rgb}{1,0.84,0}
\definecolor{color2}{rgb}{0.08,0.4,0.75}
\definecolor{color3}{rgb}{0.83,0.19,0.19}
\definecolor{color4}{rgb}{0.23,0.88,0.28}
\definecolor{color5}{rgb}{0.8,0,0.6}

\title{\texorpdfstring{Modulo $\tau^{p-1}$ motivic Hochschild homology of modulo $p$ motivic cohomology}{Modulo tau\^{}(p-1) motivic Hochschild homology of modulo p motivic cohomology}}

\author{Federico Ernesto Mocchetti\thanks{Università degli Studi di Milano - Universit{\"a}t Osnabr{\"u}ck- \href{mailto:federico.mocchetti@unimi.it}{federico.mocchetti@unimi.it}}}

\date{\today}

\begin{document}

\maketitle

\begin{abstract}
    { \footnotesize
    We use the motivic Greenlees spectral sequence from \cite{Mocchetti2024b} to compute Hochschild homology in the stable motivic homotopy category over an algebraically closed field. Our target is $MHH(M\mathbb{Z}/p)/\tau^{p-1}$, where $M\mathbb{Z}/p$ is modulo $p$ motivic cohomology, $p$ a prime number different from the characteristic of the base.
    }
\end{abstract}

\setcounter{section}{-1}
\section{Introduction}

This paper follows \cite{Mocchetti2024b}.

In the first paper, we set up a spectral sequence in motivic homotopy theory mimicking a classical construction by John Greenlees \cite{Greenlees2016}. Recall that in stable motivic homotopy theory, we define a bi-graded family of spheres
\[
    S^{a,b}= (S^1_s)^{a-b} \wedge (\mathbb{G}_m)^b
\]
where $S^1_s$ is the constant sheaf at the simplicial model for the circumference and $\mathbb{G}_m$ is the sheaf represented by the multiplicative subgroup of the affine line. Whenever we distinguish between the two degrees, we will call the first one \textit{degree} and the second one \textit{weight}.

To set up our spectral sequence, we first fix an $E_{\infty}$ algebra $Q$ in $SH(S)$ and define a certain $t$-structure on $Q$-modules, see \cite[Definition 1.2.]{Mocchetti2024b}; if the homotopy ring of $Q$ has a nice shape (for instance, it is concentrated in degree zero), we can identify for any $Q$-module $R$:
\[
    \pi_{i,j}(fib(R_{\leq n} \to R_{\leq n-1})) \cong \begin{cases}
        \pi_{n,j} R & \text{for } i=n\\
        0 & \text{else.}
    \end{cases}
\]

Our main result goes as follows (confront with \cite[Proposition 1.22.]{Mocchetti2024b}):

\begin{prop}\label{prop:generic.spectral.seq}
Suppose we are given a map of commutative motivic ring spectra in $R \to Q$ in $Alg_{E_{\infty}}(SH(S))_{Q/}$ and that $R\cong  R_{\geq 0}$ is connective with respect to our $t$-structure. Then there is a strongly convergent multiplicative spectral sequence:
\[
    E^2_{s,t,*}= \pi_{s+t,*}(Q \wedge_{R} fib(R_{\leq t} \to R_{\leq t-1})) \Rightarrow \pi_{s+t,*}(\lim_{\overleftarrow{n}} Q_n)
\]
where $Q_n= Q \wedge_R R_{\leq n}$; differentials are of the form:
\[
    d^r: E^r_{s,t,*} \to E^r_{s-r,t+r-1,*}.
\]
    
\end{prop} 

By construction, the spectral sequence lives in the upper half-plane.

We use this spectral sequence to study motivic Hochschild homology. Given $Q \in SH(S)$ a ring spectrum, we define its motivic Hochschild homology as:
\[
    MHH(Q) \cong Q\wedge_{Q \wedge Q^{op}}Q.
\]
Recall that if $S=Spec(k)$ with $k=\bar{k}$ and algebraically closed field and $Q=M\mathbb{Z}/p$ is modulo $p$ motivic cohomology, $p \neq char(k)$, then:
\[
    \pi_{*,*}(M\mathbb{Z}/p) \cong \mathbb{F}_p[\tau] \qquad |\tau|=(0,-1)
\]
(see section \ref{sec:setup} for a digression on this result). In \cite{Mocchetti2024b} we identified the homotopy groups of a $\tau$-inverted version of motivic Hochschild homology for this choice of $S$ and $Q$. In this paper, we use the motivic Greenlees spectral sequence above to compute:
\[
    \pi_{*,*}(MHH(M\mathbb{Z}/p)/\tau^{p-1}).
\]

Observe that the knowledge of the $\tau$-inverted and mod-$\tau^{p-1}$ versions can be used to recover the homotopy ring of $MHH(M\mathbb{Z}/p)$ \cite[Section 3.3]{DHOO2022}.

\clearpage
\relax

\tableofcontents
\clearpage

\section{Setting up the spectral sequence}\label{sec:setup}
We begin by proving that we can produce a spectral sequence from \cite[Proposition 1.22.]{Mocchetti2024b} involving mod-$\tau^{p-1}$ motivic Hochschild homology. First, we show that the spectrum $(M\mathbb{Z}/p)/{\tau^{p-1}}$ carries an $E_{\infty}$ ring structure whenever our base scheme is an algebraically closed field of characteristic coprime to $p$. We will prove more generally that $(M\mathbb{Z}/p)/{\tau^{n}}$ is $E_{\infty}$ in $SH(Spec(\mathbb{Z}))$, hence over any scheme by base change, for any positive integer $n$. At first, we recall what happens when we work in $SH(Spec(k))$ for $k$ an algebraically closed field of characteristic not $p$, expanding an argument from \cite[Section 3.2]{DHOO2022}. 

We begin by recalling some facts about the homotopy ring of modulo $p$ motivic cohomology, to clarify the origin of the element $\tau$. We actually look at the motivic cohomology of the base with $M\mathbb{Z}/p$ coefficients, which is the dual of $\pi_{*,*}(M\mathbb{Z}/p)$: in the end, we will have to change the sign to all degrees.

Since we are in the stable motivic homotopy category over the spectrum of a field $k$ containing $1/p$, we can apply the norm residue isomorphism theorem \cite{Voevodsky2011, Voevodsky2003b}, so there is a description:
\[
    H^{a,b}(Spec(k), M\mathbb{Z}/p) \cong 
    \begin{cases}
        H^a_{\textit{ét}}(k, \mu_p^{\otimes b}) &\text{for } a \leq b,\, b \geq 0 \\
        0 & \text{elsewhere.}
    \end{cases}
\]
Now, as $k$ is algebraically closed of characteristic different from $p$, it contains a primitive $p$-th root of unity. Then the above description specialises to:
\[
    H^{*,*}(Spec(k), M\mathbb{Z}/p) \cong K^M_*(k)/p[\tau]
\]
for some $\tau$ in degree $(0,1)$; the module $K^M_a(F)/p$ is placed in bidegree $(a,a)$.
In fact, the presence of a primitive $p$-th root of unity in $k$ implies that the sheaf of the $p$-th roots of unity is $\mu_p$ is constant and isomorphic to $\mathbb{Z}/p\mathbb{Z}$. Hence any tensor power:
    \[
        \mu_p^{\otimes b}= \underbrace{\mu_p \otimes_{\mathbb{Z}/p\mathbb{Z}} \ldots \otimes_{\mathbb{Z}/p\mathbb{Z}} \mu_p}_{b \text{ times}} \cong \mathbb{Z}/p\mathbb{Z}
    \]
is as well. The identification
\[
    H^a_{\textit{ét}}(k, \mu_p) \cong K^M_a(k)/p
\]
was proven by Voevodski in the papers mentioned above.

If $k = \bar{k}$ is moreover algebraically closed, its Milnor $K$-theory $K^M_*(k) \cong \mathbb{Z}$ is isomorphic to the integers (concentrated in degree 0); this follows straightforwardly from the definition. Hence: 
\begin{equation} \label{eq:motivic.coh.alg.closed}
    H^{*,*}(Spec(k), M\mathbb{Z}/p) \cong \mathbb{F}_p[\tau].
\end{equation}

To show that $(M\mathbb{Z}/p)/\tau^{n}$ is an $E_{\infty}$ ring spectrum, we approach it from a different category, where commutativity is more immediate.
By theorem \cite[Theorem C.2]{Spitzweck2018} $M\mathbb{Z}$ is a strongly periodizable motivic ring spectrum, hence, by \cite[Corollary C.3]{Spitzweck2018} we can see the category of derived Tate motives as the derived category of modules over a certain $E_{\infty}$ algebra in graded complexes of abelian groups. By base changing to $\mathbb{Z}/p$, we get to see $M\mathbb{Z}/p$ as an $E_{\infty}$ algebra $\mathcal{M}$ in graded complexes of $\mathbb{F}_p$ vector spaces; the equivalence, in particular, allows to identify the homology of $\mathcal{M}$ with respect to the canonical $t$-structure on graded complexes of abelian groups with the $M\mathbb{Z}/p$-cohomology of the base scheme, which 
 in our case is the graded ring $\mathbb{F}_p[\tau]$ with $|\tau|=(0,1)$.

In the category of graded complexes of $\mathbb{F}_p$ vector spaces, we consider the graded complex $\mathbf{M}$ which is trivial anywhere but in degree $0$, where it coincides with $\mathbb{F}_p[\tau]$, and has trivial differentials. Observe that $\mathbf{M}$ is a formal model for the graded ring $\mathbb{F}_p[\tau]$; in other words, $\mathbf{M}$ is quasi-isomorphic to its homology, which is the graded ring $\mathbb{F}_p[\tau]$.
This model is strictly commutative (hence admits an  $E_{\infty}$ structure). Observe that $\mathcal{M}$ and $\mathbf{M}$ are also equivalent as $E_{\infty}$ algebras in the category of graded complexes of $\mathbb{F}_p$ vector spaces, as they both arise as zero truncations of the free $E_{\infty}$ $\mathbb{F}_p$-algebra on a single generator $\tau$ in degree $(0,1)$. 

Now, given any positive integer $n$, the formal model $\mathbf{M}$ comes by construction with a strictly commutative map $\mathbf{M} \to \mathbf{M}/\tau^n$, induced by the quotient map in commutative rings $\mathbb{F}_p[\tau] \to \mathbb{F}_p[\tau]/\tau^n$.
Using the equivalence of \cite[Theorem C.2]{Spitzweck2018} in the other direction, we obtain a map $M\mathbb{Z}/p \to (M\mathbb{Z}/p)/\tau^n$ of $E_{\infty}$ algebras in the category of motives as well.
\vspace{1em}

Suppose now our base scheme is the spectrum of a Dedekind domain $B$ where $p$ is invertible. By results of Geisser \cite[Theorem 1.2]{Geisser2004} and Spitzweck \cite[Theorem 3.9]{Spitzweck2018} there is still an identification, for each integer $l$:
\[
    M\mathbb{Z}/p(l) \cong \tau_{\leq l}R\epsilon_*( \mathbb{Z}/p(l)_{\textit{ét}}) \cong \tau_{\leq l}R\epsilon_*( \mu_p^{\otimes l}[0])
\]
where $\epsilon$ is the change of site map from the étale to the Zariski site. Hence we can still link the motivic cohomology of $M\mathbb{Z}/p$ with the étale cohomology of the sheaf of roots of unity.

As the identifications from \cite[Theorem C.2, Corollary C.3]{Spitzweck2018} work in this setting as well, we still can identify $M\mathbb{Z}/p$ with an $E_{\infty}$ algebra $\mathcal{M}$ in graded complexes of $\mathbb{F}_p$ vector spaces. We consider the canonical $t$-structure on this category and call $c_0$ the connective cover functor. By abuse of notation, we also indicate with $c_0$ the connective cover functor for the $t$-structure on motives induced via \cite[Corollary C.3]{Spitzweck2018}. By the above-mentioned result by Geisser, the homology  of $c_0(\mathcal{M})$ can be recovered from the degree zero étale cohomology of $B$ with coefficients in the powers of the $p$-th roots of unity sheaf $\mu_p$:
\[
    H_{a,b}(c_0(\mathcal{M})) \cong H^{a,b}(Spec(B), c_0(M\mathbb{Z}/p)) 
    \cong \begin{cases}
        H^0_{\textit{ét}}(B, \mu_p^{\otimes b}) & \text{for } a=0\\
        0 & \text{else.}
    \end{cases}
\]
We study then the $B$ points $H^0_{\textit{ét}}(B, \mu_p^b) \cong \mu_p^{\otimes l}(B)$ for any non-negative integer $l$. 
 
To do so, consider a geometric point $\bar{x}$ of $Spec(B)$. As for each integer $l$ the sheaf $\mu_p^{\otimes l}$ is a local system and $Spec(B)$ is connected, we can associate to $\mu_p^{\otimes l}$ an action of the étale fundamental group $\pi_1^{\textit{ét}}(Spec(R), \bar{x})$ on the geometric fibre $\mu_p^{\otimes l}(\bar{x})$, such that the fixed points correspond to $\mu_p^{\otimes l}(B)$. 
Observe now that, as $p$ is invertible in $B$, $\bar{x}$ corresponds to the spectrum of an algebraically closed field where $p$ is invertible. Hence, as seen in the previous part, the fibre of any tensor power of $\mu_p$ at $\bar{x}$ is isomorphic to the constant sheaf $\mathbb{F}_p$. If $\mu_p$ corresponds to a representation:
\[
    \rho: \pi_1^{\textit{ét}}(Spec(B), \bar{x}) \to Aut(\mathbb{F}_p)\cong \mathbb{F}_p^*
\]
the sheaf $\mu_p^{\otimes l}$ is then associated with the $l$-th power of this representation $\rho$.

If $\rho^l$ is trivial, then all points are fixed: $\mu_p^{\otimes l}(B) \cong \mathbb{F}_p$; in all the other cases, the action has no fixed points, so $\mu_p^{\otimes l}(B) \cong 0$.
Hence $H_{*,*}(c_0(\mathcal{M})) \cong \mathbb{F}_p[\tau]$, where $|\tau|=(0,m)$, $m$ being the smallest positive integer such that the sheaf $\mu_p^{\otimes m} \cong \mathbb{Z}/p\mathbb{Z}$ trivialises.

As in the previous case, there is then an obvious formal model $\mathbf{M}_0$ for $c_0(\mathcal{M})$, given by the graded $\mathbb{F}_p$ vector space $\mathbb{F}_p[\tau]$ in position zero, and everywhere else a trivial entry. This model comes with a strictly commutative map:
\[
    \mathbf{M}_0 \to \mathbf{M}_0/\tau^n
\]
for each positive integer $n$, induced by the quotient map in graded $\mathbb{F}_p$ algebras $\mathbb{F}_p[\tau] \to \mathbb{F}_p[\tau]/\tau^n$. We can thus consider a pushout in $E_{\infty}$ complexes of graded $\mathbb{F}_p$ vector spaces:
\[
    \begin{tikzcd}
        \mathbf{M}_0
            \arrow[r]
            \arrow[d]
            \arrow[dr, phantom, "\ulcorner", pos=1]
        & \mathcal{M}
            \arrow[d]
        \\ \mathbf{M}_0 / \tau^n
            \arrow[r]
        & \mathcal{M}/\tau^n
    \end{tikzcd}
\]
The image of $\mathcal{M}/\tau^n$ in Tate motives via the equivalence of \cite{Spitzweck2018} is the desired $E_{\infty}$ object $(M\mathbb{Z}/p)/\tau^n$.
\vspace{1em}

Finally, suppose we are working over $Spec(\mathbb{Z})$.
Let $i: Spec(\mathbb{F}_p) \to Spec(\mathbb{Z})$ be the inclusion of the closed point $p$, and let $j:Spec(\mathbb{Z}[1/p]) \to Spec(\mathbb{Z})$ be the open complement. We repeat the constructions of the above points to these three schemes.  By \cite[Corollary 3.3]{Geisser2004} there is a distinguished triangle:
\[
    \begin{adjustbox}{max width=\textwidth, center}
    $\displaystyle
        i_{*} M\mathbb{Z}/p(n-1)^{Spec(\mathbb{F}_p)}[-2]\to M\mathbb{Z}/p(n)^{Spec(\mathbb{Z})} \to j_{*} M\mathbb{Z}/p(n)^{Spec(\mathbb{Z}[1/p])}
    $
    \end{adjustbox}
\]
which in turn produces a long exact sequence in cohomology:
\[
    \begin{adjustbox}{max width=\textwidth, center}
    $\displaystyle
        \cdots \to H^{*-2,*-1}(Spec(\mathbb{F}_p), M\mathbb{Z}/p) 
        \to H^{*,*}(Spec(\mathbb{Z}), M\mathbb{Z}/p)
        \to H^{*,*}(Spec(\mathbb{Z}[1/p]), M\mathbb{Z}/p)
        \to \cdots
    $
    \end{adjustbox}
\]
We can describe $H^{*,*}(Spec(\mathbb{F}_p), M\mathbb{Z}/p)$. In fact, by results of Geisser-Levine \cite{GeiLev2000}, $H^{s,t}(Spec(\mathbb{F}_p), M\mathbb{Z}/p)$ is non trivial only on the $s=2t$-line, where it coincides with the modulo-$p$ Milnor $K$-theory of the field with $p$ elements. From the definition, one sees that $K^M_*(\mathbb{F}_p)/p$ is non-trivial only in degree $0$, hence $H^{*,*}(Spec(\mathbb{F}_p), M\mathbb{Z}/p)$ is concentrated in bi-degree $(0,0)$.
We have then isomorphism:
\[
     H^{s,t}(Spec(\mathbb{Z}), M\mathbb{Z}/p)
    \xrightarrow{\sim} H^{s,t}(Spec(\mathbb{Z}[1/p]), M\mathbb{Z}/p)
\]
for all integers $(s,t)$ but $(2,1)$ and $(1,1)$. But then, when considering the connective covers, we find isomorphism for the whole cohomology rings:
\[
    H^{*,*}(Spec(\mathbb{Z}), c_0(M\mathbb{Z}/p)) \cong H^{*,*}(Spec(\mathbb{Z}[1/p]), c_0(M\mathbb{Z}/p)).
\]
Hence everything we deduced above for Dedekind domains with $1/p$ also applies to $\mathbb{Z}$ as well. We conclude that there is, for each $n \in \mathbb{N}$, a map:
\[
    M\mathbb{Z}/p \to (M \mathbb{Z}/p)/\tau^n
\]
of $E_{\infty}$ motivic ring spectra over $Spec(\mathbb{Z})$.
\vspace{1em}

Whatever the base scheme, we find the corresponding quotient of dual motivic Steenrod algebra via a pushout square of $E_{\infty}$ ring spectra:
\begin{equation*}
    \begin{tikzcd}
        M\mathbb{Z}/p
            \arrow[r]
            \arrow[d]
            \arrow[dr, phantom, "\ulcorner", very near end]
        & \mathcal{A}(p)
            \arrow[d]
        \\ (M\mathbb{Z}/p)/\tau^n
            \arrow[r]
        & \mathcal{A}(p)/\tau^n
    \end{tikzcd}
\end{equation*}
Hence $\mathcal{A}(p)/\tau^n$ comes by definition with the arrows to and from $(M\mathbb{Z}/p)/\tau^n$ that are needed in \cite[Proposition 1.22.]{Mocchetti2024b}.
\vspace{1em}

For the rest of the paper, we will work in $SH(Spec(k))$ with $k$ algebraically closed of characteristic different from $p$. $\mathcal{A}(p)/\tau^{n}$ is connective in our $t$-structure on $Mod_{(M\mathbb{Z}/p)/\tau^n}$ for any choice of $n$ because $\mathcal{A}(p)$ is connective in the corresponding $t$-structure on $Mod_{(M\mathbb{Z}/p)}$, as seen in \cite{Mocchetti2024b}. \cite[Proposition 1.22.]{Mocchetti2024b} applies and we get a spectral sequence as in \cite[Equation 1.23.]{Mocchetti2024b}, with $Q= (M\mathbb{Z}/p)/\tau^n$ and $R=\mathcal{A}(p)/\tau^{n}$.

Further specializing to $n=p-1$, we can improve the description of the $E^2$ and the $E^{\infty}$ pages thanks to the nice homotopy rings of the spectra involved, as in \cite{Mocchetti2024b}. The structure of the dual motivic Steenrod algebra at the prime $p$ can be recovered from \cite{HKO2013}; quotienting by $\tau^{p-1}$ we get:
\begin{gather*}
    \pi_{*,*}((M\mathbb{Z}/p)/\tau^{p-1}) \cong\mathbb{F}_p[\tau]/\tau^{p-1}\\
    \pi_{*,*}(\mathcal{A}(p)/\tau^{p-1}) \cong \begin{cases}
        \frac{\mathbb{F}_2[\tau,\, \xi_{i+1},\, \tau_i]_{i \geq 0}}{\langle \tau,\, \tau_i^2-\tau \xi_{i+1} \rangle} \cong \frac{\mathbb{F}_2[\tau,\, \xi_{i+1},\, \tau_i]_{i \geq 0}}{\langle \tau,\, \tau_i^2 \rangle}  & \text{ for } p=2\\
        \frac{\mathbb{F}_p[\tau,\, \xi_{i+1},\, \tau_i]_{i \geq 0}}{\langle \tau^{p-1},\, \tau_i^2 \rangle} & \text{ for } p \text{ odd.}
    \end{cases}
\end{gather*}
In particular, as $\pi_{*,*}((M\mathbb{Z}/p)/\tau^n)$ is an algebra over the field with $p$ elements and both $\pi_{i,*}((M\mathbb{Z}/p)/\tau^n)$ and $\pi_{i,*}(\mathcal{A}(p)/\tau^{n})$ vanish for $i < 0$, we can apply \cite[Lemma 1.25.]{Mocchetti2024b} to deduce the isomorphism on the convergence term of the spectral sequence:
\[
    \pi_{*,*}(\lim_{\overleftarrow{n}} \left((M\mathbb{Z}/p)/\tau^{p-1} \wedge_{\mathcal{A}(p)/\tau^{p-1}} (\mathcal{A}(p)/\tau^{p-1})_{\leq n}\right)) \cong \pi_{*,*}((M\mathbb{Z}/p)/\tau^{p-1}).
\]

It follows from the presentation of  $\pi_{*,*}(\mathcal{A}(p)/\tau^{p-1})$ that the shifted fibres
\[
    \Sigma^{-t,0}fib((\mathcal{A}(p)/\tau^{p-1})_{\leq n} \to (\mathcal{A}(p)/\tau^{p-1})_{\leq n-1})
\]
are flat over $(M\mathbb{Z}/p)/\tau^{p-1}$ (actually, they are a coproduct of shifted copies of $(M\mathbb{Z}/p)/\tau^{p-1}$, as $\mathcal{A}(p)/\tau^{p-1}$ is). Hence, \cite[Lemma 1.27.]{Mocchetti2024b} applies and the $E^2$ page can be expressed as:
\[
    E^2_{s,t,*}=\pi_{s,*}(MHH(M\mathbb{Z}/p)/\tau^{p-1}) \otimes_{\pi_{0,*}((M\mathbb{Z}/p)/\tau^{p-1})} \pi_{t,*}(\mathcal{A}(p)/\tau^{p-1})
\]
where
\[
    MHH(M\mathbb{Z}/p)/\tau^{p-1}= (M\mathbb{Z}/p)/\tau^{p-1} \wedge_{\mathcal{A}(p)/\tau^{p-1}}  (M\mathbb{Z}/p)/\tau^{p-1}.
\]

Combining all together:
\begin{prop}\label{prop:spectral.seq.mod.tau}
Let $k$ be an algebraically closed field, and let $p$ be a prime number, $p\neq char(k)$. Given $(M\mathbb{Z}/p)/\tau^{p-1}$ and $\mathcal{A}(p)/\tau^{p-1}$ defined as above as $E_{\infty}$ algebras in $SH(Spec(k))$, let
\[
    MHH(M\mathbb{Z}/p)/\tau^{p-1}= (M\mathbb{Z}/p)/\tau^{p-1} \wedge_{\mathcal{A}(p)/\tau^{p-1}}  (M\mathbb{Z}/p)/\tau^{p-1}.
\]
Then we have a first-quadrant multiplicative spectral sequence:
\[
\begin{adjustbox}{max width=\textwidth, center}
$\displaystyle
   E^2_{s,t,*}=\pi_{s,*}(MHH(M\mathbb{Z}/p)/\tau^{p-1}) \otimes_{\pi_{0,*}((M\mathbb{Z}/p)/\tau^{p-1})} \pi_{t,*}(\mathcal{A}(p)/\tau^{p-1}) \Rightarrow \pi_{s+t,*}((M\mathbb{Z}/p)/\tau^{p-1})
$    
\end{adjustbox}
\]
with differentials:
\[
    d^k=d^k_{s,t,*}:E^k_{s,t,*} \to E^k_{s-k,t+k-1,*}.
\]
\end{prop}
\vspace{1em}

We will analyse first what happens for $p=2$, then we will look at odd primes. The two arguments presented are essentially identical: only minor differences distinguish the spectral sequences in the two cases. However, when $p$ is even there are two major advantages: an easier ring structure (there are no multiples of $\tau$; divided powers in characteristic $2$ square to zero) and, as the degrees involved are much smaller, it is easier to understand visually the behaviour of the spectral sequence. At the risk of repeating ourselves, we keep each proof independent from the other.

Regardless of the prime number considered, our spectral sequence presents some key features, which we will use more or less implicitly throughout the whole paper:
\begin{itemize}
    \item The $E^{\infty}$ page is concentrated in degrees $(0,0,*)$, hence anything in other degrees must vanish.
    \item The spectral sequence is first quadrant, hence the $d^q$ differential acts trivially on the modules $E^q_{s,t,*}$ with $s<q$. In particular, any differential is trivial on the elements of the zeroth column.
    \item Combining the two above, the modules $E^q_{s,t,*}$ with $s<q$, $t<q-1$ and $(s,t) \neq (0,0)$ must be trivial: any homogeneous component has to vanish before a certain page.
    \item On the other hand, in degree $(0,0,*)$ we have isomorphism:
    \[
        E^2_{0,0,*} \cong E^3_{0,0,*} \cong \ldots \cong E^{\infty}_{0,0,*} \cong \pi_{0,*}((M\mathbb{Z}/p)/\tau^{p-1}) \cong \mathbb{F}_p[\tau]/\tau^{p-1}
    \]
    Observe that $\pi_{0,*}(\mathcal{A}(p)/\tau^{p-1}) \cong \mathbb{F}_p[\tau]/\tau^{p-1}$ is a free $\mathbb{F}_p[\tau]/\tau^{p-1}$-module on a single generator (which we identify with the identity $1$). As:
    \[
        E^2_{0,0,*} \cong \pi_{0,*}(MHH(M\mathbb{Z}/p)/\tau^{p-1}) \otimes_{\mathbb{F}_p[\tau]/\tau^{p-1}} \pi_{0,*}(\mathcal{A}(p)/\tau^{p-1}) 
    \]
    it follows that also $\pi_{0,*}(MHH(M\mathbb{Z}/p)/\tau^{p-1})$ is a free $\mathbb{F}_p[\tau]/\tau^{p-1}$-module on a single generator (again identified with the identity $1$). Because of the tensor product decomposition of the $E^2$-page, we get:
    \[
        E^2_{0,*,*}\cong \pi_{*,*}(\mathcal{A}(p)/\tau^{p-1})
    \]
    \item Finally, the spectral sequence is multiplicative, in the sense that the differentials satisfy a Leibniz rule, which, given our truncation, is weighted on the degree of the classes involved. This means that given two elements $a$ and $b$, we have:
    \[
        d(ab)=d(a)b + (-1)^{deg(a)}a d(b)
    \]
    where $deg(a)$ is the degree of $a$.
\end{itemize}

\begin{rmk}
    The symbols $\upgamma_i (\mu_j)$ will always denote the $i$-th divided power of the element $\mu_j$- we will later clarify what this means in positive characteristic. In particular, throughout the whole document, $\upgamma_i$ is never an element. Moreover, we always identify $\upgamma_0(\mu_j)=1$ and $\upgamma_1 (\mu_j)=\mu_j$ for any $\mu_j$. 
\end{rmk}
\clearpage
\relax


\input{diagrams/picture2}
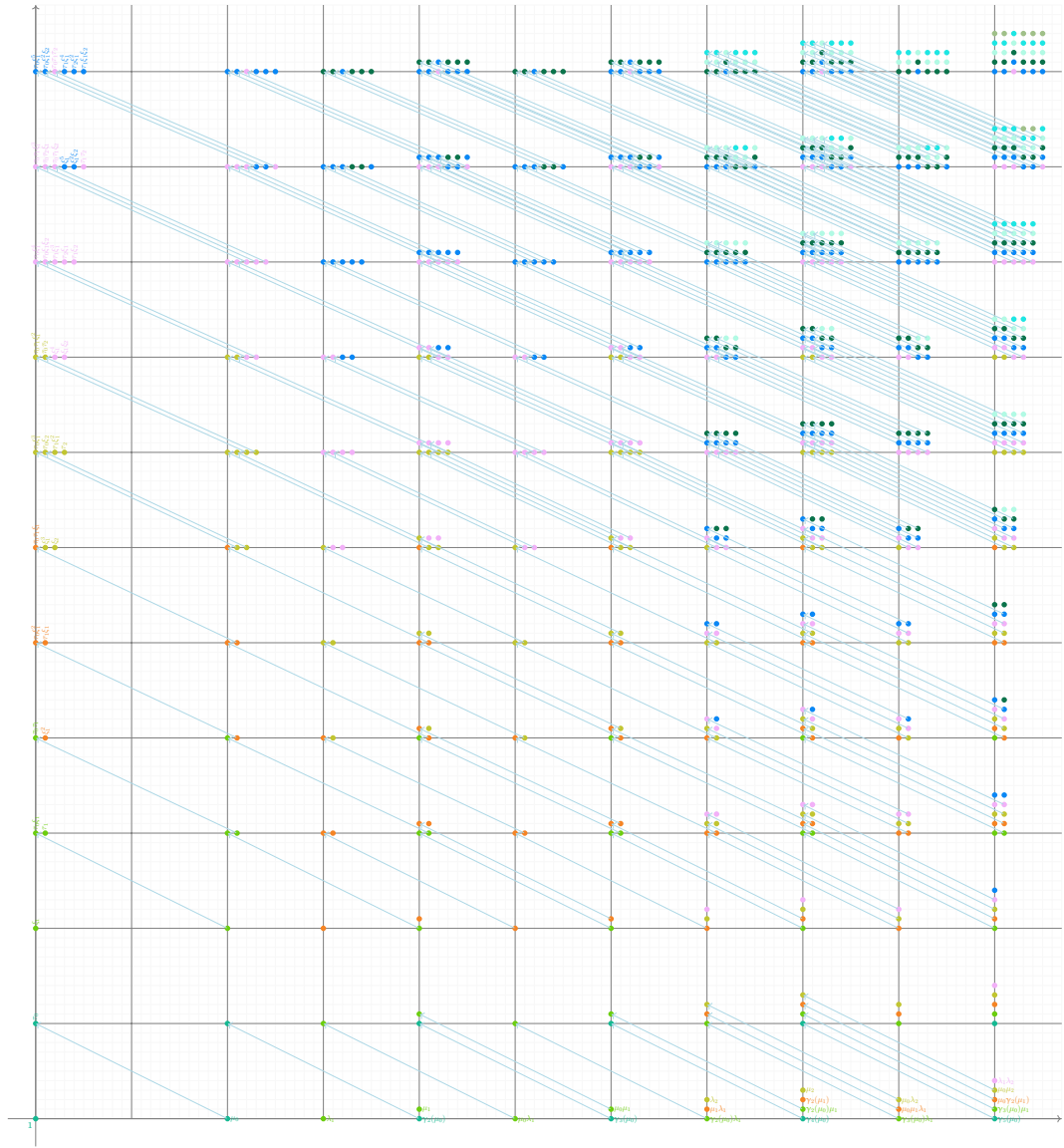
\begin{figure}
\centering
\adjustbox{max width= \textwidth, center}{
\pgfsetshortenend{0.6pt}
\pgfsetshortenstart{0.6pt}
\begin{tikzpicture}[scale=2\textwidth/117cm,line width=1pt]

\draw[step=1cm , almostwhite, very thin] (0,0) grid (107,117);
\draw[step=10cm,gray,very thin] (0,0) grid (105,115);
\draw[gray, line width=0.35pt,->](-3,0)--(105,0);
\draw[gray, line width=0.35pt,->] (0,-3)--(0,115);

{\filldraw[weight0]
(0,0) circle (6pt) node[below left, scale=0.5]{ 1}
;}
{\filldraw[weight1]
(0.0,20) circle (6pt) node[rotate=90, right, scale=0.5]{$\xi_1$}
(1,30) circle (6pt) node[rotate=90, right, scale=0.5]{$\tau_1$}
;}
{\filldraw[weight2]
(1,40) circle (6pt) node[rotate=90, right, scale=0.5]{$\xi_1^2$}
(1,50) circle (6pt) node[rotate=90, right, scale=0.5]{$\tau_1\xi_1$}
;}
{\filldraw[weight3]
(1,60) circle (6pt) node[rotate=90, right, scale=0.5]{$\xi_1^3$}
(2,60) circle (6pt) node[rotate=90, right, scale=0.5]{$\xi_2$}
(2,70) circle (6pt) node[rotate=90, right, scale=0.5]{$\tau_1\xi_1^2$}
(3,70) circle (6pt) node[rotate=90, right, scale=0.5]{$\tau_2$}
;}
{\filldraw[weight4]
(2,80) circle (6pt) node[rotate=90, right, scale=0.5]{$\xi_1^4$}
(3,80) circle (6pt) node[rotate=90, right, scale=0.5]{$\xi_1\xi_2$}
(2,90) circle (6pt) node[rotate=90, right, scale=0.5]{$\tau_1\xi_1^3$}
(3,90) circle (6pt) node[rotate=90, right, scale=0.5]{$\tau_2\xi_1$}
(4,90) circle (6pt) node[rotate=90, right, scale=0.5]{$\tau_1\xi_2$}
(5,100) circle (6pt) node[rotate=90, right, scale=0.5]{$\tau_1\tau_2$}
;}
{\filldraw[weight5]
( 3 , 100 ) circle (6pt) node[rotate=90, right, scale=0.5]{$\xi_1^5$}
( 4 , 100 ) circle (6pt) node[rotate=90, right, scale=0.5]{$\xi_1^2 \xi_2$}
( 3 , 110 ) circle (6pt) node[rotate=90, right, scale=0.5]{$\tau_1 \xi_1^4$}
( 4 , 110 ) circle (6pt) node[rotate=90, right, scale=0.5]{$\tau_2\xi_1^2$}
( 5 , 110 ) circle (6pt) node[rotate=90, right, scale=0.5]{$\tau_1\xi_1\xi_2$}
;}

{\filldraw[weight1]
(30,0) circle (6pt) node[anchor=west, scale=0.5]{$\lambda_1$}
;}
{\filldraw[weight2]
(30,20) circle (6pt)
(31,30) circle (6pt)
;}
{\filldraw[weight3]
(31,40) circle (6pt)
(31,50) circle (6pt)
;}
{\filldraw[weight4]
(31,60) circle (6pt)
(32,60) circle (6pt)
(32,70) circle (6pt)
(33,70) circle (6pt)
;}
{\filldraw[weight5]
(32,80) circle (6pt)
(33,80) circle (6pt)
(32,90) circle (6pt)
(33,90) circle (6pt)
(34,90) circle (6pt)
(35,100) circle (6pt)
;}
{\filldraw[weight6]
(33,100) circle (6pt)
(34,100) circle (6pt)
(33,110) circle (6pt)
(34,110) circle (6pt)
(35,110) circle (6pt)
;}
\draw[colarrow, line width=0.4pt,->](30,0)--(0,20);
\draw[colarrow, line width=0.4pt,->](30,20)--(1,40);
\draw[colarrow, line width=0.4pt,->](31,30)--(1,50);
\draw[colarrow, line width=0.4pt,->](31,40)--(1,60);
\draw[colarrow, line width=0.4pt,->](31,50)--(2,70);
\draw[colarrow, line width=0.4pt,->](31,60)--(2,80);
\draw[colarrow, line width=0.4pt,->](32,60)--(3,80);
\draw[colarrow, line width=0.4pt,->](32,70)--(2,90);
\draw[colarrow, line width=0.4pt,->](33,70)--(3,90);
\draw[colarrow, line width=0.4pt,->](32,80)--(3,100);
\draw[colarrow, line width=0.4pt,->](33,80)--(4,100);
\draw[colarrow, line width=0.4pt,->](32,90)--(3,110);
\draw[colarrow, line width=0.4pt,->](33,90)--(4,110);
\draw[colarrow, line width=0.4pt,->](34,90)--(5,110);

{\filldraw[weight1]
(40,1) circle (6pt) node[anchor=west, scale=0.5]{$\mu_1$}
;}
{\filldraw[weight2]
(40,21) circle (6pt)
(41,31) circle (6pt)
;}
{\filldraw[weight3]
(41,41) circle (6pt)
(41,51) circle (6pt)
;}
{\filldraw[weight4]
(41,61) circle (6pt)
(42,61) circle (6pt)
(42,71) circle (6pt)
(43,71) circle (6pt)
;}
{\filldraw[weight5]
(42,81) circle (6pt)
(43,81) circle (6pt)
(42,91) circle (6pt)
(43,91) circle (6pt)
(44,91) circle (6pt)
(45,101) circle (6pt)
;}
{\filldraw[weight6]
(43,101) circle (6pt)
(44,101) circle (6pt)
(43,111) circle (6pt)
(44,111) circle (6pt)
(45,111) circle (6pt)
;}

{\filldraw[weight2]
(70,1) circle (6pt) node[anchor=west, scale=0.5]{$\mu_1 \lambda1$}
;}
{\filldraw[weight3]
(70,21) circle (6pt)
(71,31) circle (6pt)
;}
{\filldraw[weight4]
(71,41) circle (6pt)
(71,51) circle (6pt)
;}
{\filldraw[weight5]
(71,61) circle (6pt)
(72,61) circle (6pt)
(72,71) circle (6pt)
(73,71) circle (6pt)
;}
{\filldraw[weight6]
(72,81) circle (6pt)
(73,81) circle (6pt)
(72,91) circle (6pt)
(73,91) circle (6pt)
(74,91) circle (6pt)
(75,101) circle (6pt)
;}
{\filldraw[weight7]
(73,101) circle (6pt)
(74,101) circle (6pt)
(73,111) circle (6pt)
(74,111) circle (6pt)
(75,111) circle (6pt)
;}
\draw[colarrow, line width=0.4pt,->](70,1)--(40,21);
\draw[colarrow, line width=0.4pt,->](70,21)--(41,41);
\draw[colarrow, line width=0.4pt,->](71,31)--(41,51);
\draw[colarrow, line width=0.4pt,->](71,41)--(41,61);
\draw[colarrow, line width=0.4pt,->](71,51)--(42,71);
\draw[colarrow, line width=0.4pt,->](71,61)--(42,81);
\draw[colarrow, line width=0.4pt,->](72,61)--(43,81);
\draw[colarrow, line width=0.4pt,->](72,71)--(42,91);
\draw[colarrow, line width=0.4pt,->](73,71)--(43,91);
\draw[colarrow, line width=0.4pt,->](72,81)--(43,101);
\draw[colarrow, line width=0.4pt,->](73,81)--(44,101);
\draw[colarrow, line width=0.4pt,->](72,91)--(43,111);
\draw[colarrow, line width=0.4pt,->](73,91)--(44,111);
\draw[colarrow, line width=0.4pt,->](74,91)--(45,111);

{\filldraw[weight3]
(70,2) circle (6pt) node[anchor=west, scale=0.5]{$\lambda_2$}
;}
{\filldraw[weight4]
(70,22) circle (6pt)
(71,32) circle (6pt)
;}
{\filldraw[weight5]
(71,42) circle (6pt)
(71,52) circle (6pt)
;}
{\filldraw[weight6]
(71,62) circle (6pt)
(72,62) circle (6pt)
(72,72) circle (6pt)
(73,72) circle (6pt)
;}
{\filldraw[weight7]
(72,82) circle (6pt)
(73,82) circle (6pt)
(72,92) circle (6pt)
(73,92) circle (6pt)
(74,92) circle (6pt)
(75,102) circle (6pt)
;}
{\filldraw[weight8]
(73,102) circle (6pt)
(74,102) circle (6pt)
(73,112) circle (6pt)
(74,112) circle (6pt)
(75,112) circle (6pt)
;}

{\filldraw[weight2]
(80,2) circle (6pt) node[anchor=west, scale=0.5]{$\upgamma_2 (\mu_1)$}
;}
{\filldraw[weight3]
(80,22) circle (6pt)
(81,32) circle (6pt)
;}
{\filldraw[weight4]
(81,42) circle (6pt)
(81,52) circle (6pt)
;}
{\filldraw[weight5]
(81,62) circle (6pt)
(82,62) circle (6pt)
(82,72) circle (6pt)
(83,72) circle (6pt)
;}
{\filldraw[weight6]
(82,82) circle (6pt)
(83,82) circle (6pt)
(82,92) circle (6pt)
(83,92) circle (6pt)
(84,92) circle (6pt)
(85,102) circle (6pt)
;}
{\filldraw[weight7]
(83,102) circle (6pt)
(84,102) circle (6pt)
(83,112) circle (6pt)
(84,112) circle (6pt)
(85,112) circle (6pt)
;}

{\filldraw[weight3]
(80,3) circle (6pt) node[anchor=west, scale=0.5]{$\mu_2$}
;}
{\filldraw[weight4]
(80,23) circle (6pt)
(81,33) circle (6pt)
;}
{\filldraw[weight5]
(81,43) circle (6pt)
(81,53) circle (6pt)
;}
{\filldraw[weight6]
(81,63) circle (6pt)
(82,63) circle (6pt)
(82,73) circle (6pt)
(83,73) circle (6pt)
;}
{\filldraw[weight7]
(82,83) circle (6pt)
(83,83) circle (6pt)
(82,93) circle (6pt)
(83,93) circle (6pt)
(84,93) circle (6pt)
(85,103) circle (6pt)
;}
{\filldraw[weight8]
(83,103) circle (6pt)
(84,103) circle (6pt)
(83,113) circle (6pt)
(84,113) circle (6pt)
(85,113) circle (6pt)
;}
{\filldraw[weight4]
(100,4) circle (6pt) node[anchor=west, scale=0.5]{$\lambda_1 \lambda_2$}
;}
{\filldraw[weight5]
(100,24) circle (6pt)
(101,34) circle (6pt)
;}
{\filldraw[weight6]
(101,44) circle (6pt)
(101,54) circle (6pt)
;}
{\filldraw[weight7]
(101,64) circle (6pt)
(102,64) circle (6pt)
(102,74) circle (6pt)
(103,74) circle (6pt)
;}
{\filldraw[weight8]
(102,84) circle (6pt)
(103,84) circle (6pt)
(102,94) circle (6pt)
(103,94) circle (6pt)
(104,94) circle (6pt)
(105,104) circle (6pt)
;}
{\filldraw[weight9]
(103,104) circle (6pt)
(104,104) circle (6pt)
(103,114) circle (6pt)
(104,114) circle (6pt)
(105,114) circle (6pt)
;}
\draw[colarrow, line width=0.4pt,->](100,4)--(70,22);
\draw[colarrow, line width=0.4pt,->](100,24)--(71,42);
\draw[colarrow, line width=0.4pt,->](101,34)--(71,52);
\draw[colarrow, line width=0.4pt,->](101,44)--(71,62);
\draw[colarrow, line width=0.4pt,->](101,54)--(72,72);
\draw[colarrow, line width=0.4pt,->](101,64)--(72,82);
\draw[colarrow, line width=0.4pt,->](102,64)--(73,82);
\draw[colarrow, line width=0.4pt,->](102,74)--(72,92);
\draw[colarrow, line width=0.4pt,->](103,74)--(73,92);
\draw[colarrow, line width=0.4pt,->](102,84)--(73,102);
\draw[colarrow, line width=0.4pt,->](103,84)--(74,102);
\draw[colarrow, line width=0.4pt,->](102,94)--(73,112);
\draw[colarrow, line width=0.4pt,->](103,94)--(74,112);
\draw[colarrow, line width=0.4pt,->](104,94)--(75,112);

\end{tikzpicture}}
\caption{The $E^3$ page of the spectral sequence for $p=2$. Some arrows, exceeding the boundaries of the figure, are not displayed.\\
{\color{weight0} weight 0}, 
{\color{weight1} weight 1}, 
{\color{weight2} weight 2}, 
{\color{weight3} weight 3}, 
{\color{weight4} weight 4}, 
{\color{weight5} weight 5}, 
{\color{weight6} weight 6}, 
{\color{weight7} weight 7}, 
{\color{weight8} weight 8},
{\color{weight9} weight 9}, }

\end{figure}
\begin{figure}
\centering
\adjustbox{max width= \textwidth, center}{
\pgfsetshortenend{0.6pt}
\pgfsetshortenstart{0.6pt}
\begin{tikzpicture}[scale=2\textwidth/117cm,line width=1pt]

\draw[step=1cm , almostwhite, very thin] (0,0) grid (107,117);
\draw[step=10cm, gray, very thin] (0,0) grid (105,115);
\draw[gray, line width=0.35pt,->](-3,0)--(105,0);
\draw[gray, line width=0.35pt,->] (0,-3)--(0,115);

{\filldraw[weight0]
(0,0) circle (6pt) node[anchor=north]{1}
;}
{\filldraw[weight1]
(1,30) circle (6pt) node[rotate=90, right, scale=0.5]{$\tau_1$}
;}
{\filldraw[weight3]
(2,60) circle (6pt) node[rotate=90, right, scale=0.5]{$\xi_2$}
(3,70) circle (6pt) node[rotate=90, right, scale=0.5]{$\tau_2$}
;}
{\filldraw[weight4]
(4,90) circle (6pt) node[rotate=90, right, scale=0.5]{$\tau_1\xi_2$}
(5,100) circle (6pt) node[rotate=90, right, scale=0.5]{$\tau_1\tau_2$}
;}

{\filldraw[weight1]
(40,1) circle (6pt) node[anchor=west, scale=0.5]{$\mu_1$}
;}
{\filldraw[weight2]
(41,31) circle (6pt)
;}
{\filldraw[weight4]
(42,61) circle (6pt)
(43,71) circle (6pt)
;}
{\filldraw[weight5]
(44,91) circle (6pt)
(45,101) circle (6pt)
;}
\draw[colarrow, line width=0.4pt,->](40,1)--(1,30);
\draw[colarrow, line width=0.4pt,->](42,61)--(4,90);
\draw[colarrow, line width=0.4pt,->](43,71)--(5,100);

{\filldraw[weight3]
(70,2) circle (6pt) node[anchor=west, scale=0.5]{$\lambda_2$}
;}
{\filldraw[weight4]
(71,32) circle (6pt)
;}
{\filldraw[weight6]
(72,62) circle (6pt)
(73,72) circle (6pt)
;}
{\filldraw[weight7]
(74,92) circle (6pt)
(75,102) circle (6pt)
;}

{\filldraw[weight2]
(80,2) circle (6pt) node[anchor=west, scale=0.5]{$\upgamma_2 (\mu_1)$}
;}
{\filldraw[weight3]
(81,32) circle (6pt)
;}
{\filldraw[weight5]
(82,62) circle (6pt)
(83,72) circle (6pt)
;}
{\filldraw[weight6]
(84,92) circle (6pt)
(85,102) circle (6pt)
;}
\draw[colarrow, line width=0.4pt,->](80,2)--(41,31);
\draw[colarrow, line width=0.4pt,->](82,62)--(44,91);
\draw[colarrow, line width=0.4pt,->](83,72)--(45,101);

{\filldraw[weight3]
(80,3) circle (6pt) node[anchor=west, scale=0.5]{$\mu_2$}
;}
{\filldraw[weight4]
(81,33) circle (6pt)
;}
{\filldraw[weight6]
(82,63) circle (6pt)
(83,73) circle (6pt)
;}
{\filldraw[weight7]
(84,93) circle (6pt)
(85,103) circle (6pt)
;}
\end{tikzpicture}}
\caption{The $E^4$ page of the spectral sequence for $p=2$. Some arrows, exceeding the boundaries of the figure, are not displayed.\\
{\color{weight0} weight 0}, 
{\color{weight1} weight 1}, 
{\color{weight2} weight 2}, 
{\color{weight3} weight 3}, 
{\color{weight4} weight 4}, 
{\color{weight5} weight 5}, 
{\color{weight6} weight 6}, 
{\color{weight7} weight 7}, 
{\color{weight8} weight 8},
{\color{weight9} weight 9}, }

\end{figure}
\relax



\clearpage

\section{\texorpdfstring{$MHH(M\mathbb{Z}/2)/\tau$}{MHH(MZ/2)/tau}}

\subsection{General considerations}

To reconstruct the homotopy groups of $MHH(M\mathbb{Z}/2)/\tau$ we will determine the behaviour of the whole spectral sequence.
We will show that each page decomposes as a tensor product (over $\mathbb{F}_2$) of two $\mathbb{F}_2$-algebras, namely a subalgebra of $ \pi_{*,*}(MHH(M\mathbb{Z}/2)/\tau)$ with a non-canonical splitting (which we will identify with a quotient of $ \pi_{*,*}(MHH(M\mathbb{Z}/2)/\tau)$) and a quotient algebra of the mod $\tau$ dual motivic Steenrod algebra. This happens because at each page $E^k$, we can associate a set of generators of $E^k_{0,*,*}$ with the source of all the non-trivial differentials and a set of generators of $E^k_{*,0,*}$ with the image; this produces a pattern that affects similarly all the tensor product classes so that the next page preserves the decomposition as tensor of a horizontal $\mathbb{F}_2$ algebra times a vertical one. We will moreover see that non-trivial differentials only occur at pages $E^{2^i}$ or $E^{2^{i+1}-1}$, for all integers $i \geq  1$. The spectral sequence will allow us to recover:

\begin{thm}\label{thm:MHH(MZ/2)/tau}
There is an isomorphism of graded rings:
\[
    \pi_{*,*}(MHH(M\mathbb{Z}/2)/\tau) \cong \bigotimes_{i \in \mathbb{N}} \Gamma_{\mathbb{F}_2}(\mu_i) \otimes \Lambda_{\mathbb{F}_2}(\lambda_{i+1})
\]
where:
\begin{gather*}
    \Gamma_{\mathbb{F}_2}(\mu_i) \cong \mathbb{F}_2[\upgamma_{2^j}(\mu_i)]_{j \geq 0}/\langle \upgamma_{2^j}(\mu_i)^2 \rangle\\
    \Lambda_{\mathbb{F}_2}(\lambda_{i}) \cong \mathbb{F}_2[\xi_{i}]/\langle \xi_{i}^2\rangle\\
    |\upgamma_{2^j}(\mu_i)|=(2^{i+1+j},2^j(2^i-1)) \qquad |\lambda_i|=(2^{i+1}-1,2^i-1).
\end{gather*}
The tensor product is taken in $\mathbb{F}_2$ algebras.
    
\end{thm}

A first easy but useful observation is that there are no elements of negative weight in the $E^2$-page, and, in particular, in $ \pi_{*,*}(MHH(M\mathbb{Z}/2)/\tau)$. This can be seen by induction; there are in fact no elements of negative weight in the zeroth column, corresponding to the mod $\tau$ dual motivic Steenrod algebra, of which we have an explicit description. Then, the presence of an element of a certain weight somewhere implies the existence of elements of lower weight in a previous column, as the $d^2$ differential preserves the weight. As the base ring (the field $\mathbb{F}_2$) is homogeneous of weight $0$, assuming the presence of elements of negative weight somewhere leads to a contradiction.

The general argument proceeds by induction on the index of the page; however, before formulating a precise induction hypothesis, we look at what happens in low-indexed pages, to get a hint of the general picture.

\subsection{\texorpdfstring{The $E^2$ page}{The E2 page}}

Recall the algebra description of the dual mod $\tau$ motivic Steenrod algebra:
\[
\pi_{*,*} (\mathcal{A}(2)/\tau) \cong \mathbb{F}_2[\tau_i, \xi_{i+1}]_{i  \geq  0}/\langle \tau_i^2 \rangle
\]
with $|\tau_i|=(2^{i+1}-1, 2^i-1)$, $|\xi_i|=(2^{i+1}-2, 2^i-1)$. In particular, we notice that elements of weight zero can exist only in the zeroth and first rows. This fact will help to make the argument for the $E^2$ page much simpler than that for the other $E^{2^i}$ pages.

We can conclude immediately that the column with index 1 has to be empty, as any element in $E^2_{1,0,*}$ could not support any non-trivial differential and permanent cycles outside degree $(0,0,*)$ are incompatible with the convergence hypothesis. So 
\[
    \pi_{1,*}(MHH(M\mathbb{Z}/2)/\tau) \cong E^2_{1,0,*}\cong 0.
\]
Due to the decomposition as a tensor product of rings of the $E^2$-page, $E^2_{1,*,*} \cong 0$.

Next, notice that:
\[
    E^2_{0,1,*}\cong \pi_{1,*}(\mathcal{A}(2)/\tau )\cong \mathbb{F}_2\{\tau_0\}.
\]
To avoid that $\tau_0$ becomes a permanent cycle, there must be some element $\mu_0 \in E^2_{2,0,0}$ with $d^2(\mu_0)= \tau_0$. 

Now consider $\mu_0^2$. We have that $d^2(\mu_0^2)=0$ by characteristic; moreover, there are no other weight 0 elements in higher rows, so all higher differentials applied to $\mu_0^2$ must vanish as well. So, as there cannot be permanent cycles outside degree $(0,0,*)$, $\mu_0^2=0$. 

On the other hand, also $d^2(\mu_0 \tau_0)=\tau_0^2=0$, and longer differentials would exit the first quadrant; to avoid that $\mu_0 \tau_0$ produces a permanent cycle, we need an element $\upgamma_2 (\mu_0) \in E^2_{4,0,0}$ with $d^2(\upgamma_2 (\mu_0))= \mu_0 \tau_0$. As before, $d^2((\upgamma_2 (\mu_0))^2)=0$ by characteristic, and since there are no elements of weight zero in higher rows, $(\upgamma_2 (\mu_0))^2=0$. 

Next, by simply applying the Leibniz rule, if we call $\upgamma_3 (\mu_0) = \mu_0 \upgamma_2 (\mu_0)$, then:
\[
    d^2(\upgamma_3 (\mu_0))=\upgamma_2 (\mu_0) \tau_0 + \mu_0 \mu_0 \tau_0= \upgamma_2 (\mu_0) \tau_0.
\]

One then proceeds inductively and identifies a subalgebra $\Gamma_{\mathbb{F}_2}(\mu_0) \subseteq \pi_{*,*} (MHH(M\mathbb{Z}/2)/\tau)$ of divided powers on $\mu_0$:
\[
    \Gamma_{\mathbb{F}_2}(\mu_0) \cong \mathbb{F}_2[\upgamma_{2^i}(\mu_0)]_{i  \geq  0}/\langle \upgamma_{2^i}(\mu_0)^2\rangle
\]
with $|\upgamma_{2^i}(\mu_0)| = (2^{i+1},0)$, 
so that, given any $j  \geq  0$, if one considers its binary expansion $j = \sum_{i \geq  0} a_i 2^i$, by defining:
\[
    \upgamma_j (\mu_0)=\prod_{i \geq  0} ( \upgamma_{2^i}(\mu_0))^{a_i}
\]
one gets:
\[
    d^2(\upgamma_j (\mu_0))= \upgamma_{j-1} (\mu_0) \tau_0.
\]
More in detail, this comes from the fact that $\tau_0$ is the sole class of weight 0 in the mod $\tau$ dual motivic Steenrod algebra; if one assumes to know $\Gamma_{\mathbb{F}_2}(\mu_0)$ (and the associated differentials) up to a certain $\upgamma_{j-1} (\mu_0)$, then the only possibility to kill $\upgamma_{j-1} (\mu_0) \tau_0$ compatibly with the weight will be a weight zero class on the zeroth horizontal line. This either comes as an algebraically independent class if $j=2^i$ for some positive integer $i$ or as a product of smaller degree divided powers (see lemma \ref{lmm:leibniz.on.div.pow} for some related discussion in this sense) for other values of $j$.

They also exhaust all classes supporting a non-trivial $d^2$ differential in $ \pi_{*,*}(MHH(M\mathbb{Z}/2)/\tau)$ up to a base change, in the sense that any class $\alpha \in  \pi_{*,*}(MHH(M\mathbb{Z}/2)/\tau)$ can always be written as a finite sum $\alpha= \sum_{i \geq  0} \upgamma_i (\mu_0)\beta_i$, with $d^2(\beta_i)=0$. Suppose this is not the case, by contradiction, and pick the $\alpha$ for which it is not possible with the lowest (first) degree; this must happen away from the origin, as we verified above that this thesis holds for small degrees. Then, as $d^2(\alpha)$ has a strictly smaller first degree than $\alpha$, we can write by assumption
\[
    d^2(\alpha)= \left (\sum_{i  \geq  0} \upgamma_i (\mu_0) \varepsilon_i \right ) \tau_0
\]
with $\varepsilon_i \in  \pi_{*,*}(MHH(M\mathbb{Z}/2)/\tau)$ such that $d^2(\varepsilon_i)=0$. Now consider $\sum_{i  \geq  0} \upgamma_{i+1} (\mu_0) \varepsilon_i$; one has, by the Leibniz rule:
\[
d^2\left(\sum_{i  \geq  0} \upgamma_{i+1} (\mu_0) \varepsilon_i\right)=\sum_{i  \geq  0} \upgamma_{i} \mu_0 \varepsilon_i \tau_0 = d^2(\alpha).
\]
So $\alpha-\sum_{i  \geq  0} \upgamma_{i+1} \mu_0 \varepsilon_i$ is a cycle for the $d^2$ differential, as we wanted to show. 

Observe also that the $d^2$-differentials originating from the zeroth row in the $E^2$ page surject onto the first row: any class there is a sum of elements of the form $\upgamma_i (\mu_0) \varepsilon \tau_0$, with $\varepsilon \in  \pi_{*,*}(MHH(M\mathbb{Z}/2)/\tau)$ not a multiple of $\upgamma_j (\mu_0)$ for any $j  \geq  1$.  As $d^2(\upgamma_{i+1} (\mu_0) \varepsilon)=\upgamma_{i+1} (\mu_0) d^2(\varepsilon)+\upgamma_i (\mu_0) \varepsilon \tau_0=\upgamma_i (\mu_0) \varepsilon \tau_0$ and the differential is a linear map, we get out assertion. This implies in particular that the entire horizontal line $E^3_{*,1,*}$ is null.

Observe that the same behaviour extends to the whole $E^2$-page. Due to the ring structure of the mod $\tau$ dual motivic Steenrod algebra, every line not multiple of $\tau_0$ presents a pattern of exiting $d^2$ differentials similar to the zeroth line; moreover, these will surject onto a corresponding $E^2_{0,*,*}$-module which is multiple of $\tau_0$.

Thus we can conclude that all nontrivial $d^2$-differentials originate from the divided powers $\upgamma_i (\mu_0)$ and have as image the ideal generated by $\tau_0$. More precisely, there are a commutative diagram:
\[
    \begin{tikzcd}
        E^3_{*,0,*} \arrow[r, hook] \arrow[rd, "\cong"] & 
        E^2_{*,0,*}\cong  \pi_{*,*}(MHH(M\mathbb{Z}/2)/\tau) \arrow[d, two heads] \\
        &  \pi_{*,*}(MHH(M\mathbb{Z}/2)/\tau)/\langle \upgamma_i (\mu_0) \rangle_{i  \geq  1}
    \end{tikzcd}    
\]
and an isomorphism:
\[
    E^3_{0,*,*} \cong \pi_{*,*}(\mathcal{A}(2)/\tau)/\langle \tau_0 \rangle
\]
which allow us to identify:
\[
    E^3_{\star,\bullet,*} \cong \pi_{\star,*}(MHH(M\mathbb{Z}/2)/\tau)/\langle \upgamma_i (\mu_0) \rangle_{i  \geq  1} \otimes_{\mathbb{F}_2} \pi_{\bullet,*}(\mathcal{A}(2)/\tau)/\langle \tau_0 \rangle
\]
Since the spectral sequence is first-quadrant, we can conclude that the injection:
\[
    \Gamma_{\mathbb{F}_2}(\mu_0) \hookrightarrow  \pi_{*,*}(MHH(M\mathbb{Z}/2)/\tau)
\]
must be an isomorphism in degrees smaller or equal to 2.
\subsection{\texorpdfstring{The $E^3$ page}{The E3 page}}

We proceed now with the next step: the $E^3$ page and the $d^3$ differential. 

By the conclusion of the previous step:
\[
    E^3_{0,*,*} \cong \pi_{*,*}(\mathcal{A}(2)/\tau)/\langle \tau_0 \rangle \cong \mathbb{F}_2[\tau_i, \xi_i]_{i  \geq  1}/\langle \tau_i^2\rangle
\]
In particular, the row $E^3_{*,1,*}$ is trivial. Moreover, we know that $E^3_{1,*,*}$ and $E^3_{2,*,*}$ are null as well. 

To avoid $\xi_1\in E^3_{0,2,1}$ being a permanent cycle, we must have a class $\lambda_1 \in \pi_{3,1}(MHH(M\mathbb{Z}/2)/\tau)$ surviving to the $E^3$ page and such that $d^3(\lambda_1)=\xi_1$; this differential surjects onto $E^3_{0,2,*}$, so, at the same time, not to produce permanent cycles in position $(3,0,*)$, we conclude that $E^3_{3,0,*}\cong \mathbb{F}_2\{ \lambda_1\}$. 

Observe that this differential behaviour propagates vertically to all classes in $E^3_{3,*,*}$, in the sense that, as they all are of the form $\lambda_1 \alpha$ for some $\alpha \in E^3_{0,*,*}$, we have $d^3(\lambda_1 \alpha)= \xi_1 \alpha$. There are no relations involving $\xi_1$ in $E^3_{0,*,*} \cong \pi_{*,*}(\mathcal{A}(2)/\tau)/\langle \tau_0 \rangle$, so none of this differentials is trivial, for non-trivial $\alpha$. 

We proceed to show that $\lambda_1^2=0$. One has $d^3(\lambda_1^2)=0$ for characteristic reasons. We see also that there is no other class it could hit with a longer differential, as the only other non-zero class with total degree 5, $\xi_1\tau_1 \in E^3_{0,5,2}$, vanishes after the $E^3$ page. So we conclude that $\lambda_1$ squares to zero.

We also have that $\lambda_1$ and its multiples are the source of all $d^3$ differentials, up to some base change. We can proceed as in the previous page:  suppose $\alpha \in E^3_{*,0,*}$ has a non-trivial $d^3$ differential; we wish to write $\alpha= \lambda_1 \beta + \delta$, with $d^3(\beta)=d^3(\delta)=0$ (in particular, $\beta \neq 0$). We have: $d^3(\alpha)= \xi_1 \beta \neq 0$ by hypothesis. Observe that $\beta$ must have trivial $d^3$ differential: if one had $d^3(\beta)=\xi_1 \omega \neq 0$, then $d^3(d^3(\alpha))=\xi_1^2 \omega \neq 0$, given the tensor decomposition of the  $E^3$ page. Consider then the element $\lambda_1 \beta$: one has $d^3(\lambda_1 \beta)= \xi_1 \beta= d^3(\alpha)$, so $\delta= \alpha -\lambda_1 \beta$ is a cycle for the  $d^3$ differential. Again from the tensor decomposition of $E^3$, we can conclude an analogous result for all the horizontal lines of the $E^3$ page. Hence we can identify the sources of all non-trivial $d^3$ differentials with the multiples of $\lambda_1$.

We can at the same time conclude that the whole line $E^3_{*,2,*}$ is wiped out by $d^3$ differentials: an element $\alpha \xi_1$ is either hit by $\lambda_1 \alpha$, if $\alpha$ has a trivial $d^3$ differential, or hits a multiple of $\xi_1^2$, if $d^3(\alpha) \neq 0 $. In fact,  we have just proven that in this second case we can identify $\alpha = \lambda_1 \beta+ \delta$, with $d^3(\beta)=d^3(\delta)=0$ and $\beta \neq 0$, so $d^3(\lambda_1 \beta \xi_1+ \delta \xi_1)=\beta \xi_1^2 \neq 0$. 
This can be applied as well to all the rows starting with a multiple of $\xi_1$, as there are no relations in $\pi_{*,*}(\mathcal{A}(2)/\tau)$ (hence in $E^3_{0,*,*} \cong \pi_{*,*}(\mathcal{A}(2)/\tau)/\langle \tau_0 \rangle$) that involve multiples of $\xi_1$. We conclude that in the homology of the $d^3$ differential the $E^3_{*,*,*}$ ideal generated by $\xi_1$ disappears.

We can conclude a tensor decomposition also for the $E^4$ page, in the following sense. We have a commutative triangle:
\[
    \begin{tikzcd}
        E^4_{*,0,*} \arrow[r, hook] \arrow[rd, "\cong"] & 
        E^3_{*,0,*} \arrow[d, two heads] \\
        & E^3_{*,0,*}/\langle \lambda_1 \rangle \cong  \pi_{*,*}(MHH(M\mathbb{Z}/2)/\tau)/\langle \upgamma_i (\tau_0),\, \lambda_1 \rangle_{i  \geq  1}
    \end{tikzcd}    
\]
and an isomorphism:
\[
    E^4_{0,*,*} \cong E^3_{0,*,*}/\langle \xi_1 \rangle \cong \pi_{*,*}(\mathcal{A}(2)/\tau)/\langle \tau_0, \xi_1 \rangle
\]
which allow us to identify:
\[
    E^4_{\star,\bullet,*} \cong E^4_{\star,0,*} \otimes_{\mathbb{F}_2} E^4_{0,\bullet,*} 
\]
Since the spectral sequence is first-quadrant, we can conclude that the injection:
\[
    \Gamma_{\mathbb{F}_2}(\mu_0) \otimes_{\mathbb{F}_2} \Lambda_{\mathbb{F}_2}(\lambda_1) \hookrightarrow  \pi_{*,*}(MHH(M\mathbb{Z}/2)/\tau)
\]
is an isomorphism in degrees smaller or equal to 3.
\subsection{\texorpdfstring{The $E^4$ page}{The E4 page}}

It is worth studying separately the next step as well. For the $E^4$ page, we will in fact present an argument that refines the one made for the $E^2$ page to a procedure that can be used in the general induction step. As we are about to see, the $d^4$-differentials are generated by elements of different, non-zero weights. This difference with the $d^2$-differentials is essentially what makes the proof of this step more complicated.

After the $E^3$ page, we notice in particular that the rows $E^4_{*,1,*}$ and $E^4_{*,2,*}$ and the columns $E^4_{1,*,*}$, $E^4_{2,*,*}$ and $E^4_{3,*,*}$ are trivial; the module $E^4_{0,3,*}$ is a $\mathbb{F}_2$ vector space on a single generator $\tau_1$ of weight $1$.

As in the $E^2$ page, we see that, to avoid permanent cycles, there must be a single class $\mu_1 \in E^4_{4,0,1}$ with $d^4(\mu_1)=\tau_1$; moreover, this class must square to zero: $d^4((\mu_1)^2)=0$ by characteristic, so if non-trivial, it had to support a longer differential. The weight of $\mu_1^2$ is two and the only non-trivial class with total degree 7 (the same as $d^4(\mu_1^2)$) in $E^4$ is $\tau_2$, which has weight 3, so there cannot be a $d^8$-differential from $\mu_1^2$ to $\tau_1$. Hence, $\mu_1^2=0$.  

Analogously, as $d^4(\mu_1 \tau_1)=\tau_1^2=0$, we need an element $\upgamma_2 (\mu_1) \in E^4_{8,0,2}$, with $d^4(\upgamma_2 (\mu_1))=\mu_1 \tau_1$. Now observe that, as any element in the column $E^4_{4,*,*}$ is a sum of monomials of the form $\mu_1 \tau_1^{\varepsilon} \beta$, with $\varepsilon=0$ or $1$ and $\beta \in E^4_{0,*,*}$ not a multiple of $\tau_1$, the two $d^4$ differentials we have just introduced, together with the Leibniz rule, are enough to make $E^5_{4,*,*} \cong 0$. In fact, those elements with $\varepsilon=0$ give rise to non-trivial differentials:
\[
	\mu_1 \beta \mapsto \tau_1 \beta
\]
while those with $\varepsilon=1$ are hit by a differential:
\[
	\upgamma_2 (\mu_1) \beta \mapsto \mu_1 \tau_1 \beta.
\]

We now produce an induction argument that introduces all the divided power classes $\upgamma_i (\mu_1)$ for $i \in \mathbb{N}$ and shows that they support non-trivial $d^4$ differentials $d^4(\upgamma_i (\mu_1))=\upgamma_{i-1} (\mu_1) \tau_1$. At the same time, we show that 
the divided powers $\upgamma_i (\mu_1)$ 
are all and the only sources of non-trivial $d^4$ differentials. We proceed by exploring higher and higher degrees on the horizontal line.

More precisely, let $f \in \mathbb{N}$, $f  \geq  2$. At the stage $f-1$, we suppose to know that 
for all $1 \leq l \leq f-1$ there exists a non zero class $\upgamma_l (\mu_1) \in E^4_{4l,0,l}$, such that $d^4(\upgamma_l (\mu_1))=\upgamma_{l-1} (\mu_1) \tau_1$ (recall that we identify $\upgamma_0 (\mu_1) =1$). These classes behave as divided powers in characteristic 2. In other words, we can identify the degree less than or equal to the $4(f-1)$ part of the sub-$\mathbb{F}_2$ algebra of $E^4_{*,0,*}$ generated by the $\upgamma_l (\mu_1)$ with the degree less than or equal to $4(f-1)$ part of the $\mathbb{F}_2$ algebra:
\[
	\mathbb{F}_2[\upgamma_{2^k}(\mu_1)]_{k  \geq  0}/\langle \upgamma_{2^k}(\mu_1)^2 \rangle
\]
($\upgamma_{2^k}(\mu_1)$ sitting always in degree $(2^{k+2}, 2^k)$), via the following map: if one writes $l= \sum_{k  \geq  0} \varepsilon_k 2^k$, with $\varepsilon_k \in \{0,1\}$, then $\upgamma_l (\mu_1) \mapsto \prod_k \left ( \upgamma_{2^k} (\mu_1) \right)^{\varepsilon_k}$. In particular, this implies (in $E^4_{*,0,*}$) that $(\upgamma_l (\mu_1))^2=0$ for all $l$ but the largest power of 2 smaller than or equal to $f-1$.  

Moreover we suppose to know that every class $\alpha \in E^4_{k,0,*}$ for $k \leq 4(f-1)$ is a finite sum $\alpha= \sum_{i=0}^{f-1} \upgamma_i (\mu_1)  \beta_i$ with $d^4(\beta_i)=0$. In this sense, the divided powers $\upgamma_l (\mu_1)$ of $\mu_1$ are the sole source of non-trivial $d^4$ differentials. Moreover, any such sum $\sum_{i=0}^{f-1} \upgamma_i (\mu_1)  \beta_i=0$ is null if and only if $\beta_i=0$ for all $i$.

\begin{rmk}\label{rmk:precise.independence.mu1}
	Observe that from this we can deduce inductively the algebraic independence of the various $\upgamma_l (\mu_1)$ from the $d^4$-cycles in $E^4_{*,0,*}$. In fact, we know that almost all $\upgamma_l (\mu_1)$, for $l \leq f-1$, square to zero, 
    and more generally any product of two such $\upgamma_l(\mu_1)$ is either zero or another $\upgamma_{l'} (\mu_1)$. Hence we can deduce that any homogeneous polynomial equation in degree smaller than or equal to $4(f-1)$ can be rewritten as a linear combination of the $\upgamma_l (\mu_1)$:
	\[
		\sum_{l=0}^{f-1} \beta_l \upgamma_l (\mu_1)
	\]
	where $\beta_l \in E^4_{*,0,*}$ such that $d^4(\beta_l)=0$.
	By the last claim, this sum equals zero if and only if all the $\beta_l=0$, in other words, if the relation is trivial. 
	
\end{rmk}

At each step of the induction procedure, we can draw some additional conclusions.

\begin{rmk} \label{rmk:consequence.of.ind.step.mu1}
	Before passing to the proof of the induction argument, observe that the induction hypothesis for $f-1$ implies that the $d^4$ differential is surjective on $E^4_{l,3,*}$ for $l \leq 4(f-1)-4=4(f-2)$: any element there is, by the previous point, a sum of classes of the form $\upgamma_h (\mu_1) \alpha \tau_1$, with $\alpha \in E^4_{*,0,*}$ and $d^4(\alpha)=0$ (and possibly $h=0$); then $d^4( \upgamma_{h+1} (\mu_1) \alpha)= \upgamma_h (\mu_1) \alpha \tau_1$. 
	
	More in general, since we know that the $E^4$ page is isomorphic to the tensor product of the algebras $E^4_{0,*,*}$ and $E^4_{*,0,*}$, we can apply this last argument to all the rows that start in the first column with a multiple of $\tau_1$. Hence, we then deduce that the columns over multiples of $\upgamma_h (\mu_1) \beta$, for some $h$, with the first degree of $\upgamma_h (\mu_1) \beta$ smaller than or equal to $4(f-2)$, disappear from the $E^5$ page. This is proven exactly as for the column $E^4_{4,*,*}$: any element is either the source or the target of a non-trivial $d^4$ differential. This means, symmetrically, that the portion of all rows that originate from a multiple of $\tau_1$, up to the horizontal degree $4(f-2)$, is in the image of the $d^4$ differential.
\end{rmk}

We have already discussed the base steps before, so we want to prove our assertion for the next integer $f$, assuming the thesis for the previous integer $f-1$.

If $f$ is not a power of 2, we immediately get that $\upgamma_f (\mu_1) \neq 0$ and $d^4(\upgamma_f (\mu_1))=\upgamma_{f-1} (\mu_1) \tau_1$: in fact, $\upgamma_f (\mu_1)$ is given by a product of elements we have already introduced and the action of the $d^4$ differential on it is prescribed by the Leibniz rule. 

\begin{lemma}\label{lmm:leibniz.on.div.pow}
	The Leibniz rule implies: $d^4(\upgamma_f (\mu_1))= \upgamma_{f-1} (\mu_1) \tau_1$. 
\end{lemma}
\begin{proof}
	Write the binary expansion of $f$: $f=\sum_{k=1}^{\lfloor log_2 f \rfloor} e_k 2^{k}$, where the $e_k \in \{0,1\}$.
	Recall that $\upgamma_f (\mu_1)$ is given by:
	\[
		\upgamma_f (\mu_1)= \prod_{k=1}^{\lfloor log_2 f \rfloor} \upgamma_{e_k2^{k}}( \mu_1).
	\]
Let $K:=\{ k \in \mathbb{N}: e_k \neq 0\}$.
	
	Apply the Leibniz rule to such a product:
	\begin{align*}
		d^4(\upgamma_f (\mu_1))&=d^4\left( \prod_{k\in K} \upgamma_{2^{k}}(\mu_1) \right)=
		\sum_{k\in K}  d^4 \left(\upgamma_{2^{k}} (\mu_1) \right) \prod_{\substack{u \in K\\u\neq k}} \upgamma_{2^{u}} (\mu_1)\\
		&=\sum_{k\in K}  \upgamma_{2^{k}-1} (\mu_1) \tau_1 \prod_{\substack{u \in K\\u\neq k}} \upgamma_{2^{u}} (\mu_1)
	\end{align*}
	Now $2^h-1= \sum_{e=0}^{h-1} 2^e$, so:
	\[
		d^4(\upgamma_f (\mu_1))= \sum_{k \in K} \left ( \prod_{e=0}^{k-1}\upgamma_{2^e} (\mu_1)  \prod_{\substack{u\in K\\u\neq k}} \upgamma_{2^{u}} (\mu_1) \right ) \tau_1 
	\]
	Let $\bar{k}= min (K)$. Notice now that all the summands with $k \neq \bar{k}$ are null: there the factor $\upgamma_{2^{\bar{k}}}(\mu_1)$ appears in both of the products; but $\upgamma_{2^{\bar{k}}}(\mu_1)^2=0$ by hypothesis. So:
	\[
		d^4(\upgamma_f (\mu_1))= \left (\prod_{e=0}^{\bar{k}-1}\upgamma_{2^e} (\mu_1) \prod_{\substack{u\in K\\u\neq \bar{k}}}\upgamma_{2^{u}} (\mu_1)\right ) \tau_1  = \upgamma_{f-1} (\mu_1) \tau_1.
	\]
\end{proof}

By the induction hypothesis, the element $\upgamma_{f-1} (\mu_1) \tau_1 \in E^4_{4(f-1),3,f}$ is nontrivial; given the above differential, $\upgamma_f (\mu_1) \neq 0$ as well.

Also, we get immediately that $(\upgamma_f (\mu_1))^2=0$, as we know that at least one of the divisors of $\upgamma_f (\mu_1)$ squares to zero.

Then, we must show that any class $\alpha \in E^4_{k,0,*}$ with $4f-3 \leq k \leq 4f$ can be written as $\alpha=\sum_{l=0}^{f} \beta_l \upgamma_l (\mu_1)$ with $d^4(\beta_l)=0$. If $d^4(\alpha)=0$ there is nothing to show. Now, suppose $\alpha$ supports a non-trivial $d^4$ differential; it has to be of the form 
\[
	d^4(\alpha)=\left(\sum_{l=0}^{f-1} \upgamma_l (\mu_1) \beta_l\right ) \tau_1
\]
with $d^4(\beta_l)=0$ for all $l$, as by induction hypothesis all elements in $E^4_{*,0,*}$ with degree smaller than or equal to $4(f-1)$ admit a similar expression, and $E^4_{*,3,*}\cong E^4_{*,0,*}\tau_1$. Consider now the element:
\[
	\tilde{\alpha}=\sum_{l=0}^{f-1} \upgamma_{l+1} (\mu_1) \beta_l
\]
Then $d^4(\tilde{\alpha}) = d^4(\alpha)$, so $\alpha- \tilde{\alpha}$ is a $d^4$ cycle, as required.

To conclude this step of the induction argument, we have to show the algebraic independence in degrees up to $4f$, see remark \ref{rmk:precise.independence.mu1}. Suppose we had a homogeneous equation:
\[
	\sum_{l=0}^f \beta_l \upgamma_l (\mu_1) =0 
\]
with $d^4(\beta_l)=0$; apply the $d^4$-differential:
\[
	0=d^4 \left(\sum_{l=0}^f  \upgamma_l (\mu_1) \beta_l \right) =\sum_{l=1}^f \beta_l \upgamma_{l-1} (\mu_1) \tau_1
\]
As the $E^4$ page decomposes as a tensor product over a field, we must have $\sum_{l=1}^f \beta_l \upgamma_{l-1} (\mu_1)=0$; given that the degree of this homogeneous sum is at most $4f-4$, by induction hypothesis $\beta_l=0$ for all $1\leq l \leq f$. But then $\beta_0=0$ as well, as we wanted to show.

Finally observe that applying remark \ref{rmk:consequence.of.ind.step.mu1} to this step of the induction procedure proves that the $d^4$ differential surjects onto $E^4_{l,3,*}$ for $l \leq 4 f - 4$ (and similarly for the other rows and columns).

The situation gets more complicated whenever $f$ is a power of 2, say $f=2^i$. In this case, the element we will introduce, $\upgamma_{2^i} (\mu_1)$, is algebraically independent of the previous divided powers, hence we actually have to prove its existence, rather than verifying that some differential is well-behaved. 

We begin by noticing that no element which is algebraically dependent of the $\upgamma_j (\mu_1)$, $1 \leq j \leq 2^i-1$ differentiates to $\upgamma_{2^i-1} (\mu_1) \tau_1$. According to our hypothesis, proceeding as in remark \ref{rmk:precise.independence.mu1}, we can express any such element in degree $4(2^i)$ as a sum:
\[
	x=\sum_{j=0}^{2^i-1} \beta_j \upgamma_j (\mu_1)
\]
with $d^4(\beta_j)=0$; observe incidentally that $(\upgamma_{2^{i-1}} (\mu_1))^2$ can appear as a summand of $\beta_0$, as, by the Leibniz rule, $d^4((\upgamma_{2^{i-1}} (\mu_1))^2)=0$. Suppose then by contradiction that $d^4(x)= \upgamma_{2^i-1} (\mu_1) \tau_1$.

We can then determine the $d^4$ image of $x$:
\[
	\upgamma_{2^i-1} (\mu_1) \tau_1=d^4(x)=\sum_{j=1}^{2^i-1} \beta_j \upgamma_{j-1} (\mu_1) \tau_1 
\]
Because of the decomposition as a tensor product, we conclude the identity:
\[
	\sum_{j=1}^{2^i-1} \beta_j \upgamma_{j-1} (\mu_1) = \upgamma_{2^i-1} (\mu_1)
\]
By the induction hypothesis, such an equation is impossible ($\upgamma_{2^i-1} (\mu_1)$ has coefficient 1).

To justify the presence of $\upgamma_{2^i} (\mu_1)$, we must then show that all differentials ($d^4$ and higher) out of $\upgamma_{2^i-1} (\mu_1) \tau_1$ have to be trivial. We will prove this again by contradiction, by carefully studying the weights of possible elements involved in such differentials. We make use of the following:
\begin{lemma} \label{lmm:d/w.for.E4}
	For any homogeneous $x \in E^4_{0,*,*}\cong \pi_{*,*}(\mathcal{A}(2)/\tau)/\langle \tau_0, \xi_1\rangle$, if one writes $|x|=(d,w)$ then $2 \leq \frac{d}{w} \leq 3$.
\end{lemma}
We postpone the proof of this at the end of the induction argument. We first calculate: 
\[
	|\upgamma_{2^i-1} (\mu_1) \tau_1|=(2^{i+2}-4,3,2^i)
\]
One has $d^4(\upgamma_{2^i-1} (\mu_1) \tau_1)= \upgamma_{2^i -2} (\mu_1) \tau_1^2=0$; by contradiction, we suppose that our element has a non-trivial longer differential:
\[
	\upgamma_{2^i-1} (\mu_1) \tau_1 \xrightarrow{d^{k_1}} \alpha_1 \omega_1
\]
with $k_1  \geq  5$. The degrees of the image are:
\begin{gather*}
	|\alpha_1 \omega_1|=(2^{i+2}-4-k_1,2+k_1,2^i) \\
	|\alpha_1|=(2^{i+2}-4-k_1,0,2^i-w_1) \\
	|\omega_1|=(0,2+k_1,w_1).
\end{gather*}

\begin{rmk}\label{rmk:representatives.in.E4}
	Here (and in later passages as well) we improperly represent the image of a (long) $d^{k_1}$ differential as the product of two elements, with $\alpha_1$ on the horizontal axis and $\omega_1$ on the vertical axis; this is wrong in general: first of all, we do not have an a-priori tensor decomposition for pages higher than the $E^4$ (even though we will prove that all pages are equipped with such a structure), and, even under that assumption, the image of a differential might be a sum of elementary tensors. However, any element in the $E^k$ page is given by a subquotient of a homogeneous module living in the $E^4$ page, where we have a decomposition as tensor product, so for any non-zero homogeneous element of the $E^{k_1}$ page it is possible to find an elementary tensor $\alpha_1\omega_1$ in the $E^4$ page with the same degrees, such that $0 \neq \alpha_1\omega_1 \in E^{k_1}$. The reader should then interpret this writing as: the element $\alpha_1\omega_1 \in E^4_{*,*,*}$ has the same degrees as $d^{k_1}(\upgamma_{2^i-1} (\mu_1) \tau_1)$ and the product $\alpha_1\omega_1$ is nontrivial in the $E^{k_1}$ page.
\end{rmk}

One has some constraints on these parameters:
\begin{gather*}
	k_1  \geq  5\\
	w_1  \geq  3 
\end{gather*}
as all elements in $E^4_{0,*,*}$ but $\tau_1$ have weight at least 3. Observe also that, as $\omega_1$ survives at least to the $E^5$ page, we can apply the induction hypothesis to each of its summands, and deduce in particular that they are not divisible by $\tau_0$, $\xi_1$ and $\tau_1$. Hence lemma \ref{lmm:d/w.for.E4} imposes:
\begin{equation}\label{eq:inequalities.omega1.mu1tau1}
    2 w_1 \leq k_1+2 \leq 3 w_1.
\end{equation}

Suppose first that $k_1=2^{i+2}-4$, that is, the differential hits the zeroth column. In this case, we must have $\alpha_1=1$, as no other nonzero element sits in $E^4_{0,0,*}$. Then $|\omega_1|=(0,2^{i+2}-2,2^i)$, so, seen as an element of the motivic dual Steenrod algebra:
\[
	\frac{deg(\omega_1)}{w(\omega_1)}=\frac{0+2^{i+2}-2}{2^i}= 4- \frac{2}{2^i}
\]
We performed a direct computation of the base step of the induction (that is, $i=1$), so we can suppose $i>1$, in which case $4- \frac{2}{2^i}>3$, which leads to a contradiction with lemma \ref{lmm:d/w.for.E4}. 

Suppose now $\alpha_1$ has a positive degree. We observe that $\alpha_1$ must support a non-trivial differential because the convergence term of the spectral sequence is concentrated in degree $(0,0,*)$; also, this differential must be longer than a $d^4$, as
\[
	deg(\alpha_1)= 2^{i+2}-4-k_1 \leq 2^{i+2}-4-5 <  4(2^i-2)
\]
so $\alpha_1$ belongs to the region controlled by remark \ref{rmk:consequence.of.ind.step.mu1}: if $\alpha_1$ supports a non-trivial $d^4$ differential, then the whole column $\alpha_1 E^4_{0,*,*}$ disappears after the $E^4$ page, making in particular the element $\alpha_1 \omega_1$ trivial in the $E^{k_1}$ page, contradicting our assumptions.

So $\alpha_1$ survives after the $E^4$ page, hence we must have:
\[
	\alpha_1 \xrightarrow{d^{k_2}} \alpha_2\omega_2
\]
with $k_2 \geq  7$, as $7=6+1$ and six is the smallest positive degree in $E^4_{0,*,*}$ presenting a non-zero module after 3 (where $\tau_1$ is) in the $E^4$ page. Again, these $\alpha_2$ and $\omega_2$ are just representatives in $E^4$ for the image of such a differential,  see remark \ref{rmk:representatives.in.E4}. If we write:
\begin{gather*}
	|\alpha_2 \omega_2|=(2^{i+2}-4-k_1-k_2,k_2-1,2^i-w_1) \\
	|\alpha_2|=(2^{i+2}-4-k_1-k_2,0,2^i-w_1-w_2) \\
	|\omega_2|=(0,k_2-1,w_2).
\end{gather*}
	we have constraints:
\begin{gather*}
	k_2  \geq  7\\
	w_2  \geq  3
\end{gather*}
as all elements in $E^4_{0,*,*}$ but $\tau_1$ have weight at least 3. Again, as $\omega_2$ survives at least to the $E^5$ page, given its degree, we can apply the induction hypothesis to each of its summands and deduce that they are not divisible by $\tau_0$, $\xi_1$ and $\tau_1$; by lemma \ref{lmm:d/w.for.E4}:
\[
2 w_2 \leq k_2-1 \leq 3 w_2
\]

We repeat this process until we reach a differential $d^{k_n}$ with $d^{k_n}(\alpha_{n-1})=\omega_n \in E^{k_n}_{0,*,*}$. This can be seen in figure \ref{fig:induction.E4}.

\begin{figure}
    \centering
    \adjustbox{max width=\textwidth, center}{\begin{tikzpicture}
        \filldraw[almostwhite] (10.5,0) rectangle (12.5,9.3);

        \draw[red!30] (0,3) -- (27.7,3);
        \draw[red!30] (27.5,0) -- (27.5,3.2);
        
        \draw[red!30] (0,7) -- (27.7,7); 

        \draw[gray!60] (0,8) -- (21.7,8);
        \draw[gray!60] (21.5,0) -- (21.5,8.2);

        \draw[gray!60] (0,7.5) -- (13.2,7.5);
        \draw[gray!60] (13,0) -- (13,7.7);

        \draw[gray!60] (0,8.5) -- (10.2,8.5);
        \draw[gray!60] (10,0) -- (10,8.7);

        \draw[black, ->] (0,0) -- (0,9.5);
        \draw[black] (0,0) -- (10.5,0);
        \draw[black, dashed] (10.5,0) -- (12.5,0);
        \draw[black, ->] (12.5,0) -- (28,0);

        {\filldraw[red]
        (27.5,3) circle (1pt) node[right]{\normalsize $\upgamma_{2^i-1}(\mu_1)\tau_1$}
        (27.5,0) circle (1pt) node[below]{\normalsize $\upgamma_{2^i-1}(\mu_1)$}
        (0,3) circle (1pt) node[left]{\normalsize $\tau_1$}
        (0,7) circle (1pt) node[left]{\normalsize $\xi_2$}
        ;}
        
        {\filldraw[black]
        (21.5,8) circle (1pt) 
        (21.5,0) circle (1pt) node[below]{\normalsize $\alpha_1$}
        (0,8) circle (1pt) node[left]{\normalsize $\omega_1$}
        
        (13,7.5) circle (1pt) 
        (13,0) circle (1pt) node[below]{\normalsize $\alpha_2$}
        (0,7.5) circle (1pt) node[left]{\normalsize $\omega_2$}
        
        (10,8.5) circle (1pt) 
        (10,0) circle (1pt) node[below]{\normalsize $\alpha_{n-1}$}
        (0,8.5) circle (1pt) node[left]{\normalsize $\omega_{n-1}$}
        
        (0,9) circle (1pt) node[left]{\normalsize $\omega_n$}
        ;}

        \draw[black, ->] (27.5,3) -- (21.5,8) node[right=6pt,  midway] {\normalsize  $d^{k_1}$};
        \draw[black, ->] (21.5,0) -- (13,7.5)  node[right=6pt,  midway] {\normalsize  $d^{k_2}$};
        \draw[black] (13,0) -- (12.5,0.4375);
        \draw[black, ->] (10.5,8.1) -- (10,8.5);
        \draw[black, ->] (10,0) -- (0,9)  node[right=6pt,  midway] {\normalsize  $d^{k_n}$};
        
    \end{tikzpicture}}
    \caption{A visualisation of what should happen if $\upgamma_{2^i-1}(\mu_1)\tau_1$ supported a non trivial differential.}
    \label{fig:induction.E4}
\end{figure}
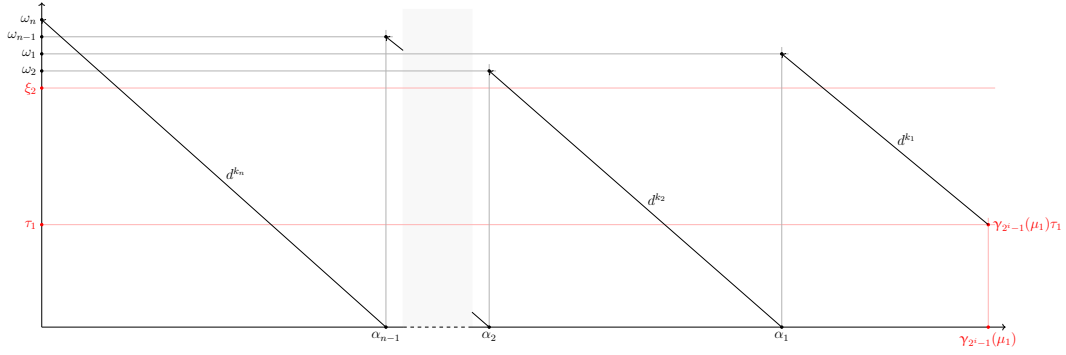

For the differentials $k_3,\ldots,k_n$ we can draw similar conclusions as for the differential $k_2$. In particular, as we have:
\[
	|\alpha_{n-1}|=(2^{i+2}-4-k_1-k_2-\ldots-k_{n-1},0,2^i-w_1-w_2-\ldots-w_{n-1})
\]
if we put:
\[
	|\omega_n|=(0,k_n-1,w_n).
\]
we get:
\begin{equation} \label{eqn:sum.kj.wj}
	\begin{cases}
		2^{i+2}-4-k_1-k_2-\ldots-k_{n-1}=k_n\\
		2^i-w_1-w_2-\ldots-w_{n-1}=w_n.
	\end{cases}
\end{equation}

We analogously get the constraints:
\begin{gather*}
	k_j  \geq  7\\
	w_j  \geq  3\\
	2 w_j \leq k_j-1 \leq 3 w_j
\end{gather*}
for all $2 \leq j \leq n$. Summing on all indices the last row of inequalities (recall that for $j=1$ we must use \ref{eq:inequalities.omega1.mu1tau1}):
\begin{gather*}
	2w_1+2\sum_{j=2}^n w_j  \leq k_1+2 +\sum_{j=2}^n (k_j-1) \leq 3w_1+3\sum_{j=2}^n w_j\\
	2\sum_{j=1}^n w_j -2+n-1\leq \sum_{j=1}^n k_j \leq 3\sum_{j=1}^n w_j -2+n-1
\end{gather*}
Plugging in equation \ref{eqn:sum.kj.wj} provides:
\begin{gather*}
	2\cdot2^i +n-3\leq 2^{i+2}-4 \leq 3\cdot2^i +n-3\\
	2\cdot2^i \leq 2^{i+2}-n-1 \leq 3\cdot2^i
\end{gather*}
So:
\[
\begin{cases}
	2 \cdot 2^i \leq 4 \cdot 2^i -n -1 \\
	4 \cdot 2^i -n -1 \leq 3 \cdot 2^i
\end{cases}
\]
From the second equation, we get $n+1  \geq  2^i$. Now, go back to equation \ref{eqn:sum.kj.wj}; by recalling that $k_1  \geq  5$ and, for $j  \geq  2$, $k_j  \geq  7$, we have:
\[
2^{i+2}=4+k_1+k_2+\ldots+k_n  \geq  4+ 5+ 7(n-1)  \geq  9+7(2^i-2)
\]
which gives:
\[
7\cdot 2^i - 4 \cdot 2^i = 3 \cdot 2^i \leq 5
\]
which is impossible for $i  \geq  1$. Observe that the same estimate can be obtained with the weights (recall that they all value at least 3):
\[
2^i= w_1+ \ldots + w_n  \geq  3 n  \geq  3(2^i -1)
\]
in other words, $2 \cdot 2^i \leq 3$, which again does not hold for all $i  \geq  1$. As we reached a contradiction, we cannot have a non-trivial differential out of $\upgamma_{2^{i}-1}(\mu_1)\tau_1$; but then, by the convergence result for the spectral sequence, we must have a class $\upgamma_{2^i} (\mu_1) \in E^4_{4 \cdot 2^i, 0, 2^i}$ with $d^4(\upgamma_{2^i} (\mu_1)) = \upgamma^{2^i -1} (\mu_1) \tau_1$. Given the description of the previous pages, this must correspond to a new class $\upgamma_{2^i} (\mu_1) \in \pi_{4 \cdot 2^i, 2^i} (MHH(M\mathbb{Z}/2)/\tau)$.

We can now show show that any homogeneous element $\alpha$ in $E^4_{k,0,*}$, with $k \leq 2^{i+2}$, is a finite sum $\alpha= \sum_{i=0}^{f} \upgamma_i (\mu_1)  \beta_i$ with $d^4(\beta_i)=0$. Observe incidentally that we have not yet proven that $(\upgamma_{2^{i-1}} (\mu_1))^2=0$; however, $d^4((\upgamma_{2^{i-1}} (\mu_1))^2)=0$ by characteristic, hence $(\upgamma_{2^{i-1}} (\mu_1))^2$ would in any case end up as a summand of $\beta_0$. If $d^4(\alpha)=0$ we have nothing to prove. In any case, by the previous step of the induction hypothesis and the tensor decomposition of the $E^4$ page, we can express:
\[
    d^4(\alpha)= \left( \sum_{i=0}^{f-1} \upgamma_i (\mu_1)  \beta_i \right ) \tau_1
\]
with $d^4(\beta_i)=0$; consider then the element:
\[
    \tilde{\alpha}=\sum_{i=0}^{f-1} \upgamma_{i+1} (\mu_1)  \beta_i.
\]
By what we have just determined: $d^4(\tilde{\alpha})=d^4(\alpha)$, hence $\tilde{\alpha}-\alpha$ is a $d^4$ cycle, as we wanted to show. Now, if:
\[
    \sum_{i=0}^{f} \upgamma_i (\mu_1)  \beta_i=0
\]
Then:
\[
    0=d^4(\sum_{i=0}^{f} \upgamma_i (\mu_1)  \beta_i)= \left( \sum_{i=1}^{f} \upgamma_{i-1} (\mu_1)  \beta_i \right ) \tau_1
\]
and by the decomposition as a tensor product  of the $E^4$ page:
\[
    \sum_{i=1}^{f} \upgamma_{i-1} (\mu_1)  \beta_i.
\]
By induction hypothesis, $\beta_i=0$ for all $1 \leq i \leq f$. But then also $\beta_0=0$ as we wished.

To conclude this step of the induction argument, we are left to prove that $(\upgamma_{2^{i-1}} (\mu_1))^2=0$.
We already noticed that $d^4((\upgamma_{2^{i-1}} (\mu_1))^2)=0$ because of the characteristic of the base ring. Suppose by contradiction that $(\upgamma_{2^{i-1}} (\mu_1))^2$ were non-zero; then, by the convergence term of the spectral sequence, it had to support some longer differential. We prove that this is impossible by mimicking the argument that excluded any non-trivial differential for $\upgamma^{2^i -1} (\mu_1) \tau_1$. 

Observe that, because of the structure of the $E^4$ page, the non-trivial differential out of $(\upgamma_{2^{i-1}} (\mu_1))^2$ has length at least 7. We adopt similar conventions as in remark \ref{rmk:representatives.in.E4} to deal with the image of such differential. As $(\upgamma_{2^{i-1}} (\mu_1))^2$ has double degree $(2^{i+2},2^i)$, if one supposes:
\[
	(\upgamma_{2^{i-1}} (\mu_1))^2 \xrightarrow{d^{k_1}} \alpha_1 \omega_1
\]
one finds:
\begin{gather*}
	|\alpha_1 \omega_1|=(2^{i+2}-k_1,k_1-1,2^i) \\
	|\alpha_1|=(2^{i+2}-k_1,0,2^i-w_1) \\
	|\omega_1|=(0,k_1-1,w_1)
\end{gather*}
with constraints:
\begin{gather*}
	k_1  \geq  7\\
	w_1  \geq  3 
\end{gather*}
as all elements in $E^4_{0,*,*}$ but $\tau_1$ have weight at least 3; as before, since $\omega_1$ survives at least to the $E^5$-page, given its degree, by the induction hypothesis its summands are not divisible by $\tau_0$, $\xi_1$ and $\tau_1$; then, by lemma \ref{lmm:d/w.for.E4} we also have:
\[
    2 w_1 \leq k_1-1 \leq 3 w_1 
\]

One then studies the differential acting non-trivially on $\alpha_1$ (if $k_1  <  2^{i+2}$: the case $k_1 = 2^{i+2}$ does not require a distinct treatment here), and one finds the same constraints on the degree and weight of the factors in its image.
By remark \ref{rmk:consequence.of.ind.step.mu1} applied to this induction step (we have already proven what is necessary for that to work), $\alpha_1$ must support a differential longer than a $d^4$, otherwise the whole column $\alpha_1 E^4_{0,*,*}$ would vanish from the $E^5$ page onwards. Hence $\alpha_1$ supports at least a $d^7$, producing constraints on a possible image $\alpha_2\omega_2$ analogous to those we have on $\alpha_1 \omega_1$.

We repeat this process $n$ times until we finally reach a differential hitting the vertical line in $E^{k_n}_{0,k_n-1, w_n}$. We have then:
\begin{equation}\label{eq:sum.ki.wi.mu1^2} 
	\begin{cases}
		2^{i+2}=k_1+k_2+\ldots+k_n\\
		2^i=w_1+w_2+\ldots+w_n.
	\end{cases}
\end{equation}
with the usual constraints:
\begin{gather*}
	2 w_j \leq k_j-1 \leq 3 w_j\\
	k_j  \geq  7\\
	w_j  \geq  3 
\end{gather*} 
for all $1 \leq j \leq n$.

Summing on all indices:
\begin{gather*}
	2 \sum_{j=1}^n w_j \leq \sum_{j=1}^n k_j -n \leq 3 \sum_{j=1}^n w_j\\
	2 \cdot 2^i \leq 2^{i+2} -n \leq 3 \cdot 2^i
\end{gather*}
So $n  \geq  2^i$ in this case. But then, using the constraints in \ref{eq:sum.ki.wi.mu1^2}:
\begin{equation*} 
	\begin{cases}
		4 \cdot 2^i= 2^{i+2} =k_1+k_2+\ldots+k_n \geq  7n  \geq  7 \cdot 2^i\\
		2^i=w_1+w_2+\ldots+w_n  \geq  3n  \geq  3 \cdot 2^i
	\end{cases}
\end{equation*}
which both provide the desired contradiction, proving: $(\upgamma_{2^{i-1}} (\mu_1))^2 =0$. This concludes the induction procedure for the $E^4$ page.

\begin{proof}[Proof of Lemma \ref{lmm:d/w.for.E4}]
We recall the standard presentation and the degrees and weights of the algebra generators in the mod-$\tau$ dual motivic Steenrod algebra $\pi_{*,*}(\mathcal{A}(2)/\tau)$:
\begin{gather*}
	\pi_{*,*}(\mathcal{A}(2)/\tau) \cong \mathbb{F}_2[\tau_i,\xi_{i+1}]_{i  \geq  0}/\langle \tau_i^2 \rangle \\
	|\tau_i|=(2^{i+1}-1,2^i-1)\\
	|\xi_i|=(2^{i+1}-2, 2^i-2).
\end{gather*}
Any homogeneous element in $E^4_{0,*,*} \cong \pi_{*,*}(\mathcal{A}(2)/\tau)/\langle \tau_0, \xi_1 \rangle$ is a homogeneous sum of products of the form:
\[
	x=\tau_1^{\varepsilon_1} \tau_2^{\varepsilon_2} \cdots \tau_k^{\varepsilon_k}\xi_2^{l_2} \cdots \xi_k^{l_k}
\]
with $\varepsilon_i \in \{0,1\}$ and $l_i \in \mathbb{N}$. Its bi-degree is:
\begin{align*}
	|x|&=\left (\sum_{i=1}^k \varepsilon_i(2^{i+1}-1)+\sum_{i=2}^k l_j(2^{i+1}-2), \sum_{i=1}^k \varepsilon_i(2^{i}-1)+\sum_{i=2}^k l_j(2^{i}-2) \right )\\
	&=\left (\sum_{i=1}^k \left[(\varepsilon_i+l_i)(2^{i+1}-2)+ \varepsilon_i\right], \sum_{i=1}^k (\varepsilon_i+l_i)(2^{i}-1) \right )
\end{align*}
where we set $l_1=0$. We study the degree/weight ratio (observe that no non-unit element has weight zero at this stage):
\[
	\frac{\sum_{i=1}^k \left[(\varepsilon_i+l_i)(2^{i+1}-2)+ \varepsilon_i\right]}{\sum_{i=1}^k (\varepsilon_i+l_i)(2^{i}-1)}
\]
We want to compare summand by summand; observe that:
\[
	\left [(\varepsilon_i+l_i)(2^{i+1}-2)+ \varepsilon_i\right]=0 \iff \varepsilon_i=l_i=0 \iff (\varepsilon_i+l_i)(2^{i}-1)=0
\]
as $\varepsilon_i$ and $l_i$ are non negative and $2^i-1$ and $2^{i+1}-2$ are strictly positive for all $i  \geq  1$. So we can safely study the ratio for any $i$ such that $\varepsilon_i\neq 0$ or  $l_i \neq 0$. One has:
\[
	2 =
	\frac{(\varepsilon_i+l_i)(2^{i+1}-2)}{ (\varepsilon_i+l_i)(2^{i}-1)} \leq
	\frac{(\varepsilon_i+l_i)(2^{i+1}-2)+ \varepsilon_i}{(\varepsilon_i+l_i)(2^{i}-1)}=
	2+ \frac{\varepsilon_i}{(\varepsilon_i+l_i)(2^{i}-1)}
\]

Now, if $\varepsilon_i=0$:
\[
	\frac{\varepsilon_i}{(\varepsilon_i+l_i)(2^{i}-1)}=0 \leq 1
	\]
If $\varepsilon_i=1$ instead:
\[
\frac{\varepsilon_i}{(\varepsilon_i+l_i)(2^{i}-1)}=\frac{1}{(1+l_i)(2^{i}-1)} \leq 1
\]
as $2^i-1  \geq  1$ for $i  \geq  1$. So we get the general estimate
\[
	2 \leq
	\frac{(\varepsilon_i+l_i)(2^{i+1}-2)+ \varepsilon_i}{(\varepsilon_i+l_i)(2^{i}-1)}
	\leq 2+1=3.
\]
This estimate, being independent of $i$, applies to the ratio of the sums as well, proving our thesis.
\end{proof}

\begin{rmk}
	This estimate cannot be perfected for the $E^4$ page, as there are classes (for instance, $\xi_2$ and $\tau_1$) that have the degree/weight ratio equal to either $2$ or $3$. We will see in lemma \ref{lmm:d/w.for.E.2.j+1} that, while in general the bound $2$ is fixed (each $\xi_i$- and hence their powers and products- has a ratio equal to 2), the upper bound will decrease if we remove from the dual Steenrod algebra classes with low degree (or equivalently, low weight).
\end{rmk}

Knowing all the non-trivial $d^4$ differentials, we can conclude, as in the previous passages, that we have a commutative triangle:
\[
    \begin{tikzcd}
		E^5_{*,0,*} \arrow[r, hook] \arrow[rdd, "\cong"] & 
		E^4_{*,0,*} \arrow[dd, two heads] \\&\\
		& E^4_{*,0,*}/\langle \upgamma_i (\mu_1) \rangle_{i  \geq  1} 
    \end{tikzcd}
\]
\[
    E^5_{*,0,*} \cong  \pi_{*,*}(MHH(M\mathbb{Z}/2)/\tau)/\langle \upgamma_i (\tau_0),\, \lambda_1,\, \upgamma_i (\mu_1) \rangle_{i  \geq  1}
\]
and an isomorphism:
\[
	E^5_{0,*,*} \cong E^4_{0,*,*}/\langle \tau_1 \rangle \cong \pi_{*,*}(\mathcal{A}(2)/\tau)/\langle \tau_0, \xi_1, \tau_1 \rangle
\]
They allow us to identify:
\[
	E^5_{\star,\bullet,*} \cong E^5_{\star,0,*} \otimes_{\mathbb{F}_2} E^5_{0,\bullet,*}
\]
Finally, since the spectral sequence is first-quadrant, we can conclude that the injection:
\[
	\Gamma_{\mathbb{F}_2}(\mu_0) \otimes_{\mathbb{F}_2} \Lambda_{\mathbb{F}_2}(\lambda_1) \otimes_{\mathbb{F}_2} \Gamma_{\mathbb{F}_2}(\mu_1) \hookrightarrow  \pi_{*,*}(MHH(M\mathbb{Z}/2)/\tau)
\]
is an isomorphism in degrees smaller or equal to 4. Observe that in the $E^5$ page all the columns $E^5_{s,*,*}$ for $1 \leq s \leq 4$ and all the rows $E^5_{*,t,*}$ for $1 \leq t \leq 3$ are empty by the argument we carried out. In fact, since the non-trivial generator of $E^5_{0,*,*} \cong \pi_{*,*}(\mathcal{A}(2)/\tau)/\langle \tau_0, \xi_1, \tau_1 \rangle$ with lowest degree is $\xi_2$ in degree $6$, we have that all the rows $E^5_{*,t,*}$ for $1 \leq t \leq 5$ are null.
\subsection{The induction step}
We are now ready to present a general induction argument. In step $k  \geq  2$ we study the behaviour of the spectral sequence at the $E^k$ page, determine the structure of $\pi_{k,*}(MHH(M\mathbb{Z}/2)/\tau)$ and describe the appearance of the $E^{k+1}$ page. 

More precisely, at step $k$, we suppose that the $E^k$ page of the spectral sequence decomposes as a tensor product over $\mathbb{F}_2$:
\[
	E^k_{\star, \bullet, *} \cong E^k_{\star, 0, *} \otimes_{\mathbb{F}_2} E^k_{0, \bullet, *}
\]

If $k=2^{j+1}$, $j  \geq  0$ an integer, we can specify isomorphisms: 
\begin{gather*}
	E^{2^{j+1}}_{*, 0, *} \cong \pi_{*,*}(MHH(M\mathbb{Z}/2)/\tau)/\langle \upgamma_i\mu_0, \ldots, \upgamma_i \mu_{j-1}, \lambda_1, \ldots \lambda_j\rangle_{i  \geq  1} \\
	E^{2^{j+1}}_{0, *, *} \cong \pi_{*,*}(\mathcal{A}(2)/\tau)/\langle \tau_0, \ldots, \tau_{j-1}, \xi_1, \ldots \xi_j\rangle
\end{gather*}
and an injection:
\begin{equation*}
    \left ( \bigotimes_{h=0}^{j-1} \Gamma_{\mathbb{F}_2}(\mu_h) \right ) \otimes_{\mathbb{F}_2} \left ( \bigotimes_{h=1}^{j}  \Lambda_{\mathbb{F}_2}(\lambda_h) \right )  \hookrightarrow \pi_{*,*}(MHH(M\mathbb{Z}/2)/\tau)
\end{equation*}
which is itself an isomorphism up to degree $2^{j+1}-1$.

Otherwise, if $2^{j}$ is the largest power of $2$ with $2^{j} < k < 2^{j+1}$, we have:
\begin{gather*}
	E^{k}_{*, 0, *} \cong \pi_{*,*}(MHH(M\mathbb{Z}/2)/\tau)/\langle \upgamma_i\mu_0, \ldots, \upgamma_i \mu_{j-1}, \lambda_1, \ldots \lambda_{j-1}\rangle_{i  \geq  1} \\
	E^{k}_{0, *, *} \cong \pi_{*,*}(\mathcal{A}(2)/\tau)/\langle \tau_0, \ldots, \tau_{j-1}, \xi_1, \ldots \xi_{j-1}\rangle
\end{gather*}
and an injection:
\begin{equation*}
    \left ( \bigotimes_{h=0}^{j-1} \Gamma_{\mathbb{F}_2}(\mu_h) \right ) \otimes_{\mathbb{F}_2} \left ( \bigotimes_{h=1}^{j-1}  \Lambda_{\mathbb{F}_2}(\lambda_h) \right )  \hookrightarrow \pi_{*,*}(MHH(M\mathbb{Z}/2)/\tau)
\end{equation*}
which is an isomorphism up to degree $k-1$.

\subsubsection{\texorpdfstring{$2^j < k < 2^{j+1}-1$}{2\^{}j < k < 2\^{}(j+1)-1}}

The previous proofs provide a base step for the induction, up to $k = 5$, so we may suppose $k  \geq  5$. Let us then prove that if the above hypothesis holds up to a certain number $k$, then it holds for $k+1$ as well. 

We first notice that if $2^j < k < 2^{j+1}-1$, for some integer $j$, then the proof is immediate, as $E^{k+1} \cong E^k$. By our induction hypothesis, in fact:
\begin{gather*}
	E^k_{\star, \bullet, *} \cong E^k_{\star, 0, *} \otimes_{\mathbb{F}_2} E^k_{0, \bullet, *}\\
	E^{k}_{0, *, *} \cong \pi_{*,*}(\mathcal{A}(2)/\tau)/\langle \tau_0, \ldots, \tau_{j-1}, \xi_1, \ldots \xi_{j-1}\rangle
	\cong \mathbb{F}_2[\tau_i,\xi_i]_{i \geq  j}/\langle\tau_i^2 \rangle
\end{gather*}
Notice that all the $\tau_i$ and $\xi_i$ with $i  \geq  j$ have degree at least $2^{j+1}-2$. From the description of the $E^k$ page as a tensor product, all rows $E^k_{*,t,*}$ are then null for $1 \leq t \leq 2^{j+1}-3$; as $k-1 < 2^{j+1}-2$ there is no space for non-trivial $d^k$ differentials exiting the horizontal line; at the same time, the decomposition as tensor product implies, via the Leibniz rule, that no class in any row of the $E^k$ page can be source for a non-trivial $d^k$ differential. So $E^{k+1} \cong E^k$: this confirms the decomposition as tensor product of the $E^{k+1}$ page and the isomorphisms:
\[\begin{gathered}
    E^{k+1}_{*,0,*} \cong E^k_{*,0,*} \cong \pi_{*,*}(MHH(M\mathbb{Z}/2)/\tau)/ \langle \upgamma_i(\mu_0), \ldots, \upgamma_i (\mu_{j-1}), \lambda_1, \ldots \lambda_{j-1}\rangle_{i  \geq  1}\\
    E^{k+1}_{0, *, *} \cong E^{k}_{0, *, *} \cong \pi_{*,*}(\mathcal{A}(2)/\tau)/\langle \tau_0, \ldots, \tau_{j-1}, \xi_1, \ldots \xi_{j-1}\rangle
\end{gathered}\]

Focusing on the zeroth horizontal line,  
the inclusion morphism we have by the previous induction step:
\[
    \begin{adjustbox}{max width=\textwidth, center}
        $\displaystyle 
        \left ( \bigotimes_{h=0}^{j-1} \Gamma_{\mathbb{F}_2}(\mu_h) \right ) \otimes_{\mathbb{F}_2} \left ( \bigotimes_{h=1}^{j-1}  \Lambda_{\mathbb{F}_2}(\lambda_h) \right )  \hookrightarrow \pi_{*,*}(MHH(M\mathbb{Z}/2)/\tau)
        $
    \end{adjustbox}
\]
must be an isomorphism in degree $k$ as well, as the image of the module $E^{k}_{k,0,*}$ cannot be altered by our first quadrant spectral sequence in pages $E^s$ with $s > k$ and has to vanish at the $E^{\infty}$ page. In addition, this implies that the whole column $E^{k+1}_{k,*,*} \cong 0 $ is trivial. 
\subsubsection{\texorpdfstring{$k=2^{j+1}-1$}{k=2\^{}{j+1}-1}}

Suppose now that $k=2^{j+1}-1$. The goal of this step is to introduce a generator $\lambda_j \in \pi_{2^{j+1}-1,2^{j}-1}(MHH(M\mathbb{Z}/2)/\tau)$, which survives up to the $E^{2^{j+1}-1}$ page and supports a differential $d^{2^{j+1}-1}(\lambda_j)=\xi_j$. As we shall see, this argument closely resembles that for $\lambda_1$.

First of all, we notice that the non-trivial elements in 
\[
    E^{2^{j+1}-1}_{0,*,*} 
    \cong \pi_{*,*}(\mathcal{A}(2)/\tau)/\langle \tau_0, \ldots, \tau_{j-1}, \xi_1, \ldots \xi_{j-1}\rangle
    \cong \mathbb{F}_2[\tau_i,\xi_i]_{i \geq  j}/\langle\tau_i^2 \rangle
\]
with the smallest degrees are:
\begin{gather*}
    \xi_j \text{ with } |\xi_j|=(2^{j+1}-2,2^{j}-1)\\
    \tau_j \text{ with } |\tau_j|=(2^{j+1}-1,2^{j}-1)\\
    \xi_j^2 \text{ with } |\xi_j^2|=(2^{j+2}-4,2^{j+1}-2)\\
    \xi_j\tau_j \text{ with } |\xi_j\tau_j|=(2^{j+2}-3,2^{j+1}-2)\\
    \xi_{j+1} \text{ with } |\xi_{j+1}|=(2^{j+2}-2,2^{j+1}-1).
\end{gather*}

There must then be a unique class $\lambda_j \in E^{2^{j+1}-1}_{2^{j+1}-1,0,2^j-1}$ with $d^{2^{j+1}-1}(\lambda_j) = \xi_j$, since $\xi_j$ must vanish on the $E_{\infty}$ page, and we cannot produce permanent cycles. 
Then every element in $E^{2^{j+1}-1}_{2^{j+1}-1,*,*}$ inherits a non-trivial $d^{2^{j+1}-1}$-differential, given by the Leibniz rule:
\[
d^{2^{j+1}-1}(\lambda_j \alpha)= \xi_j \alpha, 
\]
for any $\alpha \in E^{2^{j+1}-1}_{0,*,*} \cong \pi_{*,*}(\mathcal{A}(2)/\tau)/\langle \tau_0, \ldots, \tau_{j-1}, \xi_1, \ldots \xi_{j-1}\rangle$.

This differential is moreover non-trivial because, by the description of the $E^{2^{j+1}-1}$-page coming from the induction hypothesis, $\xi_j$ behaves as a polynomial variable in this page. 

One then sees that $\lambda_j^2=0$: in fact, $d^{2^{j+1}-1}(\lambda_j^2)=0$ because the base field has characteristic 2. By induction, the columns $E^{2^{j+1}-1}_{i,*,*}$ are trivial for $1 \leq i \leq 2^{j+1}-2$, so the only possible other non-trivial differential  for $\lambda_j^2$ is a transgressive one, that is, a $d^{2^{j+2}-2}$ differential. Now, the only non-trivial class with total degree $2^{j+2}-3$ on the zeroth column of the $E^{2^{j+1}-1}$-page is $\xi_{j}\tau_j$. But we already have a non-trivial $d^{2^{j+1}-1}$ differential:
\[
    \lambda_j\tau_j\mapsto \xi_j \tau_j
\]
so $\xi_j \tau_j$ is null after the $E^{2^{j+1}-1}$ page. Hence, $\lambda_j^2=0$ by convergence reasons.

We can now show that any element $\alpha \in E^{2^{j+1}-1}_{*,0,*}$ is of the form $\alpha=\beta \lambda_j +\delta$, where $\beta$ and $\delta$ are $d^{2^{j+1}-1}$ cycles, so that $d^{2^{i+j}-1}(\alpha)=\beta\xi_j$, using more or less the same argument as in the $E^3$ page. This means that we can identify the source of all the non-trivial $d^{2^{j+1}-1}$-differentials leaving the zeroth row with the multiples of $\lambda_j$; observe incidentally that, because of the tensor product decomposition of the $E^{2^{j+1}-1}$-page, any $\beta \lambda_j +\delta$ as above has, for $\beta \neq 0$, a non-trivial $d^{2^{j+1}-1}$-differential. To prove our claim, observe that if $d^{2^{j+1}-1}(\alpha)=0$ we are done; on the other hand, let $d^{2^{j+1}-1}(\alpha)=\beta\xi_j \neq 0$; one must then have $d^{2^{j+1}-1}(\beta)=0$, otherwise, if $d^{2^{j+1}-1}(\beta)=\omega \xi_{j} \neq 0$:
\[
    0=d^{2^{j+1}-1}(d^{2^{j+1}-1}(\alpha)))=d^{2^{j+1}-1}(\beta\xi_j)=\omega \xi_{j}^2 \neq 0
\]
because the tensor product decomposition of the $E^{2^{j+1}-1}$-page does not allow relations between $\omega \in E^{2^{j+1}-1}_{*,0,*}$ and $\xi_{j}^2 \in E^{2^{j+1}-1}_{0,*,*}$. Then: $d^{2^{j+1}-1}(\beta \lambda_j)=\beta\xi_j= d^{2^{j+1}-1}(\alpha)$, so $\delta = \alpha -\beta \lambda_j$ is the desired cycle. Hence, up to base change, we can identify all and only the multiples of $\lambda_j$ in $E^{2^{j+1}-1}_{*,0,*}$ with the sources of non-trivial $d^{2^{j+1}-1}$-differentials leaving the zeroth row. 

This behaviour propagates to the upper rows, again due to the tensor product decomposition of the page. In particular, if we identify rows with the element in $E^{2^{j+1}-1}_{0,*,*}$ that belongs to them, then those that are not multiple of $\xi_j$ will present the same behaviour of the zeroth row: a non-trivial $d^{2^{j+1}-1}$-differential out of all multiples of $\lambda_j$, and no non-trivial differential involving the other elements. Those rows that begin with a multiple of $\xi_j$ will instead present the same behaviour as the row beginning with $\xi_j$ itself: a non-trivial $d^{2^{j+1}-1}$ arising from each multiple of $\lambda_j$, a non-trivial $d^{2^{j+1}-1} $ hitting any other element. In particular, any multiple of $\xi_j$ disappears from the next page. 
No other element is involved in non-trivial $d^{2^{j+1}-1} $ differentials; otherwise, a similar behaviour would be found on the zeroth line $E^{2^{j+1}-1}_{*,0,*}$ as well. 

We can also exclude relations involving $\lambda_j$ and other elements of $E^{2^{j+1}-1}_{*,0,*}$ (and, since our induction hypothesis relates $ \pi_{*,*}(MHH(M\mathbb{Z}/2)/\tau)$ and $E^{2^{j+1}-1}_{*,0,*}$,  involving $\lambda_j$ and other elements of $ \pi_{*,*}(MHH(M\mathbb{Z}/2)/\tau)$), as follows. Suppose by contradiction that there is such a relation; as $\lambda_j^2=0$, this reduces to something of the form $\beta \lambda_j+\delta=0$ with $\beta,\,\delta \in E^{2^{j+1}-1}_{*,0,*}$. Since there is an identification between multiples of $\lambda_j$ and elements that have non-trivial $d^{2^{j+1}-1}$ differentials, it can be reduced to having $d^{2^{j+1}-1}(\beta)=d^{2^{j+1}-1}(\delta)=0$. Hence:
\[
    0=d^{2^{j+1}-1}(\beta \lambda_j+\delta)=\beta \xi_j.
\]
Again due to the tensor product decomposition of the page, one gets that $\beta=0$; hence also $\gamma=0$ as desired.

Putting everything together, we conclude that there are a commutative triangle:
\[
 \begin{tikzcd}
        E^{2^{j+1}}_{*,0,*} \arrow[r, hook] \arrow[rd, "\cong"] & 
        E^{2^{j+1}-1}_{*,0,*} \arrow[d, two heads] \\
        & E^{2^{j+1}-1}_{*,0,*}/\langle \lambda_j \rangle
    \end{tikzcd}    
\]
and an isomorphism:
\[
    E^{2^{j+1}}_{0,*,*}\cong E^{2^{j+1}-1}_{0,*,*}/ \langle \xi_j \rangle,
\]
and that:
\begin{equation}\label{eqn:concl.step.2(j+1)-1.a}
    E^{2^{j+1}}_{\star,\bullet ,*} \cong E^{2^{j+1}}_{\star,0,*} \otimes_{\mathbb{F}_2} E^{2^{j+1}}_{0, \bullet,*}.
\end{equation}

Combined with the previous induction steps, this gives the isomorphisms:
\begin{gather}
    E^{2^{j+1}}_{*, 0, *} \cong  \pi_{*,*}(MHH(M\mathbb{Z}/2)/\tau)/\langle \upgamma_i(\mu_0), \ldots, \upgamma_i (\mu_{j-1}), \lambda_1, \ldots \lambda_j\rangle_{i  \geq  1} \label{eqn:concl.step.2^j+1-1.b}\\
    E^{2^{j+1}}_{0, *, *} \cong \pi_{*,*}(\mathcal{A}(2)/\tau)/\langle \tau_0, \ldots, \tau_{j-1}, \xi_1, \ldots \xi_j\rangle\label{eqn:concl.step.2^j+1-1.c}
\end{gather}
and the injection, which is an isomorphism up to degree $2^{j+1}-1$:
\begin{equation}\label{eqn:concl.step.2^j+1-1.d}
 \left ( \bigotimes_{h=0}^{j-1} \Gamma_{\mathbb{F}_2}(\mu_h) \right ) \otimes_{\mathbb{F}_2} \left ( \bigotimes_{h=1}^{j}  \Lambda_{\mathbb{F}_2}(\lambda_h) \right )    \hookrightarrow  \pi_{*,*}(MHH(M\mathbb{Z}/2)/\tau)
\end{equation}
we were looking for.
\subsubsection{\texorpdfstring{$k=2^{j+1}$}{k=2\^{}{j+1}}}
The only case left to analyse is when $k=2^{j+1}$ is a power of two. The cases $j=0$ and $j=1$ were already treated as base step for this argument; the induction step will roughly resemble the case $j=1$ (in fact, it will be a sort of generalisation of that), as the case $j=0$ was greatly simplified by having all the elements $\upgamma_i \mu_0$ of weight 0.

Our initial assumption is the conclusion of the previous step, in particular from equation \ref{eqn:concl.step.2^j+1-1.c} we see that the non trivial elements in $E^{2^{j+1}}_{0,*,*}$ with smallest degrees are:
\begin{gather*}
    \tau_j \text{ with } |\tau_j|=(2^{j+1}-1,2^{j}-1)\\
    \xi_{j+1} \text{ with } |\xi_{j+1}|=(2^{j+2}-2,2^{j+1}-1)\\
    \tau_{j+1} \text{ with } |\tau_{j+1}|=(2^{j+2}-1,2^{j+1}-1)\\
    \tau_j\xi_{j+1} \text{ with } |\tau_j\xi_{j+1}|=(3(2^{j+1}-1),3 \cdot 2^{j}-2).
\end{gather*}
We know that columns $E^{2^{j+1}}_{i,*,*}$ with $1 \leq i \leq 2^{j+1}-1$ are empty, as well as rows $E^{2^{j+1}}_{*,h,*}$ for $1 \leq h \leq 2^{j+1}-2$. The first conclusion is the presence of an element $\mu_j \in E^{2^{j+1}}_{2^{j+1},0,2^j-1}$ and a differential:
\[
    \mu_j \xmapsto{d^{2^{j+1}}} \tau_j
\]
As nothing is left in $E^{2^{j+1}}_{0,2^{j+1}-1,*}$, we conclude that $\mu_j$ is the sole generator of $E^{2^{j+1}}_{2^{j+1},0,*}$. 

We also see that $(\mu_j)^2=0$: by characteristic reasons, $d^{2^{j+1}}((\mu_j)^2)=0$. As $E^{2^{j+1}}_{1,*,*}\cong \ldots \cong E^{2^{j+1}}_{2^{j+1}-1,*,*} \cong 0$, the only possible target of a non-trivial differential out of $(\mu_j)^2$ is in the zeroth column; by degree, it lies in the module generated by $\tau_{j+1}$. But the weight of $(\mu_j)^2$ is $2(2^j-1)=2^{j+1}-2$, and  $E^{2^{j+1}}_{0,2^{j+2}-1,2^{j+1}-2}\cong 0$. Hence $(\mu_j)^2=0$ by the convergence requirements.

On the other hand, also $\mu_j \tau_j$ has to vanish. Observe that its $d^{2^{j+1}}$-differential is trivial: $d^{2^{j+1}}(\mu_j \tau_j)=\tau_j^2=0$. As the spectral sequence is first-quadrant, our only option is to have an element $\upgamma_2 (\mu_j) \in E^{2^{j+1}}_{2^{j+2},0, 2^{j+1}-2}$ and a differential:
\[
    \upgamma_2 (\mu_j) \xrightarrow{d^{2^{j+1}}} \mu_j \tau_j.
\]

We proceed now with an induction argument, that introduces a set of the divided powers $\upgamma_l(\mu_j)$ for $l \in \mathbb{N}$ and shows that they support a non-trivial differential $d^{2^{j+1}}(\upgamma_l(\mu_j))= \upgamma_{l-1}(\mu_j) \tau_j$. Concurrently, we identify the origin of all $d^{2^{j+1}}$-differentials with the divided powers of $\mu_j$. 

In detail, let $f \in \mathbb{N}$, $f \geq  2$; 
at the step $f-1$ we suppose to know the existence of the classes $\upgamma_l (\mu_j) \in E^{2^{j+1}}_{l(2^{j+1}),0,l(2^j-1)}$ for $1 \leq l \leq f-1$ and of the associated differential $d^{2^{j+1}}(\upgamma_l(\mu_j))= \upgamma_{l-1}(\mu_j) \tau_j$. They behave as divided powers in characteristic 2: this means that, if we express $l$ in base $2$ as $l=\sum_{i \geq 0} a_i 2^i$, with $a_i \in \{0,1\}$, we can identify:
\[
    \upgamma_l(\mu_j) = \prod_{i \geq 0} (\upgamma_{2^i}(\mu_j))^{a_i}. 
\]
We also know that they all square to zero, but the one indexed by the largest power of two smaller than or equal to $f-1$.

We also suppose that each class $\alpha \in E^{2^{j+1}}_{k,0,*}$, with $k \leq 2^{j+1}(f-1)$, can be identified with a finite sum:
\[
    \alpha = \sum_{i=0}^{f-1} \upgamma_{i}(\mu_{j})\beta_i
\]
with $d^{2^{j+1}}(\beta_i)=0$. In other words, every $d^{2^{j+1}}$-differential is induced by the divided powers of $(\mu_j)$. Moreover, any such sum is trivial if and only if each $\beta_i$ is: we conclude, as in remark \ref{rmk:precise.independence.mu1}, that the only algebra relations in $E^{2^{j+1}}_{*,0,*}$ (and consequently in $\pi_{*,*}(MHH(M\mathbb{Z}/p)/\tau)$, given what happens in the previous pages) involving the $\upgamma_i(\mu_j)$ are those induced by the divided power structure.



\begin{rmk}\label{rmk:consequence.of.ind.step.muj}
    As proven in remark \ref{rmk:consequence.of.ind.step.mu1} for the divided powers of $\mu_1$, one can show that the induction hypothesis for $f-1$ implies that the $d^{2^{j+1}}$ differential surjects onto $E^{2^{j+1}}_{h, 2^{j+1}-1,*}$ for $h \leq 2^{j+1}(f-2)$. Because of the above statements, any element in this region can in fact be written as:
    \[
        \sum_{i=0}^{f-2}\upgamma_i(\mu_j) \beta_i \tau_j
    \]
    with $\beta_i \in E^{2^{j+1}}_{*,0,*}$ such that $d^{2^{j+1}}(\beta_i)=0$.
    Now, observe that $h+2^{j+1} \leq 2^{j+1}(f-1)$: the induction hypothesis states that there exist an element
    \[
        \sum_{i=0}^{f-2}\upgamma_{i+1}(\mu_j) \beta_i \in E^{2^{j+1}}_{h+2^{j+1}, 0,*}
    \]
    and a differential:
    \[
        \sum_{i=0}^{f-2}\upgamma_{i+1}(\mu_j) \beta_i \xrightarrow{d^{2^{j+1}}} \sum_{i=0}^{f-2}\upgamma_i(\mu_j) \beta_i \tau_j.
    \]
    Given the tensor product decomposition of the $E^{2^{j+1}}$-page, a similar deduction can be applied to all rows beginning with a multiple of $\tau_j$: at stage $f-1$ we see that they vanish up to degree $2^{j+1}(f-2)$. Symmetrically, any multiple of some divided power of $(\mu_j)$ (hence any homogeneous sum of them) in degree smaller than or equal to $2^{j+1}(f-2)$ disappears from the $E^{2^{j+1}+1}$-page, as it is either the target or the source of a non-trivial $d^{2^{j+1}}$-differential. The first case corresponds to the non-trivial multiples of $\tau_j$, while in the second we find the non $\tau_j$-divisible elements.
\end{rmk}

We proceed now with the induction argument: we already saw some base steps above, so suppose the hypothesis holds for some integer $f-1$; we wish to prove it for the next integer $f$. We will distinguish two cases.

If $f$ is not a power of 2 then $\upgamma_f (\mu_j)$ is just given by a product of already introduced classes, and the Leibniz rule tells us that $d^{2^{j+1}}(\upgamma_f (\mu_j))=\upgamma_{f-1} (\mu_j) \tau_j$. This was proven in lemma \ref{lmm:leibniz.on.div.pow} when $j=1$, and the proof for generic $j$ is analogous. It squares to zero since at least one of its factors is known to square to zero. 

Let now $\alpha \in E^{2^{j+1}}_{k,0,*}$, with $k \leq 2^{j+1} f$; we would like to have:
\[
    \alpha = \sum_{i=0}^{f} \upgamma_{i}(\mu_{j})\beta_i
\]
with $d^{2^{j+1}}(\beta_i)=0$. If $d^{2^{j+1}}(\alpha)=0$ we are done; if $d^{2^{j+1}}(\alpha)\neq0$, we can write, by the previous inductive step (the first degree of $d^{2^{j+1}}(\alpha)$ is at most $2^{j+1}(f-1)$) and the tensor product decomposition of the $E^{2^{j+1}}$-page:
\[
    d^{2^{j+1}}(\alpha)= \left( \sum_{i=0}^{f-1} \upgamma_{i}(\mu_{j})\beta_i \right) \tau_j
\]
with $d^{2^{j+1}}(\beta_i)=0$. Consider then the element $\sum_{i=0}^{f-1} \upgamma_{i+1}(\mu_{j})\beta_i$; its $d^{2^{j+1}}$-differential is:
\[
    d^{2^{j+1}}\left (\sum_{i=0}^{f-1} \upgamma_{i+1}(\mu_{j})\beta_i \right)= \left( \sum_{i=0}^{f-1} \upgamma_{i}(\mu_{j})\beta_i \right) \tau_j = d^{2^{j+1}}(\alpha)
\]
Hence $\alpha - \sum_{i=0}^{f-1} \upgamma_{i+1}(\mu_{j})\beta_i$ is a $ d^{2^{j+1}}$-cycle, as required.

Next, suppose we have a homogeneous equation in degree smaller than or equal to $2^{j+1}f$:
\[
    \sum_{l=0}^{f} \beta_l \upgamma_l (\mu_j)=0.
\]
with $d^{2^{j+1}}(\beta_l)=0$. Apply the $d^{2^{j+1}}$-differential:
\[
    0=\sum_{l=1}^{f} \beta_l \upgamma_{l-1} (\mu_j) \tau_j
\]
which, by the product decomposition on the $E^{2^{j+1}}$ page implies:
\[
    \sum_{l=1}^{f} \beta_l \upgamma_{l-1} (\mu_j)=0
\]
This homogeneous equation happens in a degree smaller than or equal to $2^{j+1}(f-1)$, so, by induction hypothesis, one must have $\beta_1=\ldots=\beta_l=0$; then also $\beta_0=0$, as required.

As one might expect from the base step, whenever $f=2^i$ is a power of two the proof gets more involved. In fact, by the presentation of divided powers in characteristic two, at powers of two we encounter algebraically independent elements. 
To see that this is the case in our spectral sequence, we will in particular prove:
\begin{itemize}
    \item No polynomial in the $\upgamma_l (\mu_j)$, for $l  <  2^i$, differentiates to $\upgamma_{2^i-1} (\mu_j) \tau_l$.
    \item All differentials out of $\upgamma_{2^i-1} (\mu_j) \tau_l$ are trivial differential.
    \item All differentials out of $(\upgamma_{2^{i-1}} (\mu_j))^2$ are trivial differential.
\end{itemize}
The first two conclusions, combined with the convergence result for the spectral sequence, prove the existence of an element $\upgamma_{2^i}(\mu_j) \in E^{2^{j+1}}_{2^{i+j+1},0, 2^i (2^j-1)}$ and of a differential $\upgamma_{2^i}(\mu_j)\xrightarrow{d^{2^{j+1}}}\upgamma_{2^i-1}(\mu_j) \tau_j$. The third step, on the other hand, extends our knowledge of the algebra structure of the $E^{2^{j+1}}$-page, hence of $\pi_{*,*}(MHH(M\mathbb{Z}/p)/\tau)$. Each of the three steps will be proven by a contradiction argument. The proof of the remaining statements in this induction step are essentially similar to the case of $f$ not a power of $2$ we have just seen.

So, suppose that there is some element $x$, equal to a homogeneous polynomial in the $\upgamma_{l} (\mu_j)$, $1 \leq l \leq 2^i-1$, with $d^{2^{j+1}}(x)=\upgamma_{2^i-1} (\mu_j) \tau_l$. We know that $(\upgamma_{l}(\mu_j))^2=0$ for all $1 \leq l \leq 2^i-1$ but $l=2^{i-1}$, and that $d^{2^{j+1}}((\upgamma_{2^{i-1}}(\mu_j))^2)=0$ by characteristics, so we may suppose:
\[
     x= \sum_{l=1}^{2^i-1}\beta_l\upgamma_{l} (\mu_j).
\]
Observe that $x$ must lie in degree $2^i\cdot 2^{j+1}$, so all the $\beta_l$ have degree bounded above by $(2^i-1)2^{j+1}$; here we can apply the previous induction step and suppose without loss of generality that $d^{2^{j+1}}(\beta_l)=0$ for all $l$. So:
\[
     d^{2^{j+1}}(x)= \sum_{l=1}^{2^i-1}\beta_l\upgamma_{l-1} (\mu_j) \tau_j=\upgamma_{2^i-1} (\mu_j) \tau_j
\]
By the product decomposition of the $E^{2^{j+1}}$-page, we obtain the homogeneous equation:
\[
    \sum_{l=1}^{2^i-1}\beta_l\upgamma_{l-1} (\mu_j)-\upgamma_{2^i-1} (\mu_j) =0
\]
which sits in degree $(2^i -1)2^{j+1}$; by induction hypothesis, such an equation is impossible. Hence, no polynomial in the already introduced divided powers can differentiate to $\upgamma_{2^i-1} (\mu_j) \tau_j$.

Next, we show that $\upgamma_{2^i-1} (\mu_j) \tau_j$ cannot support a non-trivial differential.
It is worth looking again at the elements with the lowest degrees in:
\begin{equation}\label{eqn:A/tau_j-1,xi_j}
\begin{aligned}
    E^{2^{j+1}}_{0, *, *} &\cong \pi_{*,*}(\mathcal{A}(2)/\tau)/\langle \tau_0, \ldots, \tau_{j-1}, \xi_1, \ldots \xi_j\rangle \\
    &\cong \mathbb{F}_2[\tau_i,\xi_{i+1}]_{i \geq j}/\langle \tau_i^2 \rangle.
\end{aligned}
\end{equation}
    
Those are:
\begin{gather*}
    \tau_j \text{ with } |\tau_j|=(2^{j+1}-1,2^{j}-1)\\
    \xi_{j+1} \text{ with } |\xi_{j+1}|=(2^{j+2}-2,2^{j+1}-1)\\
    \tau_{j+1} \text{ with } |\tau_{j+1}|=(2^{j+2}-1,2^{j+1}-1)\\
    \tau_j\xi_{j+1} \text{ with } |\tau_j\xi_{j+1}|=(3(2^{j+1}-1),3 \cdot 2^{j}-2).
\end{gather*}
In particular, notice that the second shortest differential (after the $d^{2^{j+1}}$) that can leave the horizontal line is a $d^{2^{j+2}-1}$; moreover, the only element with weight $2^{j}-1$ is $\tau_j$, all the other elements have weight at least $2^{j+1}-1$. 

If one supposes by contradiction that there is some $k_1 \in \mathbb{N}$, such that $d^{k_1}(\upgamma_{2^i-1} (\mu_j) \tau_j)$ is non trivial, one must have $k_1  \geq  2^{j+1}+1$, as all the differentials up to the $d^{2^{j+1}}$ are trivial on this element. The degrees of $\upgamma_{2^i-1} (\mu_j) \tau_j$ are:
\[
    |\upgamma_{2^i-1} (\mu_j) \tau_j|=((2^i-1)2^{j+1},2^{j+1}-1,2^i(2^j-1))
\]
so the image via $d^{k_1}$ in the $E^{k_1}$ has degrees: 
\[
|d^{k_1}(\upgamma_{2^i-1} (\mu_j) \tau_j)|=((2^i-1)2^{j+1}-k_1,2^{j+1}+k_1-2,2^i(2^j-1)).
\]
Observe that the vertical degree is $2^{j+1}+k_1-2 \geq 2^{j+1}+2^{j+1}+1-2= 2^{j+2}-1$: the presence of $\tau_j+1$ in such degree implies that such differential could be non-trivial, hence, at least from this initial analysis, we cannot refine the lower bound for $k_1$.

As in remark \ref{rmk:representatives.in.E4}, we can identify $d^{k_1}(\upgamma_{2^i-1} (\mu_j) \tau_j)$ with a product $\alpha_1 \omega_1$ in the $E^{2^{j+1}}$-page, such that $\alpha_1 \in E^{2^{j+1}}_{*,0,*}$ and $\omega_1 \in E^{2^{j+1}}_{0,*,*}$. We will have in particular:
\begin{gather*}
    |\alpha_1|= ((2^i-1)2^{j+1}-k_1,2^i(2^j-1)-w_1)\\
    |\omega_1|= (2^{j+1}+k_1-2,w_1).
\end{gather*}
Given the structure of $E^{2^{j+1}}_{0,*,*}$ appearing in equation \ref{eqn:A/tau_j-1,xi_j}, we must have $w_1  \geq  2^{j+1}-1$. 

We will first exclude the case $k_1=(2^i-1)2^{j+1}$, or equivalently, $\alpha_1=1$ and $w_1=2^i(2^j-1)$, and then have a look at the other possibilities. We will use the following:

\begin{lemma}\label{lmm:d/w.for.E.2.j+1}
Let $x$ be a homogeneous element in:
\[
    E^{2^{j+1}}_{0,*,*} \cong \pi_{*,*}(\mathcal{A}(2)/\tau)/\langle \tau_0, \ldots, \tau_{j-1}, \xi_1, \ldots \xi_j\rangle
\]
If $|x|=(d,w)$, then $2 \leq \frac{d}{w} \leq 2+ \frac{1}{2^j-1}$.
\end{lemma}

We postpone the proof of this at the end of the induction argument.

By the induction hypothesis (in particular remark \ref{rmk:consequence.of.ind.step.muj}) $\omega_1$ can be seen as an element of $\pi_{*,*}(\mathcal{A}(2)/\tau)/\langle \tau_0, \ldots, \tau_{j-1}, \xi_1, \ldots \xi_j\rangle$; we can then apply lemma \ref{lmm:d/w.for.E.2.j+1} to get constraints on its degree and weight. Suppose first that $d^{k_1}$ hits the zeroth column, with $|\omega_1|= (2^{j+1}+(2^i-1)2^{j+1}-2,2^i(2^j-1))=(d_{\omega_1},w_{\omega_1})$. Then:
\begin{align*}
    \frac{d_{\omega_1}}{w_{\omega_1}}&=
    \frac{2^{j+1}+(2^i-1)2^{j+1}-2}{2^i(2^j-1)}=
    \frac{2^{1+j+1}-2}{2^{i+j}-2^i)}\\ &=
    \frac{2^{i+j+1}-2^{i+1}+2^{i+1}-2}{2^{i+j}-2^i)}\\ &=
    2+\frac{2^{i+1}-2}{2^i(2^j-1)}=
    2+\frac{2-\frac{2}{2^i}}{2^j-1}
\end{align*}

However,
\[
    \frac{2-\frac{2}{2^i}}{2^j-1}> \frac{1}{2^j-1}
\]
for any $i>1$, which is incompatible with lemma \ref{lmm:d/w.for.E.2.j+1}, providing the desired contradiction.

Suppose now $k_1 < (2^i-1)2^{j+1}$. By the convergence of the spectral sequence, $\alpha_1$ must support a non-trivial differential. This differential must be longer than a $d^{2^{j+1}}$: in fact,
\begin{equation}\label{eq:deg(alpha1)}
    deg(\alpha_1)=(2^i-1)2^{j+1}-k_1 < (2^i-2)2^{j+1}
\end{equation} 
so the induction hypothesis assures that, if $\alpha_1$ supported a non-trivial $d^{2^{j+1}}$-differential, also $\alpha_1 \omega_1$ would disappear after the $E^{2^{j+1}}$ page. So let $\alpha_1$ support a $d^{k_2}$ differential; by the observations after \ref{eqn:A/tau_j-1,xi_j}, $k_2  \geq  2^{j+2}-1$. Again, pick $\alpha_2 \in E^{2^{j+1}}_{*,0,*}$ and $\omega_2 \in E^{2^{j+1}}_{0,*,*}$ such that their product $\alpha_2 \omega_2$ is non-zero in the $E^{k_2}$ page and has the same degrees as $d^{k_2}(\alpha_1)$. 
We have:
\begin{gather*}
    |\alpha_2|= ((2^i-1)2^{j+1}-k_1-k_2,2^i(2^j-1)-w_1-w_2)\\
    |\omega_2|= (k_2-1,w_2).
\end{gather*}
Again by what was said after \ref{eqn:A/tau_j-1,xi_j}, $w_2  \geq  2^{j+1}-1$.

Repeat the last passage $(n-2)$-times, until the $d^{k_n}$-differential of $\alpha_{n-1}$ reaches the zeroth column; we associate to it the element $\omega_n \in E^{2^{j+1}}_{0,*,*}$. See figure \ref{fig:induction.E2j+1} for a visualisation of the elements and differentials involved.
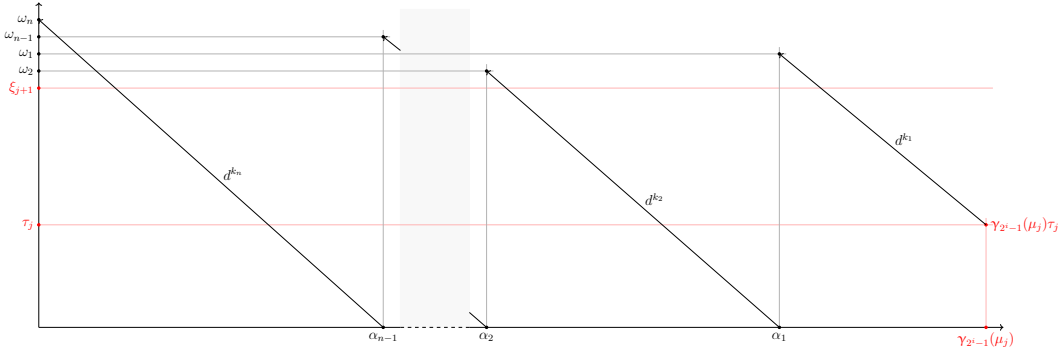
\begin{figure}[htb]
    \centering
    \adjustbox{max width=\textwidth, center}{\begin{tikzpicture}
        \filldraw[almostwhite] (10.5,0) rectangle (12.5,9.3);

        \draw[red!30] (0,3) -- (27.7,3);
        \draw[red!30] (27.5,0) -- (27.5,3.2);
        
        \draw[red!30] (0,7) -- (27.7,7); 

        \draw[gray!60] (0,8) -- (21.7,8);
        \draw[gray!60] (21.5,0) -- (21.5,8.2);

        \draw[gray!60] (0,7.5) -- (13.2,7.5);
        \draw[gray!60] (13,0) -- (13,7.7);

        \draw[gray!60] (0,8.5) -- (10.2,8.5);
        \draw[gray!60] (10,0) -- (10,8.7);

        \draw[black, ->] (0,0) -- (0,9.5);
        \draw[black] (0,0) -- (10.5,0);
        \draw[black, dashed] (10.5,0) -- (12.5,0);
        \draw[black, ->] (12.5,0) -- (28,0);

        {\filldraw[red]
        (27.5,3) circle (1pt) node[right]{\normalsize $\upgamma_{2^i-1}(\mu_j)\tau_j$}
        (27.5,0) circle (1pt) node[below]{\normalsize $\upgamma_{2^i-1}(\mu_j)$}
        (0,3) circle (1pt) node[left]{\normalsize $\tau_j$}
        (0,7) circle (1pt) node[left]{\normalsize $\xi_{j+1}$}
        ;}
        
        {\filldraw[black]
        (21.5,8) circle (1pt) 
        (21.5,0) circle (1pt) node[below]{\normalsize $\alpha_1$}
        (0,8) circle (1pt) node[left]{\normalsize $\omega_1$}
        
        (13,7.5) circle (1pt) 
        (13,0) circle (1pt) node[below]{\normalsize $\alpha_2$}
        (0,7.5) circle (1pt) node[left]{\normalsize $\omega_2$}
        
        (10,8.5) circle (1pt) 
        (10,0) circle (1pt) node[below]{\normalsize $\alpha_{n-1}$}
        (0,8.5) circle (1pt) node[left]{\normalsize $\omega_{n-1}$}
        
        (0,9) circle (1pt) node[left]{\normalsize $\omega_n$}
        ;}

        \draw[black, ->] (27.5,3) -- (21.5,8) node[right=6pt,  midway] {\normalsize  $d^{k_1}$};
        \draw[black, ->] (21.5,0) -- (13,7.5)  node[right=6pt,  midway] {\normalsize  $d^{k_2}$};
        \draw[black] (13,0) -- (12.5,0.4375);
        \draw[black, ->] (10.5,8.1) -- (10,8.5);
        \draw[black, ->] (10,0) -- (0,9)  node[right=6pt,  midway] {\normalsize  $d^{k_n}$};
        
    \end{tikzpicture}}
    \caption{A visualisation of what should happen if $\upgamma_{2^i-1}(\mu_j)\tau_j$ supported a non trivial differential.}
    \label{fig:induction.E2j+1}
\end{figure}
If $|\omega_n|=(0, k_n-1, w_n)$, one has equation on the degrees and weights of $\alpha_{n-1}$:
  \begin{gather*}
    deg(\alpha_{n-1})=(2^i-1)2^{j+1}-k_1-k_2- \ldots -k_{n-1}=k_n \\
    w(\alpha_{n-1})=2^i(2^j-1)-w_1-w_2-\ldots-w_{n-1}=w(\omega_n)=w_n.
\end{gather*}
So:
\begin{equation} \label{eqn:deg.wei.contr.mu.l}
  \begin{cases}
    (2^i-1)2^{j+1}=\sum_{l=1}^n k_l\\
    2^i(2^j-1)= \sum_{l=1}^n w_l.
\end{cases}  
\end{equation}

Observe that all the differentials $d^{k_2}, \ldots, d^{k_n}$ satisfy the same constraint: $k_l  \geq  2^{j+2}-1$ and consequently $w_l  \geq  2^{j+1}-1$.

There are also other conditions on these quantities, coming from lemma \ref{lmm:d/w.for.E.2.j+1}. By the induction hypothesis (in particular remark \ref{rmk:consequence.of.ind.step.muj}), as all the $\omega_l$ survive after the $E^{2^{j+1}}$-page, they can be seen as elements of $\pi_{*,*}(\mathcal{A}(2)/\tau)/\langle \tau_0, \ldots, \tau_{j-1}, \xi_1, \ldots \xi_j\rangle$. In particular, there are the already seen constraints for $\omega_1$: 
\[
    2^{j+1}+k_1-2\leq w_1\left ( 2+ \frac{1}{2^j-1} \right)
\]
and similar ones for all the other $\omega_l$, $l \geq 2$:
\[
    k_l-1\leq w_l\left ( 2+ \frac{1}{2^j-1} \right).
\]

Now, sum on all $l$:
\begin{gather*}
    2^{j+1}+k_1-2+ \sum_{l=2}^n (k_l-1)\leq \left ( 2+ \frac{1}{2^j-1} \right)\sum_{l=1}^n w_l \\
    2^{j+1}-n-1+ \sum_{l=1}^n k_l\leq \left ( 2+ \frac{1}{2^j-1} \right)\sum_{l=1}^n w_l
\end{gather*}
Plug in equations \ref{eqn:deg.wei.contr.mu.l}:
\begin{gather*}
    2^{j+1}-n-1+ (2^i-1)2^{j+1} \leq \left ( 2+ \frac{1}{2^j-1} \right)2^i(2^j-1)\\
    2^{j+1}-n-1+2^{i+j+1}-2^{j+1} \leq  2^{i+l+1}-2^{i+1}+2^i \\
    2^i-1 \leq n.
\end{gather*}
Recall that each $w_l  \geq  2^{j+1}-1$, so, from the second equality of \ref{eqn:deg.wei.contr.mu.l}:
\begin{gather*}
    2^i(2^j-1)= \sum_{l=1}^n w_l  \geq  n(2^{j+1}-1)  \geq  (2^i-1)(2^{j+1}-1)\\
    2^{i+j}-2^i  \geq  2^{i+j+1}-2^{i}-2^{j+1}+1\\
    0  \geq  2^{i+j}-2^{j+1}+1
\end{gather*}
which does not hold true for any integers $i \geq  1$ and $j$ positive. 

Hence there must be an element $\upgamma_{2^i}(\mu_j)$, algebraically independent of the previous divided powers of $\mu_j$, with:
\begin{gather*}
    |\upgamma_{2^i}(\mu_j)|=(2^{i+j+1}, 2^i(2^j-1)) \\
    d^{2^{j+1}}(\upgamma_{2^i}(\mu_j))=\upgamma_{2^i-1} (\mu_j) \tau_j.
\end{gather*}

We can now prove that every class $\alpha \in E^{2^{j+1}}_{k,0,*}$, for $k \leq 2^i 2^{j+1}$, is a finite sum $\alpha= \sum_{l=
0}^{2^i} \beta_l \cdot \upgamma_{l} (\mu_j)+ \delta$, with $d^{2^{j+1}}(\beta_l)=d^{2^{j+1}}(\delta)=0$. In particular, all non-trivial $d^{2^{j+1}}$-differential originate from the divided powers $\upgamma_{l}(\mu_j)$. Observe that $(\upgamma_{2^{i-1}} (\mu_j))^2$ sits in the largest degree possible, and $d^{2^{j+1}}((\upgamma_{2^{i-1}} (\mu_j))^2)=0$ by the Leibniz rule. Observe that we have not yet proven that $(\upgamma_{2^{i-1}} (\mu_j))^2=0$, hence, at this stage, $(\upgamma_{2^{i-1}} (\mu_j))^2$ could appear as a nonzero summand of $\beta_0$.

This assertion is proven as for $f$ not a power of 2: consider $\alpha$ as above; observe that we can assume $k  \geq  (2^i-1)2^{j+1}+1$, as smaller $k$'s are covered by the previous induction step. Let:
\[
    d^{2^{j+1}}(\alpha)= x \tau_j
\]
Since $x \in E^{2^{j+1}}_{*,0,*}$ has degree smaller then or equal to $(2^i-1)2^{j+1}$, by the induction hypothesis we can write:
\[
    x=\sum_{l=1}^{2^i} \beta_l \cdot \upgamma_{l-1} (\mu_j)
\]
with $d^{2^{j+1}}(\beta_l)=0$ for all $l$. Then:
\[
    d^{2^{j+1}}(\sum_{l=1}^{2^i} \beta_l \cdot \upgamma_{l} (\mu_j))= \sum_{l=1}^{2^i} \beta_l \cdot \upgamma_{l-1} (\mu_j) \tau_j =d^{2^{j+1}}(\alpha)
\]
so 
\[
    \alpha=\sum_{l=1}^{2^i} \beta_l \cdot \upgamma_{l} (\mu_j)+ \delta
\]
with $\delta$ a $d^{2^{j+1}}$ cycle.

\begin{rmk} \label{rmk:cons.of.ind.step.muj.2}
    We can then apply remark \ref{rmk:consequence.of.ind.step.muj}: fix any class $\alpha \omega \in E^{2^{j+1}}$ with $\alpha \in E^{2^{j+1}}_{l,0,*}$,  $l \leq  2^{j+1}(2^i-1)$, and $\omega \in E^{2^{j+1}}_{0,*,*}$, such that $d^{2^{j+1}}(\alpha) \neq 0$. Then either we have that $d^{2^{j+1}}(\alpha \omega) \neq 0$, or there exists an element $x \in E^{2^{j+1}}_{l+2^{j+1},*,*}$ with $d^{2^{j+1}}(x)= \alpha \omega$.
\end{rmk}

Next, one must prove that $(\upgamma_{2^{i-1}} (\mu_j))^2=0$ by showing that it cannot support any non-trivial differentials. The argument resembles the one used with $\upgamma_{2^i-1}(\mu_j)\tau_j$.

By contradiction, suppose $(\upgamma_{2^{i-1}} (\mu_j))^2$ supports a non-trivial $d^{k_1}$ differential. As $(\upgamma_{2^{i-1}} (\mu_j))^2$ is a square in the $E^{2^{j+1}}$ page, $k_1>2^{j+1}$;  by equation \ref{eqn:A/tau_j-1,xi_j} it must be $k_1  \geq  2^{j+2}-1$. As above, we identify $d^{k_1}((\upgamma_{2^{i-1}} (\mu_j))^2)$ with a product $\alpha_1 \omega_1 \in E^{2^{j+1}}$ sharing the same degrees, with $\alpha_1 \in E^{2^{j+1}}_{*,0,*}$ and $\omega_1 \in E^{2^{j+1}}_{0,*,*}$; in particular:
\begin{gather*}
    |\alpha_1|= (2^{i+j+1}-k_1,0, 2^i(2^j-1)-w_1)\\
    |\omega_1|= (0,k_1-1,w_1).
\end{gather*}
By \ref{eqn:A/tau_j-1,xi_j}, $w_1  \geq  2^{j+1}-1$. 

If $\alpha_1=1$, one can apply the considerations at the end of this part of the proof (choose $n=1$) to deduce a contradiction. If $\alpha_1 \neq 1$, it must have a non-trivial differential at some page. We claim that it cannot be a $d^{2^{j+1}}$ differential, so it has to be longer. Observe that since:
\[
    \deg(\alpha_1)=2^{i+j+1}-k_1\leq 2^{i+j+1}-2^{j+2}+1=(2^i-2)2^{j+1}+1
\]
which is:
\[
    \deg(\alpha_1)\leq (2^i-2)2^{j+1}+1\leq (2^i -1)2^{j+1}
\]
for $j  \geq  0$. So the element $\alpha_1\omega_1$ lives in the region controlled by remark \ref{rmk:cons.of.ind.step.muj.2}; in particular, if $\alpha_1$ had a non-trivial $d^{2^{j+1}}$-differential, $\alpha_1\omega_1$ would disappear after the $E^{2^{j+1}}$ page. So $\alpha_1$ must support a longer differential $d^{k_2}$; by equation \ref{eqn:A/tau_j-1,xi_j} it must be $k_2  \geq  2^{j+2}-1$.

We apply the above considerations to this step as well and identify  $d^{k_2}(\alpha_1)$ with a product $\alpha_2 \omega_2 \in E^{2^{j+1}}$ such that:
\begin{gather*}
    |\alpha_2|= (2^{i+j+1}-k_1-k_2,0, 2^i(2^j-1)-w_1-w_2)\\
    |\omega_2|= (0,k_2-1,w_2).
\end{gather*}
In particular, $w_2  \geq  2^{j+1}-1$. This process is repeated $n-1$ times until we reach a differential $d^{k_n}(\alpha_{n-1}) \in E^{k_n}_{0,*,*}$; we identify the image of this last differential with an element $\omega_n \in E^{2^{j+1}}_{0,k_n-1,w_n}$ sharing the same degrees. Since analogous considerations on degrees and weights apply at each step, we get the equations:
\begin{gather*}
    deg(\alpha_{n-1})=2^{i+j+1}-k_1-k_2- \ldots -k_{n-1}=k_n \\
    w(\alpha_{n-1})=2^i(2^j-1)-w_1-w_2-\ldots-w_{n-1}=w(\omega_n)=w_n.
\end{gather*}
So:
\begin{equation} \label{eqn:kl.wl.gamma.muj.2}
  \begin{cases}
    2^{i+j+1}=\sum_{l=1}^n k_l\\
    2^i(2^j-1)= \sum_{l=1}^n w_l.
\end{cases}  
\end{equation}

Now, because they survive after the $E^{2^{j+1}}$ page and live in the range covered by the induction hypothesis, the elements $\omega_j$ can be thought to live in $\pi_{*,*}(\mathcal{A}(2)/\tau)/\langle \tau_0, \ldots, \tau_{j-1}, \xi_1, \ldots \xi_j\rangle$; we can then apply \ref{lmm:d/w.for.E.2.j+1} to the pairs $(k_l-1,w_l)$ to get:
\begin{gather*}
    \sum_{l=1}^n (k_l-1)\leq \sum_{l=1}^n w_l\left ( 2+ \frac{1}{2^j-1} \right) 
    \intertext{which becomes: }
    -n+ \sum_{l=1}^n k_l\leq \left( \sum_{l=1}^n w_l\right )\left ( 2+ \frac{1}{2^j-1} \right)
\end{gather*}
Plug in equations \ref{eqn:kl.wl.gamma.muj.2}:
\begin{equation*}
    -n+ 2^{i+j+1} \leq 2^i(2^j-1) \left ( 2+ \frac{1}{2^j-1} \right)
\end{equation*}
which gives:
\begin{equation*}    
    n  \geq  2^{i+j+1} - 2^i(2^j-1) \left ( 2+ \frac{1}{2^j-1} \right) = 2^{i+j+1} - 2^{i+j+1}+2^{i+1}-2^i = 2^i
\end{equation*}

Back to \ref{eqn:kl.wl.gamma.muj.2}, using the estimate $w_l  \geq  2^{j+1}-1$ valid for all $l$:
\begin{gather*}
    2^i(2^j-1)= \sum_{l=1}^n w_l  \geq  n (2^{j+1}-1)  \geq  2^i (2^{j+1}-1)
\end{gather*}
which is a contradiction. Hence $(\upgamma_{2^{i-1}} (\mu_j))^2=0$ as desired.

We conclude that any homogeneous polynomial in degree smaller than or equal to $2^i2^{j+1}$ can be rewritten as a linear combination of the $\upgamma_l(\mu_j)$ with coefficients in $d^{2^j+1}$-cycles; the polynomial is moreover trivial if and only if all coefficients are trivial. This inductively proves that the divided powers  $\upgamma_l(\mu_j)$ are algebraically independent of the $d^{2^{j+1}}$-cycles, in other words, from everything else appearing in $E^{2^{j+1}}_{*,0,*}$.

Combining all together, we can conclude that there are a commutative triangle:
\[
 \begin{tikzcd}
        E^{2^{j+1}+1}_{*,0,*} \arrow[r, hook] \arrow[rd, "\cong"] & 
        E^{2^{j+1}}_{*,0,*} \arrow[d, two heads] \\
        & E^{2^{j+1}}_{*,0,*}/\langle \upgamma_i (\mu_j) \rangle_{i \geq 1}
    \end{tikzcd}    
\]
an isomorphism:
\[
    E^{2^{j+1}+1}_{0,*,*} \cong E^{2^{j+1}}_{0,*,*}/\langle \tau_j \rangle
\]
and a tensor product decomposition:
\[
    E^{2^{j+1}+1}_{\star,\bullet,*} \cong E^{2^{j+1}+1}_{\star,0,*}\otimes_{\mathbb{F}_2} E^{2^{j+1}+1}_{0,\bullet,*}
\]
Plugging in the isomorphisms from the previous induction passages, we get:
\begin{gather*}
    E^{2^{j+1}+1}_{*,0,*} \cong  \pi_{*,*}(MHH(M\mathbb{Z}/2)/\tau)/\langle \upgamma_i (\mu_0), \ldots \upgamma_i (\mu_j), \lambda_1, \ldots, \lambda_{j}\rangle_{i  \geq  0}   \\
    E^{2^{j+1}+1}_{0,*,*} \cong \pi_{*,*}(\mathcal{A}(2)/\tau)/\langle \tau_0, \ldots, \tau_j, \xi_1,\ldots \xi_j \rangle .
\end{gather*}

Finally, since the spectral sequence is first-quadrant,  the injection:
\begin{equation*}
    \left ( \bigotimes_{h=0}^{j} \Gamma_{\mathbb{F}_2}(\mu_h) \right ) \otimes_{\mathbb{F}_2} \left ( \bigotimes_{h=1}^{j}  \Lambda_{\mathbb{F}_2}(\lambda_h) \right )  \hookrightarrow  \pi_{*,*}(MHH(M\mathbb{Z}/2)/\tau)
\end{equation*}
is an isomorphism in degrees smaller or equal to $2^{j+1}$. 

This concludes this induction step.

\begin{proof}[Proof of Lemma \ref{lmm:d/w.for.E.2.j+1}]
This proof is very similar to that of lemma \ref{lmm:d/w.for.E4}, which can be seen as a special case of lemma \ref{lmm:d/w.for.E.2.j+1} by choosing $j=1$. 

Any homogeneous element in 
\[
    \pi_{*,*}(\mathcal{A}(2)/\tau)/\langle \tau_0, \ldots, \tau_{j-1}, \xi_1, \ldots \xi_j\rangle
\]
is a sum of products of the form:
\[
    x=\tau_j^{\epsilon_j} \tau_{j+1}^{\epsilon_{j+1}} \cdots \tau_k^{\epsilon_k}\xi_{j+1}^{l_{j+1}} \cdots \xi_k^{l_k}
\]
with $\epsilon_i \in \{0,1\}$ and $l_i \in \mathbb{N}$. Its bidegree is:
\begin{align*}
    &|x|=\\
    &\left (\sum_{i=j}^k \epsilon_i(2^{i+1}-1)+\sum_{i=j+1}^k l_j(2^{i+1}-2), \sum_{i=j}^k \epsilon_i(2^{i}-1)+\sum_{i=j+1}^k l_j(2^{i}-2) \right )\\
    &=\left (\sum_{i=j}^k \left[(\epsilon_i+l_i)(2^{i+1}-2)+ \epsilon_i\right], \sum_{i=j}^k (\epsilon_i+l_i)(2^{i}-1) \right )=(d,w)
\end{align*}
where we set $l_j=0$. The ratio degree/weight is given by:
\[
\frac{d}{w}= \frac{\sum_{i=j}^k \left[(\epsilon_i+l_i)(2^{i+1}-2)+ \epsilon_i\right]}{ \sum_{i=j}^k (\epsilon_i+l_i)(2^{i}-1)}
\]

Observe that:
\[
\left[(\epsilon_i+l_i)(2^{i+1}-2)+ \epsilon_i\right]=0 \iff \epsilon_i=l_i=0 \iff (\epsilon_i+l_i)(2^{i}-1) =0
\]
We can then estimate $\frac{d}{w}$ by giving uniform bounds on those summands:
\[
\frac{(\epsilon_i+l_i)(2^{i+1}-2)+ \epsilon_i}{ (\epsilon_i+l_i)(2^{i}-1)},
\]
for which at least one between $\epsilon_i$ and $l_i$ is nonzero. On one side we have:
\[
\frac{(\epsilon_i+l_i)(2^{i+1}-2)+ \epsilon_i}{ (\epsilon_i+l_i)(2^{i}-1)} \geq  \frac{(\epsilon_i+l_i)(2^{i+1}-2)}{ (\epsilon_i+l_i)(2^{i}-1)}=2.
\]
On the other side:
\[
\frac{(\epsilon_i+l_i)(2^{i+1}-2)+ \epsilon_i}{ (\epsilon_i+l_i)(2^{i}-1)}= 2+\frac{\epsilon_i}{ (\epsilon_i+l_i)(2^{i}-1)} \leq 2+\frac{1}{ 2^{i}-1}\leq  2+\frac{1}{ 2^{j}-1}
\]
So:
\[
     2 \leq \frac{d}{w} \leq 2+\frac{1}{ 2^{j}-1}
\]
as we wanted to show.
\end{proof}

This concludes the argument for the spectral sequence at $p=2$.

\clearpage
\relax


\input{diagrams/E2}
\begin{figure}
\centering
\adjustbox{max width= \textwidth, center}{
\pgfsetshortenend{0.6pt}
\pgfsetshortenstart{0.6pt}
\begin{tikzpicture}[scale=2\textwidth/190cm,line width=1pt]

\draw[step=1cm , almostwhite, very thin] (0,0) grid (189.5,189.5);
\draw[step=10cm ,gray,very thin] (0,0) grid (189.5,189.5);
\draw[gray, line width=0.35pt,->](0,0)--(190,0);
\draw[gray, line width=0.35pt,->] (0,0)--(0,190);

\node foreach \y in {1,...,18} at (0, 10*\y) [left, scale=0.45] {\y};
\node foreach \x in {1,...,18} at (10*\x, 0) [below, scale=0.45] {\x};

{\filldraw[weight0]
(0,0) circle (6pt) node[below left, scale=0.45]{ 1}
;}
{\filldraw[weight2]
(0,40) circle (6pt) node[rotate=90, right, scale=0.45]{$\xi_1$}
(1,50) circle (6pt) node[rotate=90, right, scale=0.45]{$\tau_1$}
;}
{\filldraw[weight4]
(0,80) circle (6pt) node[rotate=90, right, scale=0.45]{$\xi_1^2$}
(1,90) circle (6pt) node[rotate=90, right, scale=0.45]{$\tau_1\xi_1$}
;}
{\filldraw[weight6]
(0,120) circle (6pt) node[rotate=90, right, scale=0.45]{$\xi_1^3$}
(1,130) circle (6pt) node[rotate=90, right, scale=0.45]{$\tau_1\xi_1^2$}
;}
{\filldraw[weight8]
(0,160) circle (6pt) node[rotate=90, right, scale=0.45]{$\xi_1^4$}
(1,160) circle (6pt) node[rotate=90, right, scale=0.45]{$\xi_2$}
(2,170) circle (6pt) node[rotate=90, right, scale=0.45]{$\tau_1\xi_1^3$}
(3,170) circle (6pt) node[rotate=90, right, scale=0.45]{$\tau_2$}
;}

{\filldraw[weight2]
(50,0) circle (6pt) node[right, scale=0.45]{ $\lambda_1$}
;}
{\filldraw[weight4]
(50,40) circle (6pt)
(51,50) circle (6pt)
;}
{\filldraw[weight6]
(50,80) circle (6pt)
(51,90) circle (6pt)
;}
{\filldraw[weight8]
(50,120) circle (6pt)
(51,130) circle (6pt)
;}
{\filldraw[weight10]
(50,160) circle (6pt)
(51,160) circle (6pt)
(52,170) circle (6pt)
(53,170) circle (6pt)
;}

\draw[colarrow, line width=0.4pt,->](50,0)--(0,40);
\draw[colarrow, line width=0.4pt,->](50,40)--(0,80);
\draw[colarrow, line width=0.4pt,->](51,50)--(1,90);
\draw[colarrow, line width=0.4pt,->](50,80)--(0,120);
\draw[colarrow, line width=0.4pt,->](51,90)--(1,130);
\draw[colarrow, line width=0.4pt,->](50,120)--(0,160);
\draw[colarrow, line width=0.4pt,->](51,130)--(2,170);

{\filldraw[weight2]
(60,1) circle (6pt) node[right, scale=0.45]{ $\mu_1$}
;}
{\filldraw[weight4]
(60,41) circle (6pt)
(61,51) circle (6pt)
;}
{\filldraw[weight6]
(60,81) circle (6pt)
(61,91) circle (6pt)
;}
{\filldraw[weight8]
(60,121) circle (6pt)
(61,131) circle (6pt)
;}
{\filldraw[weight10]
(60,161) circle (6pt)
(61,161) circle (6pt)
(62,171) circle (6pt)
(63,171) circle (6pt)
;}

{\filldraw[weight2]
(60,1) circle (6pt) node[right, scale=0.45]{ $\mu_1$}
;}
{\filldraw[weight4]
(60,41) circle (6pt)
(61,51) circle (6pt)
;}
{\filldraw[weight6]
(60,81) circle (6pt)
(61,91) circle (6pt)
;}
{\filldraw[weight8]
(60,121) circle (6pt)
(61,131) circle (6pt)
;}
{\filldraw[weight10]
(60,161) circle (6pt)
(61,161) circle (6pt)
(62,171) circle (6pt)
(63,171) circle (6pt)
;}

{\filldraw[weight4]
(110,1) circle (6pt) node[right, scale=0.45]{ $\mu_1\lambda_1$}
;}
{\filldraw[weight6]
(110,41) circle (6pt)
(111,51) circle (6pt)
;}
{\filldraw[weight8]
(110,81) circle (6pt)
(111,91) circle (6pt)
;}
{\filldraw[weight10]
(110,121) circle (6pt)
(111,131) circle (6pt)
;}
{\filldraw[weight12]
(110,161) circle (6pt)
(111,161) circle (6pt)
(112,171) circle (6pt)
(113,171) circle (6pt)
;}

\draw[colarrow, line width=0.4pt,->](110,1)--(60,41);
\draw[colarrow, line width=0.4pt,->](110,41)--(60,81);
\draw[colarrow, line width=0.4pt,->](111,51)--(61,91);
\draw[colarrow, line width=0.4pt,->](110,81)--(60,121);
\draw[colarrow, line width=0.4pt,->](111,91)--(61,131);
\draw[colarrow, line width=0.4pt,->](110,121)--(60,161);
\draw[colarrow, line width=0.4pt,->](111,131)--(62,171);

{\filldraw[weight4]
(120,2) circle (6pt) node[right, scale=0.45]{ $\upgamma_2(\mu_1)$}
;}
{\filldraw[weight6]
(120,42) circle (6pt)
(121,52) circle (6pt)
;}
{\filldraw[weight8]
(120,82) circle (6pt)
(121,92) circle (6pt)
;}
{\filldraw[weight10]
(120,122) circle (6pt)
(121,132) circle (6pt)
;}
{\filldraw[weight12]
(120,162) circle (6pt)
(121,162) circle (6pt)
(122,172) circle (6pt)
(123,172) circle (6pt)
;}

{\filldraw[weight6]
(170,2) circle (6pt) node[right, scale=0.45]{ $\upgamma_2(\mu_1)\lambda_1$}
;}
{\filldraw[weight8]
(170,42) circle (6pt)
(171,52) circle (6pt)
;}
{\filldraw[weight10]
(170,82) circle (6pt)
(171,92) circle (6pt)
;}
{\filldraw[weight12]
(170,122) circle (6pt)
(171,132) circle (6pt)
;}
{\filldraw[weight14]
(170,162) circle (6pt)
(171,162) circle (6pt)
(172,172) circle (6pt)
(173,172) circle (6pt)
;}

\draw[colarrow, line width=0.4pt,->](170,2)--(120,42);
\draw[colarrow, line width=0.4pt,->](170,42)--(120,82);
\draw[colarrow, line width=0.4pt,->](171,52)--(121,92);
\draw[colarrow, line width=0.4pt,->](170,82)--(120,122);
\draw[colarrow, line width=0.4pt,->](171,92)--(121,132);
\draw[colarrow, line width=0.4pt,->](170,122)--(120,162);
\draw[colarrow, line width=0.4pt,->](171,132)--(122,172);

{\filldraw[weight8]
(170,3) circle (6pt) node[right, scale=0.45]{ $\lambda_2$}
;}
{\filldraw[weight10]
(170,43) circle (6pt)
(171,53) circle (6pt)
;}
{\filldraw[weight12]
(170,83) circle (6pt)
(171,93) circle (6pt)
;}
{\filldraw[weight14]
(170,123) circle (6pt)
(171,133) circle (6pt)
;}
{\filldraw[weight16]
(170,163) circle (6pt)
(171,163) circle (6pt)
(172,173) circle (6pt)
(173,173) circle (6pt)
;}

{\filldraw[weight6]
(180,3) circle (6pt) node[right, scale=0.45]{ $\upgamma_3(\mu_1)$}
;}
{\filldraw[weight8]
(180,43) circle (6pt)
(181,53) circle (6pt)
;}
{\filldraw[weight10]
(180,83) circle (6pt)
(181,93) circle (6pt)
;}
{\filldraw[weight12]
(180,123) circle (6pt)
(181,133) circle (6pt)
;}
{\filldraw[weight14]
(180,163) circle (6pt)
(181,163) circle (6pt)
(182,173) circle (6pt)
(183,173) circle (6pt)
;}

{\filldraw[weight8]
(180,4) circle (6pt) node[right, scale=0.45]{ $\mu_2$}
;}
{\filldraw[weight10]
(180,44) circle (6pt)
(181,54) circle (6pt)
;}
{\filldraw[weight12]
(180,84) circle (6pt)
(181,94) circle (6pt)
;}
{\filldraw[weight14]
(180,124) circle (6pt)
(181,134) circle (6pt)
;}
{\filldraw[weight16]
(180,164) circle (6pt)
(181,164) circle (6pt)
(182,174) circle (6pt)
(183,174) circle (6pt)
;}

\end{tikzpicture}}
\caption{The $E^5$ page of the spectral sequence for $p=3$. Some arrows, exceeding the boundaries of the figure, are not displayed.\\
{\color{weight0} weight 0},
{\color{weight2} weight 2},
{\color{weight4} weight 4},
{\color{weight6} weight 6},
{\color{weight8} weight 8},
{\color{weight10} weight 10},
{\color{weight12} weight 12},
{\color{weight14} weight 14},
{\color{weight16} weight 16},
}
\label{fig:E5}
\end{figure}
\begin{figure}
\centering
\adjustbox{max width= \textwidth, center}{
\pgfsetshortenend{0.6pt}
\pgfsetshortenstart{0.6pt}
\begin{tikzpicture}[scale=2\textwidth/190cm,line width=1pt]

\draw[step=1cm , almostwhite, very thin] (0,0) grid (189.5,189.5);
\draw[step=10cm ,gray,very thin] (0,0) grid (189.5,189.5);
\draw[gray, line width=0.35pt,->](0,0)--(190,0);
\draw[gray, line width=0.35pt,->] (0,0)--(0,190);

\node foreach \y in {1,...,18} at (0, 10*\y) [left, scale=0.45] {\y};
\node foreach \x in {1,...,18} at (10*\x, 0) [below, scale=0.45] {\x};

{\filldraw[weight0]
(0,0) circle (6pt) node[below left, scale=0.45]{ 1}
;}
{\filldraw[weight2]
(1,50) circle (6pt) node[rotate=90, right, scale=0.45]{$\tau_1$}
;}
{\filldraw[weight8]
(1,160) circle (6pt) node[rotate=90, right, scale=0.45]{$\xi_2$}
(3,170) circle (6pt) node[rotate=90, right, scale=0.45]{$\tau_2$}
;}

{\filldraw[weight2]
(60,1) circle (6pt) node[right, scale=0.45]{ $\mu_1$}
;}
{\filldraw[weight4]
(61,51) circle (6pt)
;}
{\filldraw[weight10]
(61,161) circle (6pt)
(63,171) circle (6pt)
;}

\draw[colarrow, line width=0.4pt,->](60,1)--(1,50);

{\filldraw[weight4]
(120,2) circle (6pt) node[right, scale=0.45]{ $\upgamma_2(\mu_1)$}
;}
{\filldraw[weight6]
(121,52) circle (6pt)
;}
{\filldraw[weight12]
(121,162) circle (6pt)
(123,172) circle (6pt)
;}

\draw[colarrow, line width=0.4pt,->](120,2)--(61,51);

{\filldraw[weight8]
(170,3) circle (6pt) node[right, scale=0.45]{ $\lambda_2$}
;}
{\filldraw[weight10]
(171,53) circle (6pt)
;}
{\filldraw[weight16]
(171,163) circle (6pt)
(173,173) circle (6pt)
;}

{\filldraw[weight6]
(180,3) circle (6pt) node[right, scale=0.45]{ $\upgamma_3(\mu_1)$}
;}
{\filldraw[weight8]
(181,53) circle (6pt)
;}
{\filldraw[weight14]
(181,163) circle (6pt)
(183,173) circle (6pt)
;}

\draw[colarrow, line width=0.4pt,->](180,3)--(121,52);

{\filldraw[weight8]
(180,4) circle (6pt) node[right, scale=0.45]{ $\mu_2$}
;}
{\filldraw[weight10]
(181,54) circle (6pt)
;}
{\filldraw[weight16]
(181,164) circle (6pt)
(183,174) circle (6pt)
;}

\end{tikzpicture}}
\caption{The $E^6$ page of the spectral sequence for $p=3$. Some arrows, exceeding the boundaries of the figure, are not displayed.\\
{\color{weight0} weight 0},
{\color{weight2} weight 2},
{\color{weight4} weight 4},
{\color{weight6} weight 6},
{\color{weight8} weight 8},
{\color{weight10} weight 10},
{\color{weight12} weight 12},
{\color{weight14} weight 14},
{\color{weight16} weight 16},
}
\label{fig:E6}
\end{figure}
\relax



\clearpage

\section{\texorpdfstring{$MHH(M\mathbb{Z}/p)/\tau^{p-1}$ for odd $p$}{MHH(MZ/p)/tau\^{}(p-1) for odd p}}
\subsection{General considerations}

The induction argument for odd primes proceeds much similarly to the already studied $p=2$. The main difference here is that each generator- each point in the diagrams- does not represent a homogeneous module isomorphic to $\mathbb{F}_2$, the field with two elements, but a module isomorphic to $\mathbb{F}_p[\tau]/\tau^{p-1}$, as appears from the spectral sequence description. Observe in particular that this ring is not homogeneous in weight.

Nonetheless, we can adopt a similar strategy in this situation as well and prove:
\begin{thm}\label{thm:MHH(MZ/p)/tau^p-1}
    There is an isomorphism of graded rings:
    \[
    \pi_{*,*}(MHH(M\mathbb{Z}/p)/\tau^{p-1}) \cong
     \bigotimes_{h\in \mathbb{N}} \left ( \Gamma_{\mathbb{F}_p[\tau]/\tau^{p-1}}(\mu_h)  \otimes_{\mathbb{F}_p[\tau]/\tau^{p-1}}  \Lambda_{\mathbb{F}_p[\tau]/\tau^{p-1}}(\lambda_{h+1}) \right )
    \]
    where 
    \begin{gather*}
        \Gamma_{\mathbb{F}_p[\tau]/\tau^{p-1}}(\mu_h) \cong \mathbb{F}_p[\tau, \upgamma_{p^i}(\mu_h)]/\langle \tau^{p-1}, \upgamma_{p^i}(\mu_h)^p \rangle\\
        \Lambda_{\mathbb{F}_p[\tau]/\tau^{p-1}}(\lambda_h) \cong \mathbb{F}_p[\tau, \lambda_h]/\langle \tau^{p-1}, \lambda_h^2 \rangle
    \end{gather*}
    and the tensor product is taken in $\mathbb{F}_p[\tau]/\tau^{p-1}$-algebras.
    The degrees are:
    \begin{align*}
        |\tau|=(0,-1) && |\upgamma_{p^i}(\mu_h)|=(2p^{h+i}, p^i(p^h -1)) && |\lambda_h|=(2p^h-1, p^h-1).
    \end{align*}    
\end{thm}

The argument proceeds by induction on the number of the page; as for the case $p=2$, the key step is to prove that each page of the spectral sequence decomposes as a $\mathbb{F}_p[\tau]/\tau^{p-1}$-tensor product of two graded rings: one concentrated on the horizontal axis and one on the vertical axis. The proof articulates as follows. On most pages, we will see that nothing happens. On the other hand, at some specific pages (namely, those of index $2p^i$ or $2p^{i+1}-1$ for some $i \geq 0$) we will identify an algebra generator (or a family of generators) of $\pi_{*,*}(MHH(\mathbb{F}_p[\tau])/\tau^{p-1})$, that survive to that page and generate all non-trivial differentials. At the same time, we can associate to the image of such differentials a generator of $\pi_{*,*}(\mathcal{A}(p)/\tau^{p-1})$, so that the next page still decomposes as a $\mathbb{F}_p[\tau]/\tau^{p-1}$-tensor product of a quotient ring of $\pi_{*,*}(MHH(\mathbb{F}_p[\tau])/\tau^{p-1})$ and a quotient ring of $\pi_{*,*}(\mathcal{A}(p)/\tau^{p-1})$. The process is explicit and allows us to identify inductively all the elements $\pi_{*,*}(MHH(\mathbb{F}_p[\tau])/\tau^{p-1})$. We begin by looking at low-index pages to get an idea of what the spectral sequence looks like and to provide a base step for the induction argument, formulated in full precision in a \hyperref[subs:induction.step.p]{later subsection}.

The first result needed is the following lemma dealing with degrees and weights of monomials in the motivic mod $p$ dual Steenrod algebra. We recall that:
\[
    \pi_{*,*}(\mathcal{A}(p)/\tau^{p-1}) \cong \mathbb{F}_p[\tau, \tau_i, \xi_{i+1}]_{i \geq 0}/\langle \tau^{p-1}, \tau_i^2 \rangle
\]
for odd primes $p$; the degrees are:
\begin{align*}
    |\tau|=(0,-1) && |\tau_i|=(2p^i-1,p^i-1) && |\xi_i|=(2p^i-2,p^i-1).
\end{align*}

\begin{lemma} \label{lemma:estimate.on.monomials}
    Consider a monomial:
    \[
        x=\tau^a \tau_{i_1}^{b_1}  \tau_{i_2}^{b_2}  \cdots\tau_{i_n}^{b_n} \xi_{i_1}^{c_1} \xi_{i_2}^{c_2} \cdots \xi_{i_n}^{c_n} \in \pi_{*,*}(\mathcal{A}(p)/\tau^{p-1}),
    \]
    with $i_1<i_2<\ldots<i_n$, $a,\, b_j,\, c_j \geq 0$ and either $b_1\neq 0$ or $c_1 \neq 0$. Let $|x|=(d,w)$; if:
    \[
        Q(i)=\frac{1}{2}-\frac{p-2}{2p^{i}-2}-\frac{1}{4p^{i}-2}
    \]
    then:
    \[
        Q(i_1) \leq \frac{w}{d} \leq \frac{1}{2}
    \]
\end{lemma}
Observe that any non-zero monomial in $\pi_{*,*}(\mathcal{A}(p)/\tau^{p-1})$ can be expressed in this form. This will provide an analogue to lemma \ref{lmm:d/w.for.E.2.j+1} at odd primes.
\begin{proof}
    The double degree of $x$ is:
    \begin{align*}
        |x| &= \left (0 + \sum_{j=1}^n \left[b_j(2p^{i_j}-1)+c_j(2p^{i_j}-2)\right] , -a + \sum_{j=1}^n (b_j+c_j)(p^{i_j}-1)\right ) \\
        &= \left (\sum_{j=1}^n \left[2(b_j+c_j)(p^{i_j}-1)+b_j\right] , -a + \sum_{j=1}^n (b_j+c_j)(p^{i_j}-1)\right )
    \end{align*}
    As $b_1\neq 0$ or $c_1 \neq 0$, the degree is non-zero, so it makes sense to evaluate:
    \[
        \frac{w}{d}= \frac{-a}{d}+\frac{\sum_{j=1}^n (b_j+c_j)(p^{i_j}-1)}{d}
    \]
    Now, $-p+2 \leq a \leq 0$ as $\tau^{p-1}=0$ and the minimum degree is $d \geq 2p^{i_1}-2$. So:
    \begin{equation}\label{eq:-a/d}
        \frac{-p+2}{2p^{i_1}-2} \leq \frac{-a}{d} \leq 0
    \end{equation}
        
    To evaluate the fraction:
    \[
        \frac{\sum_{j=1}^n (b_j+c_j)(p^{i_j}-1)}{\sum_{j=1}^n \left[2(b_j+c_j)(p^{i_j}-1)+b_j\right]}
    \]
    we make a homogeneous estimate on each fraction:
    \[
        \left(\frac{(b_j+c_j)(p^{i_j}-1)}{2(b_j+c_j)(p^{i_j}-1)+b_j}\right)^{-1}=\frac{2(b_j+c_j)(p^{i_j}-1)+b_j}{(b_j+c_j)(p^{i_j}-1)},
    \]
    obviously restricting to those $j$ for which either $b_j \neq 0$ or $c_j \neq 0$, so that the denominator is non-zero. Observe that this is equivalent to asking that the numerator is non-zero. As $b_j \geq 0$, there is an immediate lower bound:
    \[
        \frac{2(b_j+c_j)(p^{i_j}-1)+b_j}{(b_j+c_j)(p^{i_j}-1)} \geq 2
    \]
    For the lower bound, notice that since $\tau_j^2=0$, each $b_j$ is either 0 or 1. For $b_j=0$, 
    \[
        \frac{2+c_j(p^{i_j}-1)}{c_j(p^{i_j}-1)} = 2
    \]
    equals the lower bound (all the elements $\xi_j$ have $d/w=2$). For $b_j=1$:
    \[
        \frac{2(1+c_j)(p^{i_j}-1)+1}{(1+c_j)(p^{i_j}-1)} = 2+ \frac{1}{(1+c_j)(p^{i_j}-1)}
    \]
    As $c_j \geq 0$:
    \[
        \frac{1}{(1+c_j)(p^{i_j}-1)} \leq \frac{1}{p^{i_j}-1} \leq \frac{1}{p^{i_1}-1}
    \]
    We obtain the homogeneous upper bound:
    \[
        \frac{2(1+c_j)(p^{i_j}-1)+1}{(1+c_j)(p^{i_j}-1)} \leq 2+ \frac{1}{p^{i_1}-1}=\frac{2p^{i_1}-1}{p^{i_1}-1}
    \]
    Hence:
    \[
        \frac{1}{2}-\frac{1}{4p^{i_1}-2}=\frac{p^{i_1}-1}{2p^{i_1}-1} \leq \frac{\sum_{j=1}^n (b_j+c_j)(p^{i_j}-1)}{\sum_{j=1}^n \left[2(b_j+c_j)(p^{i_j}-1)+b_j\right]} \leq \frac{1}{2}
    \]
    Summing this and \ref{eq:-a/d} produces the thesis.    
\end{proof} 

We find a rougher estimate that will be more useful for computations.
\begin{cor} \label{cor:rough.est.Q(i)}
    $Q(i)\geq \frac{2p^i-2p+1}{4p^i-4}$
\end{cor}
\begin{proof}
    \[
    \begin{aligned}
        Q(i)&=\frac{1}{2}-\frac{p-2}{2p^{i}-2}-\frac{1}{4p^{i}-2} \geq \frac{1}{2}-\frac{p-2}{2p^{i}-2}-\frac{1}{4p^{i}-4}\\
        & = \frac{1}{2}-\frac{2p-4}{4p^{i}-4}-\frac{1}{4p^{i}-4}=\frac{1}{2}-\frac{2p-3}{4p^{i}-4}\\
        & = \frac{2p^{i}-2-2p+3}{4p^{i}-4}=\frac{2p^i-2p+1}{4p^i-4}.
    \end{aligned}
    \]
\end{proof}


\begin{rmk}
    As we are working with differential graded algebras in odd characteristics, for any two elements $a$ and $b$, $d(ab)$ might be different from $d(ba)$, as $ab$ in general is, in general, different from $ba$. However, to relieve the reader from a large number of minus signs, we will safely ignore this whenever possible; all equalities have then to be considered true up to a sign. 
    
    As we will see, this sign coming from graded commutativity plays a very small role in our spectral sequence: it actively influences the behaviour of a differential only when dealing with the square of the elements $\lambda_i$.
\end{rmk}
\subsection{\texorpdfstring{The $E^2$ page}{The E2 page}}\label{sec:E2}

The starting point is the $E^2$ page, see figure \ref{fig:E2}. First of all, as the spectral sequence is first-quadrant, the convergence constraint imposes $\pi_{1,*}(MHH(M\mathbb{Z}/p)/\tau^{p-1}) \cong E^2_{1,0,*} \cong 0$.

The presence of $\tau_0 \in E^2_{0,1,0}$ and the convergence constraint imply the presence of an element $\mu_0 \in \pi_{2,0}(MHH(M\mathbb{Z}/p)/\tau^{p-1})$ with $d^2(\mu_0)= \tau_0$. Notice that, as $d^2(\tau)=0$, $d^2(\tau^i \mu_0)=\tau^i \tau_0$, which in particular implies that $\tau^i \mu_0 \neq 0$ for $0 \leq i \leq p-2$.

Consider now the element $\tau_0 \mu_0 \in E^2_{2,1,0}$. It cannot have non-trivial differentials; two reasons apply in this case. First and foremost, one could argue this just by looking at the distribution of nontrivial modules in the $E^2$ page. The second argument, which will prove useful also for the next steps, consists in noticing that all the generators in $\pi_{*,*}(\mathcal{A}(p)/\tau^{p-1})$ with degree larger than 1 have weight at least $p-1$; so elements in $\pi_{*,*}(\mathcal{A}(p)/\tau^{p-1})$ with degree larger than 1 have weight at least $1$. Hence an element with weight 0, as $\tau_0 \mu_0$, cannot differentiate to them. This implies that there exists an element $\upgamma_2 (\mu_0) \in \pi_{4,0}(MHH(M\mathbb{Z}/p)/\tau^{p-1}) \cong E^2_{4,0,0}$ with $d^2(\upgamma_2 (\mu_0)) =\tau_0 \mu_0$. Observe that also the square $\mu_0^2$ sits in $E^2_{4,0,0}$; the Leibniz rule tells that $d^2(\mu_0^2)=2\tau_0 \mu_0$, so we can identify $\mu_0^2=2\upgamma_2 (\mu_0)$.

If we apply the Leibniz rule to the elements of the form $\mu_0 \beta$, for $\beta \in E^2_{0,*,*}$, we notice that whenever $\beta$ is not a multiple of $\tau_0$, $\mu_0 \beta$ supports a non-trivial $d^2$-differential. Moreover, these $d^2$ differentials surject onto the module $\langle \tau_0 \rangle \subset \pi_{*,*}(\mathcal{A}(p)/\tau^{p-1})$. In particular, there are no elements of negative weight in $E^{3}_{0,*,*}$ and in the higher pages. Similarly, the $d^2$ differentials out of $\upgamma_2(\mu_0) \beta$, for $\beta \in E^2_{0,*,*}$, surject onto the module $\mu_0\tau_0 E^2_{0,*,*} \subset E^2_{2,*,*}$. Hence the whole $E^2_{2,*,*}$ column disappears from the $E^3$ page, and in particular, there are no elements of negative weight in the columns $E^3_{i,*,*}$ for $0 \leq i \leq 2$.  

We proceed now with an induction argument that uncovers all the divided powers of $\mu_0$ and identifies them as the source of all the $d^2$ differentials. At step $f-1$, we suppose to know the existence of elements $\upgamma_1(\mu_0), \ldots, \upgamma_{f-1}(\mu_0)$ behaving as divided powers in characteristic $p$ up to degree $2(f-1)$, meaning that if $p^k$ is the largest power of $p$ smaller than or equal to $f-1$, there is a homomorphism:
\begin{multline*}
    \mathbb{F}_p[\tau, \upgamma_{p^0}(\mu_0), \ldots, \upgamma_{p^k}(\mu_0)]/\langle \tau^{p-1},  \upgamma_{p^0}(\mu_0)^p, \ldots, \upgamma_{p^{k-1}}(\mu_0)^p\rangle\\
    \rightarrow \pi_{*,*}(MHH(M\mathbb{Z}/p)/\tau^{p-1})
\end{multline*}
which is an injection up to degree $2(f-1)$. The divided powers not indexed by powers of $p$ are obtained by the modulo $p$ decomposition of the index. Moreover, we know that any element in $E^2_{*,0,*} \cong \pi_{*,*}(MHH(M\mathbb{Z}/p)/\tau^{p-1})$ with degree smaller then or equal to $f-1$ is a linear combination:
\[
    \sum_{i=0}^{f-1} \beta_i \upgamma_{i}(\mu_0)
\]
with $d^2(\beta_i)=0$ and $\sum_{i=0}^{f-1} \beta_i \upgamma_{i}(\mu_0)=0$ if and only if $\beta_i=0$ for all $i$. Any multiple of $\tau_0$ in the columns $E^2_{i,*,*}$ for $0 \leq i \leq 2(f-2)$ is hit by a non-trivial differential, and any non-zero multiple of a $\upgamma_i(\mu_0)$ in the same region either supports or receives a non trivial $d^2$-differential. In particular, all elements in this region of the $E^3$ page, but in degree $(0,0,*)$, have positive weight.

The next step of the induction procedure follows pretty straightforwardly. $\upgamma_{f-1}(\mu_0)\tau_0$ has trivial $d^2$ differential and weight zero; as, by inductive hypothesis, no element in the $E^3$-page (hence in all the higher pages) in the region where a differential from this element would land has weight zero, there must be an element $\upgamma_{f}(\mu_0)\in E^2_{2f,0,0}$ supporting a differential $d^2(\upgamma_{f}(\mu_0))=\upgamma_{f-1}(\mu_0)\tau_0$, which can be identified with a product of the previous divided powers by the Leibniz rule whenever $f$ is not a power of $p$. The bound on the weight implies that, if $f$ is a power of $p$, $\upgamma_{f-1}(\mu_0)^p=0$, as it has trivial $d^2$ differential by the Leibniz rule ($\upgamma_{f-1}(\mu_0)$ has even degree) and cannot support longer differentials because of the weight. 

Next, given $\alpha \in E^2_{*,0,*}$ in degree smaller than or equal to $2f$, we can express its image via the $d^2$ differential as:
\[
    d^2(\alpha)=\left (\sum_{i=0}^{f-1} \beta_i \upgamma_{i}(\mu_0) \right) \tau_0
\]
with $d^2(\beta_i)=0$. Then:
\[
    d^2 \left (\sum_{i=0}^{f-1} \beta_i \upgamma_{i+1}(\mu_0) \right )=\left (\sum_{i=0}^{f-1} \beta_i \upgamma_{i}(\mu_0) \right) \tau_0= d^2(\alpha)
\]
so $\alpha- \sum_{i=0}^{f-1} \beta_i \upgamma_{i+1}(\mu_0)$ is a $d^2$ cycle. Consider now a sum $\sum_{i=0}^{f} \beta_i \upgamma_{i+1}(\mu_0)$, with the $\beta_i$ $d^2$ cycles. First, by a similar argument as the previous one, this sum is trivial if and only if all the $\beta_i$ are. Next, observe that the $\beta_i$, which must support a non-trivial differential, must have positive weight by the induction hypothesis. Extending linearly the differential implies that all multiples of the $\upgamma_i(\mu_0)$ and of $\tau_0$ in $E^2_{i,*,*}$ with $i \leq 2(f-1)$ vanish from the $E^3$-page; in particular, there are no elements of weight smaller than $1$ in the $E^3$ page in this region (but in degree $(0,0,*)$). This ends the induction step.

Thus the following hold:
\[
\begin{tikzcd}
    E^3_{*,0,*} 
        \arrow[r, hook]
        \arrow[rd, "\sim"]&
    E^2_{*,0,*}
        \arrow[d, two heads] \\
    &
    E^2_{*,0,*}/\langle \upgamma_i (\mu_0) \rangle \cong \pi_{*,*}(MHH(M\mathbb{Z}/p)/\tau^{p-1}) /\langle \upgamma_i (\mu_0) \rangle_{i \geq 0}
\end{tikzcd}
\]
\[
    E^3_{0,*,*} \cong E^2_{0,*,*} /\langle \tau_0 \rangle \cong \pi_{*,*}(\mathcal{A}(p)/\tau^{p-1}) /\langle \tau_0 \rangle
\]
\[
    E^3_{\star, \bullet, *} \cong E^3_{\star,0,*}\otimes_{\mathbb{F}_{p}[\tau]/\tau^{p-1}} E^3_{0,\bullet,*}
\]

In particular, we can identify a subalgebra:
\[
    \Gamma_{\mathbb{F}_p[\tau]/\tau^{p-1}}(\mu_0) \hookrightarrow \pi_{*,*}(MHH(M\mathbb{Z}/p)/\tau^{p-1})
\]
such that the inclusion is an isomorphism in degrees smaller than or equal to 2.

As the non-trivial class with the smallest positive degree in $E^3_{0,*,*}$ is $\xi_1 \in E^3_{0,2p-2,p-1}$, there can be no non-trivial differentials of length 3 coming out of the horizontal line $E^3_{*,0,*}$. As the $E^3$ page decomposes as a tensor product, we conclude that there are no non-trivial differentials of length 3 in general, and we have a graded algebra isomorphism $E^3\cong E^4$. The same applies to differentials of length from 4 up to $2p-2$, so we conclude:
\[
    E^3 \cong E^4 \cong\ldots \cong E^{2p-1}
\]
Observe that, since the spectral sequence is first quadrant, this implies that the map:
\[
    \Gamma_{\mathbb{F}_p[\tau]/\tau^{p-1}}(\mu_0)  \hookrightarrow \pi_{*,*}(MHH(M\mathbb{Z}/p)/\tau^{p-1})
\]
is an isomorphism in degree smaller then or equal to $2p-2$, and that the columns $E^{2p-1}_{i,*,*}$ vanish for $1 \leq i \leq 2p-2$.

Observe also that, as an additional conclusion, from the $E^3$ page onwards we know that there are no elements of negative weight in the spectral sequence but in degree $(0,0,*)$.

\subsection{\texorpdfstring{The $E^{2p-1}$ page}{The E\^(2p-1) page}}\label{sec:E^{2p-1}}

The next non-trivial differentials have length $2p-1$. In fact, the presence of $\xi_1 \in E^{2p-1}_{0,2p-2,p-1}$ and the convergence conditions for the spectral sequence imply the existence of an element $\lambda_1 \in E^{2p-1}_{2p-1,0,p-1}$ with $d^{2p-1}(\lambda_1)=\xi_1$. By the Leibniz rule, for any $x \in E^{2p-1}_{0,*,*}$, $d^{2p-1}(\lambda_1 x)=\xi_1 x$; these differentials are all non trivial, for non-trivial $x$, as $\xi_1 \in \pi_{*,*}(\mathcal{A}(p)/\tau^{p-1})$ is not involved in any non-trivial relation. The $d^{2p-1}$ differential produces in particular an isomorphism:
\[
    \mathbb{F}_p[\tau]/\tau^{p-1}\{\lambda_1\} \xrightarrow{d^{2p-1}} \mathbb{F}_p[\tau]/\tau^{p-1}\{\xi_1\} \cong E^{2p-1}_{0,2p-2,*}
\]
Hence $E^{2p-1}_{2p-1,0,*} \cong \mathbb{F}_p[\tau]/\tau^{p-1}\{\lambda_1\}$. Observe then that the whole column $E^{2p-1}_{2p-1,*,*}$ supports non trivial $d^{2p-1}$-differentials, hence $E^{2p}_{2p-1,*,*}\cong 0$.

Consider now the square $\lambda_1^2 \in E^{2p-1}_{2(2p-1),0,2(p-1)}$. Its $d^{2p-1}$ differential is trivial because of the Leibniz rule ($\lambda_1$ has odd degree). If $\lambda_1^2 \neq 0$, it survives to a higher page, and, due to the convergence constraints, it has to support a non-trivial differential. This cannot happen: the classes with the smallest degrees in $E^{2p-1}_{0,*,*}$ surviving to a higher page are:
\begin{gather*}
    1 \in E^{2p-1}_{0,0,0} \\
    \tau_1 \in E^{2p-1}_{0,2p-1,p-1} \\
    \xi_2 \in E^{2p-1}_{0,2p^2-2,p^2-1}
\end{gather*}
Columns $E^{2p-1}_{1,*,*}$ to $E^{2p-1}_{2p-2,*,*}$ are empty by the induction hypothesis, so the only possibility to check is $\lambda_1^2$ supporting a transgressive differential. As $2p-1<2(2p-1)-1 < 2p^2-2$ for all $p>1+\sqrt{2}/2$, this case is also ruled out.

Now, observe that any element $x \in E^{2p-1}_{*,0,*}$ can be written as $x=a\lambda_1+b$ with $d^{2p-1}(a)=d^{2p-1}(b)=0$. In fact, let $d^{2p-1}(x)=a \xi_1$; as:
\[
    0=d^{2p-1}(d^{2p-1}(x))=d^{2p-1}(a \xi_1)=d^{2p-1}(a)\xi_1
\]
by the tensor product decomposition of the $E^{2p-1}$ page and the structure of $E^{2p-1}_{0,*,*}$ it must be $d^{2p-1}(a)=0$. Consider now the element $\lambda_1 a$; one has:
\[
    d^{2p-1}(a \lambda_1)=a \xi_1 = d^{2p-1}(x)
\]
Hence $x-a\lambda_1$ is a $d^{2p-1}$-cycle. We conclude that the $d^{2p-1}$-differential is non-trivial on all the multiples of $\lambda_1$ in the $E^{2p-1}$ page by linearity; moreover its image is always a multiple of $\xi_1$, and actually the only multiples of $\xi_1$ that are not in the image of the $d^{2p-1}$-differential are the multiples of $\lambda_1\xi_1$ (which on the other hand support a non-trivial $d^{2p-1}$).

More accurately, there are commutative diagrams:
\[
\begin{tikzcd}
    E^{2p}_{*,0,*} 
        \arrow[r, hook]
        \arrow[rd, "\sim"]&
    E^{2p-1}_{*,0,*}
        \arrow[d, two heads] \\
    &
    E^{2p-1}_{*,0,*}/\langle \lambda_1 \rangle 
\end{tikzcd}
\]
\[
    E^{2p}_{*,0,*}\cong \pi_{*,*}(MHH(M\mathbb{Z}/p)/\tau^{p-1}) /\langle \upgamma_i (\mu_0), \lambda_1 \rangle_{i \geq 0}
\]
\[
    E^{2p}_{0,*,*} \cong E^{2p-1}_{0,*,*} /\langle \xi_1 \rangle \cong \pi_{*,*}(\mathcal{A}(p)/\tau^{p-1}) /\langle \tau_0, \xi_1 \rangle
\]
\[
    E^{2p}_{\star, \bullet, *} \cong E^{2p}_{\star,*}\otimes_{\mathbb{F}_{p,\tau}} E^{2p}_{\bullet,*}
\]

In particular, we can identify a subalgebra:
\[
   \Gamma_{\mathbb{F}_p[\tau]/\tau^{p-1}}(\mu_0)  \otimes_{\mathbb{F}_p[\tau]/\tau^{p-1}}  \Lambda_{\mathbb{F}_p[\tau]/\tau^{p-1}}(\lambda_1) \hookrightarrow \pi_{*,*}(MHH(M\mathbb{Z}/p)/\tau^{p-1})
\]
such that the inclusion is an isomorphism in degrees smaller than or equal to $2p-1$.

Observe that, by the above conclusions, the rows $E^{2p}_{*,i,*}$ for $1 \leq i \leq 2p-2$ and the columns $E^{2p}_{j,*,*}$ for $1 \leq j \leq 2p-1$ are empty.
\subsection{\texorpdfstring{The $E^{2p}$ page}{The E\^(2p) page}} \label{sec:E^{2p}}

This section is dedicated to the study of the $E^{2p}$ page; this is in fact just a special case of what is treated in subsection \ref{sec:E^{2p^j}}. It is exposed independently because weighting arguments are simplified by the absence of the varying exponent and the degrees are small enough (at least when $p=3$) to allow visualisation of the page in small degrees by direct computation, see figure \ref{fig:E6}.

The following proof will show a behaviour very similar to that of the $E^2$ page: there is an infinite family of elements of $\pi_{*,*}(MHH(M\mathbb{Z}/p)/\tau^{p-1})$, organised in a divided power structure, call them $\upgamma_i (\mu_1)$, that is the only source of all the differentials of length $2p$. The independence this family and other elements in motivic Hochschild homology then implies that they are everything that disappears from $E^{2p}_{*,0,*}$ passing to the next page; symmetrically, so does the ideal $\langle \tau_1 \rangle \in E^{2p}_{0,*,*}$, which lies in the image of the $d^{2p}$-differential. This produces a tensor product decomposition in the $E^{2p+1}$-page as well. The major difference with the $E^2$-page, and the major difficulty to overcome, is that the $\upgamma_i (\mu_1)$ do not share all the same weight; in fact we will see:
\[
    |\upgamma_i (\mu_1)|=(i \cdot 2p, i\cdot (p-1)).
\]
As $i$ increases, elements like $\upgamma_i (\mu_1)$ or $\upgamma_i (\mu_1) \tau_1$ could potentially be the source of very long differentials. A straightforward weight argument is not enough to exclude this possibility this time, so one of the main concerns of the following proof is to solve this issue.

The argument proceeds essentially by induction on $s$, deducing at each stage the existence and properties of $\upgamma_s(\mu_1)$ and establishing what happens in the columns $E^{2p}_{i,*,*}$ for $i<2p(s-1)$. More precisely, for each integer $s\geq 1$ we formulate four statements:
\begin{enumerate}[start=1, label={$\arabic*_s.$}]
    \item There exist elements $\mu_1=\upgamma_1 (\mu_1),\, \upgamma_2 (\mu_1), \ldots \upgamma_s (\mu_1)$, with $\upgamma_i (\mu_1) \in E^{2p}_{(2pi, 0, i(p-1))}$, respecting all the multiplicative relations of divided powers in characteristic $p$ that are expressed by a homogeneous equation of degree smaller than or equal to $2ps$. They support $d^{2p}$-differentials $d^{2p}(\upgamma_i (\mu_1))=\upgamma_{i-1} (\mu_1) \tau_1$. Recall that we use the convention $\upgamma_0 (\mu_1)=1$.
    
    \item The identity $(\upgamma_i (\mu_1))^p=0$ is known to hold for all integers $i \leq s$ but the largest power of $p$. 
    
    \item Up to degree $2ps$, every element in $E^{2p}_{*,0,*}$ can be written as $\sum_{i=0}^s \alpha_i  \upgamma_i(\mu_1)$, with $d^{2p}(\alpha_i)=0$ and $\upgamma_0(\mu_1)=1$. In particular, every non-trivial $d^{2p}$ differential originates in an element of the form $\sum_{i=1}^s \alpha_i  \upgamma_i(\mu_1)$, with $d^{2p}(\alpha_i)=0$.\\
    Moreover $\sum_{i=0}^s \alpha_i  \upgamma_i(\mu_1)=0$ if and only if $\alpha_i=0$ for all $i$.
    
    \item Every row starting with a multiple of $\tau_1$ vanishes after the $E^{2p}$-page, up to degree $j \leq 2p(s-1)$. In particular, all such elements are targets of non-trivial $d^{2p}$-differentials.
\end{enumerate}




The proof is divided into three parts.

\begin{enumerate}[label={\Alph*.}]
    \item We first show that, for each $s \geq 1$,
    \[
        1_s, 3_{s-1}\Rightarrow 3_s \Rightarrow 4_s. 
    \]
    
    \item Then we show that it is enough to prove statements $1$ and $2$ on the powers of $p$, namely the implication:
    \[
        1_1,1_p,\ldots, 1_{p^j},2_1,2_p,\ldots, 2_{p^j} \Rightarrow 1_s,2_s \text{ for all } s \leq p^{j+1}-1.
    \]
    \item Finally, we prove $1_{p^{j}}$ and $2_{p^{j}}$ with an induction argument; the base step consists of showing $1_1$ and $2_1$(which is a trivial statement), while the induction one is:
    \[
        1_{p^{j}-1},2_{p^{j}-1},3_{p^{j}-1},4_{p^{j}-1} \Rightarrow 1_{p^{j}},2_{p^{j}} \text{ for all } j\geq 1.
    \]

\end{enumerate}

The first implication roughly follows from the fact that $E^{2p}_{0,2p-1,*}$ is a free $\mathbb{F}_p[\tau]/(\tau^{p-1})$-module on a single generator. In fact, suppose we have proven statements $1_s$ and $3_{s-1}$. To prove $3_s$, consider an element $x \in E^{2p}_{j,0,*}$ for some $j\leq s$ such that:
\[
    d^{2p}(x)=y\tau_1 \text{ with } y \in E^{2p}_{j-2p,0,*}.
\]
As $j-2p \leq 2p(s-1)$ we can apply statement $3_{s-1}$ and write:
\[
    y=\sum_{i=0}^{s-1} \alpha_i \cdot \upgamma_i(\mu_1) \text{ with } d^{2p}(\alpha_i)=0 \text{ and } \upgamma_0 (\mu_1) = 1.
\]
Observe that, for $s=1$, $y$ is an element in the zeroth column: this statement reduces to the observation that its $d^{2p}$-differential is trivial.
Now consider the element:
\[
    \tilde{x}= \sum_{i=0}^{s-1} \alpha_i \cdot \upgamma_{i+1}(\mu_1)
\]
Statement $1_s$ ensures the existence of such $\upgamma_{i+1}(\mu_1)$. Moreover, by the same hypothesis:
\[
    d^{2p}(\tilde{x})=\sum_{i=0}^{s-1} \alpha_i \cdot \upgamma_i(\mu_1) \tau_1 = d^{2p}(x)
\]
But then $\beta=x-\tilde{x}$ is a $d^{2p}$-cycle, hence one can conclude:
\[
    x= \tilde{x}+\beta=\sum_{i=0}^{s-1} \alpha_i \cdot \upgamma_{i+1}(\mu_1) +\beta
\]
as desired.

One proves $4_s$ by a very similar argument: pick an element $0 \neq x \in E^{2p}_{j,*,*}$, with $j \leq 2p(s-1)$, which is a multiple of $\tau_1$. By what we have just shown and by the tensor product decomposition of the $E^{2p}$-page, this means that it can be written as:
\[
    x=\left(\sum_{i=0}^{s-1} \alpha_i\upgamma_i(\mu_1) \right)y \tau_1
\]
with $d^{2p}(\alpha_i)=0$ and $y$ not a multiple of $\tau_1$ (otherwise $x=0$ as $\tau_1^2=0$).
By what we have just seen, the element:
\[
    \sum_{i=0}^{s-1} \alpha_i\upgamma_{i+1}(\mu_1) y
\]
exists and $d^{2p}$-differentiates to $x$.

To conclude, consider a sum:
\[
    \sum_{i=0}^s \alpha_i \cdot \upgamma_i(\mu_1)=0
\]
living in a degree smaller than or equal to $2ps$. By applying the $d^{2p}$ differential:
\[
0=d^{2p}(\sum_{i=0}^s \alpha_i \upgamma_i(\mu_1))=\sum_{i=1}^s \alpha_i  \upgamma_{i-1}(\mu_1) \tau_1.
\]
From the tensor product decomposition, this implies that $\sum_{i=1}^s \alpha_i  \upgamma_{i-1}(\mu_1)=0$, hence, by $3_{s-1}$, that $\alpha_1=\ldots=\alpha_s=0$. Then also $\alpha_0=0$.

Part $B.$ of the proof is a consequence of the presentation of the divided power structure in positive characteristic $p$:
\[
   \Gamma_{\mathbb{F}_p[\tau]/\tau^{p-1}}(\mu_1) \cong \mathbb{F}_p[\tau, \upgamma_{p^i}(\mu_1)]_{i \geq 0}/\langle \tau^{p-1}, (\upgamma_{p^i}(\mu_1))^p \rangle
\]
For any $s \in \mathbb{N}$, write $s= \sum_{i \in \mathbb{N}} a_i p^i$, with $0 \leq a_i <p$. Observe that the choice of the $a_i$ is unique (and the sum is finite). Then one defines:
\[
    \upgamma_s (\mu_1) := \prod_{i \in \mathbb{N}} \left(a_i!\right)^{-1} (\upgamma_{p^i}(\mu_1))^{a_i}
\]
Under our assumptions, these products are well defined for all $s \leq p^{j+1}-1$ as they involve only the $\upgamma_{p^i}(\mu_1)$ up to $i=j$. It is also immediate from the hypothesis that their $p$-th power vanishes, thus proving $2s$. It is finally an algebraic exercise to show:

\begin{lemma}
    With the above definitions, for $s \leq p^{j+1}-1$, $d^{2p}(\upgamma_s (\mu_1))=\upgamma_{s-1} (\mu_1) \tau_1$.
\end{lemma}

\begin{proof}
    Recall that:
    \[
        d^{2p}(\upgamma_{p^i}(\mu_1))=\upgamma_{p^i-1}(\mu_1) \tau_1= \left(\prod_{k=1}^{i-1} \left((p-1)!\right)^{-1} (\upgamma_{p^k}(\mu_1))^{p-1} \right)\tau_1.
    \]
    One has:
    \[
    \begin{split}
        d^{2p}(\upgamma_s (\mu_1))&= d^{2p}\left(\prod_{i \in \mathbb{N}} \left(a_i!\right)^{-1} (\upgamma_{p^i}(\mu_1))^{a_i} \right) \\
        &= \sum_{\substack{i \in \mathbb{N} \\ i:\, a_i > 0}} \left((a_i-1)!\right)^{-1} (\upgamma_{p^i}(\mu_1))^{a_i-1} \left(\prod_{k=1}^{i-1} \left((p-1)!\right)^{-1} (\upgamma_{p^k}(\mu_1))^{p-1} \right) \\&\qquad \tau_1\prod_{j \in \mathbb{N}-\{i\}} \left((a_j)!\right)^{-1} (\upgamma_{p^j}(\mu_1))^{a_j}
    \end{split}
      \]  
    
    by the Leibniz rule (observe that the degree of each factor is even, so no minus signs are appearing). 
    Let $\bar{i}$ be the smallest index for which $a_i>0$. Then each summand with $i \neq \bar{i}$ is null: the $p$-th power of $\upgamma_{p^i}(\mu_1)$ appears as a factor, and by hypothesis $2_{p^{\bar{i}}}$ it is zero. Hence:
    \begin{multline*}
        d^{2p}(\upgamma_s (\mu_1))= \left((a_{\bar{i}}-1)!\right)^{-1} (\upgamma_{p^{\bar{i}}}(\mu_1))^{a_{\bar{i}}-1} \left(\prod_{k=1}^{\bar{i}-1} \left((p-1)!\right)^{-1} (\upgamma_{p^k}(\mu_1))^{p-1}\right)  \\ \left(\prod_{j \in \mathbb{N}-\{\bar{i}\}} \left((a_j)!\right)^{-1} (\upgamma_{p^j}(\mu_1))^{a_j}\right) \tau_1
    \end{multline*}
    which is exactly $\upgamma_{s-1} (\mu_1) \tau_1$ according to our notation.
\end{proof}

We proceed then with point $C.$ of the proof. The base step has already been discussed in the introduction to this section; for greater clarity, we give a separate argument also for $j=1$, as it is slightly simpler than the general case and is close enough to the origin to be visualised on the spectral sequence graph. In particular, we will show:
\begin{itemize}
    \item That $\upgamma_{p-1}(\mu_1) \tau_1$ cannot support a non-trivial differential; hence, by the convergence assumption, it must be hit by some non-trivial differential, which, by the induction hypothesis on $E^{2p}$ page, must originate from the zeroth horizontal line. We call $\upgamma_p (\mu_1)$ the origin of such differential.
    \item That $(\mu_1)^p=0$. 
\end{itemize}




Let's argue the first claim by contradiction. Given the tensor product decomposition of the $E^{2p}$-page, trivial elements in $E^{2p}_{0,*,*}$ detect trivial rows. The non-trivial $\mathbb{F}_p[\tau]/\tau^{p-1}$-module generators in the zeroth column with the smallest degrees are:
\begin{equation}\label{eq:degr.weights.E2p}
    \begin{gathered} 
    1 \text{ with (degree, weight): } (0,0) \\
    \tau_1 \text{ with (degree, weight): } (2p-1, p-1) \\
    \xi_2 \text{ with (degree, weight): } (2p^2-2,p^2-1) \\
    \tau_2 \text{ with (degree, weight): } (2p^2-1,p^2-1) \\
\end{gathered}
\end{equation}
The weights of elements of higher degrees are higher as well. Each of these elements generates a free module over $\mathbb{F}_p[\tau]/\tau^{p-1}$, so weights range from that of the element to that minus $p-2$.

Suppose then by contradiction that $\upgamma_{p-1}(\mu_1) \tau_1$ supports a non-trivial differential, say of length $l_1$. As $d^{2p}(\upgamma_{p-1}(\mu_1) \tau_1)=0$, $l_1>2p$.
Even though we do not know (yet) that the $E^{l_1}$ enjoys the tensor-product decomposition, every element there is a sub-quotient of a homogeneous module in the $E^{2p}$-page. One can then choose elements $\alpha_1 \in E^{2p}_{*,0,*}$ and $\beta_1  \in E^{2p}_{0,*,*}$ such that $\alpha_1 \beta_1$ is a non-trivial summand of an element belonging to the image $d^{l_1}(\upgamma_{p-1}(\mu_1) \tau_1)$. Then:
\[
    |\alpha_1 \beta_1| = ((p-1)2p-l_1, 2p-1+l_1-1, p(p-1))
\]
Now, $\beta_1$ lies in the row of $\xi_2$ or higher, so it has minimum degree $2p^2-2$, so $2p-2+l_1 \geq 2p^2-2$, in other words, $l_1 \geq 2p^2-2p$. But then the degree of $\alpha_1$ is $(p-1)2p-l_1\leq (p-1)2p-2p^2-2p=0$, so it must be $0$, by the shape of the spectral sequence. (We can then assume $\alpha_1=1$ without loss of generality). Hence:
\[
    |\alpha_1 \beta_1| = (0, 2p^2-2, p^2-p)
\]
In the $E^{2p}$-page, the module in that position is generated by $\xi_2$, which has weight $p^2-1$. The minimum weight achieved there, given by the greatest non-trivial power of  $\tau$, $\tau^{p-2}$, is $p^2-p+2>p^2-p$, hence there can be no non-trivial differential out of $\upgamma_{p-1}(\mu_1) \tau_1$. There must be then an element $\upgamma_{p} (\mu_1) \in E^{2p}_{2p^2, p(p-1)}$ with $d^{2p}(\upgamma_p (\mu_1))= \upgamma_{p-1}(\mu_1) \tau_1$.

Consider now $(\mu_1)^p$. We already argued that $d^{2p}((\mu_1)^p)=0$. We show that also all the higher differentials are trivial. Given the shape of the $E^{2p}$-page, the next (and only) possible non-trivial differential could be a $d^{2p^2}$, reaching degrees $(0, 2p^2-1, p(p-1))$. The module here is (a subquotient of the module) generated by $\tau_2$: the elements here have weight $p^2-1-p+2 \leq w \leq p^2-1$. But $p(p-1)=p^2-p < p^2-p+1$, so no non-trivial arrow is possible. We can conclude that, as $(\mu_1)^p$ is a permanent cycle, given the constraints from the convergence result of the spectral sequence, we must have $(\mu_1)^p=0$.

Let's move to the general induction argument. We first prove statement $1_{p^j}$. Consider the element $\upgamma_{p^j-1} (\mu_1) \tau_1 \in E^{2p}_{(p^j-1)2p, 2p-1, p^j(p-1)}$. The proof is split into the following three considerations:
\begin{enumerate}[label={(\roman*)}]
    \item No polynomial of elements of degree smaller then $p^j(2p)$ hits it with a $d^{2p}$ differential.
    \item It has trivial $d^{2p}$.
    \item It cannot support a longer non-trivial differential.
\end{enumerate}
Since it must disappear in the $E^{\infty}$-page, there must be then a new class $\upgamma_{p^j} (\mu_1) \in E^{2p}_{(p^j)2p, 0, p^j(p-1)}$ with $d^{2p}(\upgamma_{p^j} (\mu_1))=\upgamma_{p^j-1} (\mu_1) \tau_1$, proving $1_{p^j}$.

To prove $(i)$, suppose by contradiction such a polynomial $P$ exists. Recall that the columns $E^{2p}_{i, *, *}$ are empty for all $1\leq i \leq 2p$, so all factors in the monomials of the polynomial have degree at most $2p(p^j-1)$. We can then apply our induction hypotheses $1_{p^j-1}$, $2_{p^j-1}$ and $3_{p^j-1}$ to each monomial. Hence, by rearranging the terms, we can write:
\[
    P= \sum_{s=1}^{p^j-1} a_s \upgamma_s (\mu_1) + k (\upgamma_{p^{j-1}} (\mu_1))^p
\]
with $d^{2p}(a_s)=0$ and $k \in \mathbb{F}_p[\tau]/\tau^{p-1}$.

Apply the $d^{2p}$ differential:
\[
 d^{2p}(P)=\sum_{s=1}^{p^j-1} a_s \upgamma_{s-1} (\mu_1) \tau_1=\upgamma_{p^{j}-1} (\mu_1) \tau_1.
\]
By the tensor product decomposition, this is equivalent to a relation:
\[
\upgamma_{p^{j}-1} (\mu_1) - \sum_{s=1}^{p^j-1} a_s \upgamma_{s-1} (\mu_1)=0
\]
living in degree $(p^j-1)2p$, which contradicts hypothesis $3_{p^j-1}$.

The proof of $(ii)$ is a simple computation:
\[
d^{2p}(\upgamma_{p^{j}-1} (\mu_1) \tau_1)=\upgamma_{p^{j}-2} (\mu_1) \tau_1 \tau_1=0
\]
because $\tau_1^2=0$.

The proof of $(iii)$ is a bit more involved; it roughly goes as follows: one assumes by contradiction that there exists a non-trivial differential out of $\upgamma_{p^{j}-1} (\mu_1) \tau_1$, which, in light of $(ii)$, is longer than a $d^{2p}$, and consider the degree and weight of the element that is hit. This happens on a higher page, so the element does not correspond in general to a single element of the $E^{2p}$-page, but rather to a homogeneous submodule. Since our argument involves only degrees and weights, we might as well identify it with a monomial in the $E^{2p}$-page, decomposing it as a product of some $\alpha$ on the horizontal line $E^{2p}_{*,0,*}$ and some $\beta$ on the vertical line $E^{2p}_{0,*,*}$. We then look at a non-trivial image of $\alpha$ via some longer differential and similarly decompose it thanks to analogous considerations. We repeat the process until we hit the vertical line. By carefully tracing variation in degrees and weights (we will see that the differentials involved have to be quite long) and finally applying lemma \ref{lemma:estimate.on.monomials} we reach a contradiction. This is visually rendered in figure \ref{fig:induction.E2p}.

\begin{figure}
    \centering
    \adjustbox{max width=\textwidth, center}{\begin{tikzpicture}
        \filldraw[almostwhite] (10.5,0) rectangle (12.5,9.3);

        \draw[red!30] (0,3) -- (27.7,3);
        \draw[red!30] (27.5,0) -- (27.5,3.2);
        
        \draw[red!30] (0,7) -- (27.7,7); 

        \draw[gray!60] (0,8) -- (21.7,8);
        \draw[gray!60] (21.5,0) -- (21.5,8.2);

        \draw[gray!60] (0,7.5) -- (13.2,7.5);
        \draw[gray!60] (13,0) -- (13,7.7);

        \draw[gray!60] (0,8.5) -- (10.2,8.5);
        \draw[gray!60] (10,0) -- (10,8.7);

        \draw[black, ->] (0,0) -- (0,9.5);
        \draw[black] (0,0) -- (10.5,0);
        \draw[black, dashed] (10.5,0) -- (12.5,0);
        \draw[black, ->] (12.5,0) -- (28,0);

        {\filldraw[red]
        (27.5,3) circle (1pt) node[right]{\normalsize $\upgamma_{2p^j-1}(\mu_1)\tau_1$}
        (27.5,0) circle (1pt) node[below]{\normalsize $\upgamma_{2p^j-1}(\mu_1)$}
        (0,3) circle (1pt) node[left]{\normalsize $\tau_1$}
        (0,7) circle (1pt) node[left]{\normalsize $\xi_2$}
        ;}
        
        {\filldraw[black]
        (21.5,8) circle (1pt) 
        (21.5,0) circle (1pt) node[below]{\normalsize $\alpha_1$}
        (0,8) circle (1pt) node[left]{\normalsize $\omega_1$}
        
        (13,7.5) circle (1pt) 
        (13,0) circle (1pt) node[below]{\normalsize $\alpha_2$}
        (0,7.5) circle (1pt) node[left]{\normalsize $\omega_2$}
        
        (10,8.5) circle (1pt) 
        (10,0) circle (1pt) node[below]{\normalsize $\alpha_{n-1}$}
        (0,8.5) circle (1pt) node[left]{\normalsize $\omega_{n-1}$}
        
        (0,9) circle (1pt) node[left]{\normalsize $\omega_n$}
        ;}

        \draw[black, ->] (27.5,3) -- (21.5,8) node[right=6pt,  midway] {\normalsize  $d^{k_1}$};
        \draw[black, ->] (21.5,0) -- (13,7.5)  node[right=6pt,  midway] {\normalsize  $d^{k_2}$};
        \draw[black] (13,0) -- (12.5,0.4375);
        \draw[black, ->] (10.5,8.1) -- (10,8.5);
        \draw[black, ->] (10,0) -- (0,9)  node[right=6pt,  midway] {\normalsize  $d^{k_n}$};
        
    \end{tikzpicture}}
    \caption{A visualisation of what should happen if $\upgamma_{2p^j-1}(\mu_1)\tau_1$ supported a non trivial differential.}
    \label{fig:induction.E2p}
\end{figure}
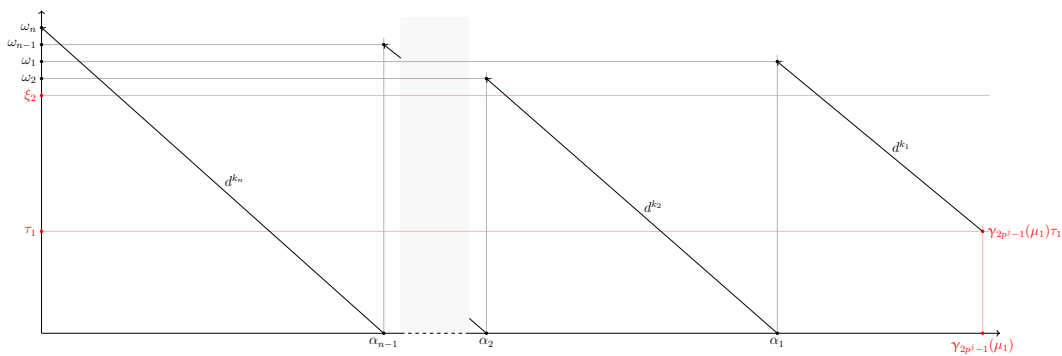

First things first, suppose then $\upgamma_{p^{j}-1} (\mu_1) \tau_1$ survives to a higher page and supports a non-trivial longer differential $d^{k_1}$. Let 
\begin{gather*}
    \alpha_1 \in E^{2p}_{(p^j-1)2p-k_1,0, p^j(p-1)-w_1}\\
    \beta_1 \in E^{2p}_{0,2p+k_1-2, w_1}
\end{gather*}
be such that $\alpha_1\beta_1$ has the same triple degree of $d^{k_1}(\upgamma_{p^{j}-1} (\mu_1) \tau_1)$ and $\alpha_1\beta_1 \neq 0 \in E^{k_1}$. 

Such elements exist because the $E^{k_1}$-page is a homogeneous subquotient of the $E^{2p}$-page, which by hypothesis decomposes as a tensor product of a vertical and a horizontal ring. So elements in the $E^{k_1}$-page can be seen as homogeneous sums of elements from the $E^{2p}$-page, modulo some additional relation.
We can impose a lower bound on $k_1$. From \ref{eq:degr.weights.E2p}, the next non trivial row at the $E^{2p}$ page is the $(2p^2-2)$-th, represented by the element $\xi_2$, so $2p+k_1-2\geq 2p^2-2$, or $k_1 \geq 2p^2-2p$. Observe that we have a lower bound on the weight: as the weights of the generators of the dual Steenrod algebra, the smallest weight appearing in the module generated by $\xi_2$ (that of $\tau^{p-2}\xi_2$) is the smallest weight that $\beta_1$ can have. In particular, $w_1 \geq p^2-p+1$.

Next, we consider the behaviour in the spectral sequence of $\alpha_1$. Observe that its degree is 
\begin{align*}
    (p^j-1)2p-k_1\leq (p^j-1)2p-2p^2+2p=2p^{j+1}-2p^2 \\
    =2p(p^j-p)\leq 2p(p^j-2),
\end{align*}
so we are in the range where both $3_{p^j-1}$ and $4_{p^j-1}$ hold: in particular,  $\tau_1$ does not divide $\beta_1$. Also, $\alpha_1$ has to support a differential longer than a $d^{2p}$: if that were not the case, by $3_{p^j-1}$ we could write:
\[
    \alpha_1= \sum_{i=0}^{p^j-1} \tilde{\alpha_i} \upgamma_i(\mu_1)
\]
$d^{2p}(\tilde{\alpha_i})=0$, with some $\tilde{\alpha_i}\neq 0$ for $i >0$. By $4_{p^j-1}$, $\beta_1$ is no multiple of $\tau_1$ (otherwise $\alpha_1 \beta_1$ would vanish after the $E^{2p}$-page). But then also $\alpha_1 \beta_1$ would support a non-trivial $d^{2p}$-differential.

Say then $\alpha_1$ supports a $d^{k_2}$. As before, we can identify:
\begin{gather*}
    \alpha_2 \in E^{2p}_{(p^j-1)2p-k_1-k_2,0, p^j(p-1)-w_1-w_2}\\
    \beta_2 \in E^{2p}_{0,k_2-1, w_2}
\end{gather*}
such that $\alpha_2\beta_2$ has the same triple degree of $d^{k_2}(\alpha_1)$ and $\alpha_2\beta_2 \neq 0 \in E^{k_2}$. In particular, $\tau_1$ does not divide $\beta_2$.
As before, the shortest possibility for $d^{k_2}$ is to reach the line generated by $\xi_2$. This imposes:
\begin{gather*}
    k_2 \geq 2p^2-1\\
    w_2 \geq p^2-p+1.
\end{gather*}

One can apply the same reasoning to the element $\alpha_2$, thus finding elements:
\begin{gather*}
    \alpha_3 \in E^{2p}_{(p^j-1)2p-k_1-k_2-k_3,0, p^j(p-1)-w_1-w_2-w_3}\\
    \beta_3 \in E^{2p}_{0,k_3-1, w_3}
\end{gather*}
with
\begin{gather*}
    k_3 \geq 2p^2-1\\
    w_3 \geq p^2-p+1,
\end{gather*}
with $\tau_1$ not dividing $\beta_2$. We repeat the process $n$ times, finding analogoud sonclusions for the $\alpha_i$ and $\beta_i$, until the differential $d^{k_n}(\alpha_{n-1})=\beta_n$ lands in the zeroth column. Let $w_n$ be the weight of such $\beta_n$; observe that the conditions:
\begin{gather*}
    k_n \geq 2p^2-1\\
    w_n \geq p^2-p+1
\end{gather*}
hold also in this case. From the degrees and weights of the $\alpha_i$, we get the equations:
\begin{equation}\label{eq:sum.ki.wi.mu1}
    \begin{gathered}
    (p^j-1)2p-\sum_{i=1}^n k_i=0 \\
    p^j(p-1)-\sum_{i=1}^{n-1} w_i=w_n.
\end{gathered}
\end{equation}

Lemma \ref{lemma:estimate.on.monomials} provides further estimates on the quantities $k_i$ and $w_i$. In fact, the $\beta_i$ all come from elements of the dual Steenrod algebra that can be expressed as homogeneous sums of monomials where $\tau_0$, $\xi_1$ and $\tau_1$ are absent. We have then inequalities:
\begin{gather*}
    Q(2)(2p+k_1-2) \leq w_1 \leq \frac{1}{2}(2p+k_1-2) \\
    Q(2)(k_j-1) \leq w_j \leq \frac{1}{2}(k_j-1) \text{ for } j \geq 2
\end{gather*}
where:
\[
Q(2)=\frac{1}{2}-\frac{p-2}{2p^{2}-2}-\frac{1}{4p^{2}-2}.
\]

Summing on all indices $j$:
\[
    Q(2) \left (2p+k_1-2+ \sum_{j=2}^n (k_j-1) \right) \leq \sum_{j=1}^n w_j \leq \frac{1}{2} \left (2p+k_1-2+ \sum_{j=2}^n (k_j-1) \right )
\]
Or:
\[
Q(2)\sum_{j=1}^n k_j + Q(2)(2p-n-1) \leq \sum_{j=1}^n w_j\leq \frac{1}{2}\sum_{j=1}^n k_j + \frac{1}{2}(2p-n-1)
\]
Plugging in the equalities \ref{eq:sum.ki.wi.mu1}:
\[
Q(2)(p^j-1)2p + Q(2)(2p-n-1) \leq p^j(p-1)\leq (p^j-1)p + \frac{1}{2}(2p-n-1)
\]
Rearranging the first inequality in particular gives:
\begin{equation*}
    \begin{gathered}
    Q(2)(p^j-1)2p- p^j(p-1)\leq Q(2)(n+1-2p)\\
    (p^j-1)2p- \frac{p^j(p-1)}{Q(2)} +2p \leq n+1\\
    2p^{j+1}- \frac{p^j(p-1)}{Q(2)} -1 \leq n.
\end{gathered}
\end{equation*}
Now go back to equation \ref{eq:sum.ki.wi.mu1}; in particular using the fact that each $w_i \geq p^2-p+1$:
\[
 p^j(p-1)=\sum_{i=1}^{n} w_i\geq (p^2-p+1)n \geq (p^2-p+1) \left(2p^{j+1}- \frac{p^j(p-1)}{Q(2)} -1 \right)
\]
This is a contradiction. In fact:
\[
\begin{gathered}
    \frac{p^j(p-1)}{p^2-p+1}  \geq 2p^{j+1}- \frac{p^j(p-1)}{Q(2)} -1 \\
    2p^{j+1}-\frac{p^j(p-1)}{p^2-p+1} -1 \leq 2p^{j+1}-\frac{p^j(p-1)}{p^2-p+1} \leq  \frac{p^j(p-1)}{Q(2)}\\
    2p-\frac{p-1}{p^2-p+1} \leq \frac{p-1}{Q(2)}\\
    \frac{2p}{p-1}-\frac{1}{p^2-p+1}\leq \frac{1}{Q(2)}
\end{gathered}
\]
Plug now in the estimate \ref{cor:rough.est.Q(i)}:
\[
\begin{gathered}
    \frac{2p}{p-1}-\frac{1}{p^2-p+1}\leq \frac{4(p^2-1)}{2p^2-2p+1}\\
    \frac{2p(2p^2-2p+1)}{p-1}-\frac{2p^2-2p+1}{p^2-p+1}\leq 4p^2-4\\
    \frac{2p(2p^2-2p)}{p-1}+\frac{2p}{p-1}-\frac{2p^2-2p+1+1-1}{p^2-p+1}\leq 4p^2-4\\
    4p^2+\frac{2p-2+2}{p-1}-\frac{2(p^2-p+1)-1}{p^2-p+1}\leq 4p^2-4\\
    2+\frac{2}{p-1}-2+\frac{1}{p^2-p+1}\leq -4\\
    \frac{2}{p-1}+\frac{1}{p^2-p+1}\leq -4\\
\end{gathered}
\]
which is a contradiction, as the left-hand side is positive. This concludes the proof of $3_{p^j}$.

The proof of $4_{p^j}$ follows a similar strategy. In particular, as:
\[
    d^{2p}((\upgamma_{p^{j-1}} (\mu_1))^p)=p (\upgamma_{p^{j-1}} (\mu_1))^{p-1} \tau_1=0
\]
(recall that $\upgamma_{p^{j-1}} (\mu_1)$ sits in even degree), we just need to prove that it cannot support a nontrivial longer differential. Then, since there are no permanent cycles away from $(0,0,*)$, we must have $(\upgamma_{p^{j-1}} (\mu_1))^p=0$. Recall that:
\[
|(\upgamma_{p^{j-1}} (\mu_1))^p|=(p\cdot 2p^{j}, 0, p \cdot p^{j-1}(p-1))=(2p^{j+1},0,p^j(p-1)).
\]

Suppose by contradiction that there is some $d^{k_1}((\upgamma_{p^{j-1}} (\mu_1))^p)\neq 0 \in E^{k_1}$. We can then identify elements:
\begin{gather*}
    \alpha_1 \in E^{2p}_{2p^{j+1}-k_1,0, p^j(p-1)-w_1}\\
    \beta_1 \in E^{2p}_{0,k_1-1, w_1}
\end{gather*}
such that $\alpha_1\beta_1 \neq 0 \in E^{k_1}$ sits in the same degrees and weight of $d^{k_1}((\upgamma_{p^{j-1}} (\mu_1))^p)$; as before, such elements can always be found. Moreover, due to the induction hypothesis, we know that $\tau_1$ cannot divide $\beta_1$.
As $k_1 > 2p$, because of the structure of the $E^{2p}$-page, we can conclude by \ref{eq:degr.weights.E2p}:
\[
\begin{gathered}
    k_1-1 \geq 2p^2-2\\
    w_1 \geq p^2-p+1.
\end{gathered}
\]

We study then the behaviour of $\alpha_1$ in the spectral sequence; as its degree is:
\[
    2p^{j+1}-k_1\leq 2p^{j+1}-2p^2+1 \leq 2p(p^j-2)
\]
by a simple computation, we are in the range where both $3_{p^{j}-1}$ and $4_{p^{j}-1}$ apply. The same argument as before 
gives that $\alpha_1$ supports a differential $d^{k_2}$ with $k_2>2p$; in particular, it allows to find elements:
\begin{gather*}
    \alpha_2 \in E^{2p}_{2p^{j+1}-k_1-k_2,0, p^j(p-1)-w_1-w_2}\\
    \beta_2 \in E^{2p}_{0,k_2-1, w_2}
\end{gather*}
with $\alpha_2\beta_2 \neq 0 \in E^{k_2}$ sitting in the same degrees and weight of $d^{k_2}(\alpha_1)$ and $\tau_1$ not dividing $\beta_2$; the constraints:
\[
\begin{gathered}
    k_2-1 \geq 2p^2-2\\
    w_2 \geq p^2-p+1
\end{gathered}
\]
apply here as well.

We proceed in the same way for other $n-2$ times, drawing the same conclusions, until
\[
    d^{k_n}(\alpha_{n-1})=\beta_n \in E^{2p}_{0,k_n-1, w_n}
\]
lands in the zeroth column. Naturally, $k_n$ and $w_n$ satisfy the same constraints as the other $k_i$ and $w_i$.
Putting everything together, we get the equations:
\begin{equation} \label{eq:sum.ki.wi.gamma.pj.mu1.^p}
    \begin{gathered}
    2p^{j+1}= \sum_{i=1}^n k_i\\
    p^j(p-1)= \sum_{i=1}^n w_i.
\end{gathered}
\end{equation}

Because the $\beta_i$ can be seen as elements of the dual motivic Steenrod algebra modulo $p$ that are not multiples of $\tau_0$, $\tau_1$ or $\xi_1$, from \ref{lemma:estimate.on.monomials}  we have the additional constraints:
\[
    Q(2)(k_i-1) \leq w_i \leq \frac{k_i-1}{2} \text{ for all } 1 \leq i \leq n.
\]
Focus on the first inequality. Summing on all indices $i$ and plugging in \ref{eq:sum.ki.wi.gamma.pj.mu1.^p}:
\[
\begin{gathered}
    Q(2) \left (\sum_{i=1}^n k_i \right)-Q(2)n \leq \sum_{i=1}^n w_i \\
    Q(2)(2p^{j+1})-Q(2)n \leq p^j(p-1)
\end{gathered}
\]
So
\[
n \geq 2p^{j+1} - \frac{p^j(p-1)}{Q(2)}.
\]
Now go back to \ref{eq:sum.ki.wi.gamma.pj.mu1.^p}, and recall that each $k_i \geq 2p^2-1$:
\[
    2p^{j+1}= \sum_{i=1}^n k_i \geq n(2p^2-1) \geq \left (2p^{j+1} - \frac{p^j(p-1)}{Q(2)}\right)(2p^2-1).
\]
This is a contradiction. In fact:
\[
\begin{gathered}
    2p^{j+1}\geq \left (2p^{j+1} - \frac{p^j(p-1)}{Q(2)}\right)(2p^2-1)\\
    2p \geq \left (2p-\frac{p-1}{Q(2)}\right)(2p^2-1)\\
    Q(2) \geq \left (Q(2)-\frac{p-1}{2p}\right)(2p^2-1)\\
    \frac{p-1}{2p}(2p^2-1) \geq Q(2)(2p^2-2)   
\end{gathered}  
\]
Now use the estimate \ref{cor:rough.est.Q(i)}:
\[
\begin{gathered}
    \frac{p-1}{2p}(2p^2-1) \geq Q(2)(2p^2-2)\geq \frac{2p^2-2p+1}{4(p^2-1)}(2p^2-2)=\frac{2p^2-2p+1}{2}\\
    (p-1)(2p^2-1)\geq p(2p^2-2p+1)\\
    2p^3-2p^2-p+1 \geq 2p^3-2p^2+p\\
    1 \geq 2p
\end{gathered}
\]
which is a clear contradiction. We have thus proven $4_{p^j}$, concluding the induction argument for the $E^{2p}$-page.

We can then conclude that there are commutative diagrams:
\[
\begin{tikzcd}
    E^{2p+1}_{*,0,*} 
        \arrow[r, hook]
        \arrow[rd, "\sim"]&
    E^{2p}_{*,0,*}
        \arrow[d, two heads] \\
    &
    E^{2p}_{*,0,*}/\langle \upgamma_i (\mu_1) \rangle_{i \geq 1}
\end{tikzcd}
\]
\[
    E^{2p+1}_{*,0,*}\cong \pi_{*,*}(MHH(M\mathbb{Z}/p)/\tau^{p-1}) /\langle \upgamma_i (\mu_0), \lambda_1, \upgamma_i (\mu_1) \rangle_{i \geq 1}
\]
\[
    E^{2p+1}_{0,*,*} \cong E^{2p}_{0,*,*} /\langle \tau_1 \rangle \cong \pi_{*,*}(\mathcal{A}(p)/\tau^{p-1}) /\langle \tau_0, \xi_1, \tau_1 \rangle
\]
\[
    E^{2p+1}_{\star, \bullet, *} \cong E^{2p+1}_{\star,0,*} \otimes_{\mathbb{F}_{p}[\tau]/\tau^{p-1}} E^{2p+1}_{0,\bullet,*}
\]

In particular, we can identify a subalgebra:
\begin{multline}\label{eq:incl.iso.E2p}
    \Gamma_{\mathbb{F}_p[\tau]/\tau^{p-1}}(\mu_0)  \otimes_{\mathbb{F}_p[\tau]/\tau^{p-1}}  \Lambda_{\mathbb{F}_p[\tau]/\tau^{p-1}}(\lambda_1) \otimes_{\mathbb{F}_p[\tau]/\tau^{p-1}} \Gamma_{\mathbb{F}_p[\tau]/\tau^{p-1}}(\mu_1)\\
    \hookrightarrow \pi_{*,*}(MHH(M\mathbb{Z}/p)/\tau^{p-1})
\end{multline} 
such that the inclusion is an isomorphism in degrees smaller than or equal to $2p$.

Observe that, by the above conclusions, the rows $E^{2p+1}_{*, i,*}$ for $1 \leq i \leq 2p^2-3$ are empty; consequently, by the convergence constraint, so must be the columns $E^{2p+1}_{j,*,*}$ for $1 \leq j \leq 2p^2-2$. This shows in particular that the inclusion appearing in equation \ref{eq:incl.iso.E2p} is an isomorphism actually up to degree $2p^2-2$. We will expand on this in the next section.
\subsection{The induction step}\label{subs:induction.step.p}

We prove now the general induction step. First of all, we clarify the induction hypothesis, which we have kept more or less implicit up to now. The $t$-th stage of the induction studies the behaviour of the spectral sequence in the $E^t$-page and illustrates the appearance of the $E^{t+1}$-page. More precisely, we distinguish for the induction hypothesis $H_t$ three cases:

\begin{itemize}
    \item[$H_{2p^j-1}$] For $t=2p^j-1$, for some integer $j$, the source of all the $d^{2p^j-1}$-differentials is the subalgebra generated by a square-zero element $\lambda_j \in E^{2p^j-1}_{(2p^j-1, 0, p^j-1)}$, that hits transgressively $\xi_j \in E^{2p^j-1}_{(0, 2p^j-2, p^j-1)}$. In fact, there are commutative diagrams:
\[
\begin{tikzcd}
    E^{2p^j}_{*,0,*} 
        \arrow[r, hook]
        \arrow[rd, "\sim"]&
    E^{2p^j-1}_{*,0,*}
        \arrow[d, two heads] \\
    &
    E^{2p^j-1}_{*,0,*}/\langle \lambda_j \rangle
\end{tikzcd}
\]
so that:
\[
    \begin{adjustbox}{max width=\textwidth, center}
    $\displaystyle
        E^{2p^j}_{*,0,*} \cong \pi_{*,*}(MHH(M\mathbb{Z}/p)/\tau^{p-1})/\langle \upgamma_i(\mu_0), \upgamma_i(\mu_1), \ldots, \upgamma_i(\mu_{j-1}), \lambda_1, \ldots, \lambda_j \rangle_{i\geq 1},
    $    
    \end{adjustbox}
\]
\begin{align*}
    E^{2p^j}_{0,*,*} &\cong E^{2p^j-1}_{0,*,*} /\langle \xi_j \rangle
    \\ &\cong \pi_{*,*}(\mathcal{A}(p)/\tau^{p-1}) /\langle \tau_0, \tau_1, \ldots, \tau_{j-1}, \xi_1, \ldots, \xi_j \rangle
\end{align*}
and a decomposition:
\begin{align*}
    E^{2p^j}_{\star, \bullet, *} \cong E^{2p^j}_{\star,0,*} \otimes_{\mathbb{F}_{p}[\tau]/\tau^{p-1}} E^{2p^j}_{0,\bullet,*}
\end{align*}

In particular, we can identify a subalgebra:
\begin{multline*}
    \left ( \bigotimes_{h=0}^{j-1}  \Gamma_{\mathbb{F}_p[\tau]/\tau^{p-1}}(\mu_h) \right ) \otimes_{\mathbb{F}_p[\tau]/\tau^{p-1}} \left (  \bigotimes_{h=1}^{j}  \Lambda_{\mathbb{F}_p[\tau]/\tau^{p-1}}(\lambda_h) \right ) \\
    \hookrightarrow \pi_{*,*}(MHH(M\mathbb{Z}/p)/\tau^{p-1})
\end{multline*}
such that the inclusion is an isomorphism in degrees smaller than or equal to $2p^j-1$.
\begin{rmk}
    This description of $\pi_{*,*}(MHH(M\mathbb{Z}/p)/\tau^{p-1})$, together with the already known  structure of $\pi_{*,*}(\mathcal{A}(p)/\tau^{p-1})$, implies that in the $E^{2p^j}$-page the rows with indices from 1 to $2p^j-2$ and the columns with indices from 1 to $2p^j-1$ are empty. This is compatible with the spectral sequence being first-quadrant and the sparsity of the $E^{\infty}$-page.
\end{rmk}

\item[$H_{2p^j}$]  For $t=2p^j$, for some integer $j$, the source of all the $d^{2p^j}$-differentials is the subalgebra generated by the divided powers $\upgamma_i (\mu_j)$ on an element $\mu_j \in E^{2p^j}_{(2p^j, 0, p^j-1)}$. In particular, $\mu_j$ hits transgressively $\tau_j \in E^{2p^j}_{(0, 2p^j-1, p^j-1)}$. In fact, there are commutative diagrams:
\[
\begin{tikzcd}
    E^{2p^j+1}_{*,0,*} 
        \arrow[r, hook]
        \arrow[rd, "\sim"]&
    E^{2p^j}_{*,0,*}
        \arrow[d, two heads] \\
    &
    E^{2p^j}_{*,0,*}/\langle \upgamma_i (\mu_j) \rangle_{i \geq 1}
\end{tikzcd}
\]
so that:
\[
\begin{adjustbox}{max width=\textwidth, center}
    $\displaystyle
    E^{2p^j+1}_{*,0,*} \cong \pi_{*,*}(MHH(M\mathbb{Z}/p)/\tau^{p-1})/\langle \upgamma_i(\mu_0), \upgamma_i(\mu_1), \ldots, \upgamma_i(\mu_{j}), \lambda_1, \ldots, \lambda_j \rangle_{i\geq 1},
    $
\end{adjustbox}
\]
\begin{align*}
    E^{2p^j+1}_{0,*,*} &\cong E^{2p^j}_{0,*,*} /\langle \tau_j \rangle
                    \\ &\cong \pi_{*,*}(\mathcal{A}(p)/\tau^{p-1}) /\langle \tau_0, \tau_1, \ldots, \tau_{j}, \xi_1, \ldots, \xi_j \rangle
\end{align*}
and a decomposition:
\begin{align*}
    E^{2p^j+1}_{\star, \bullet, *} \cong & E^{2p^j+1}_{\star,0,*} \otimes_{\mathbb{F}_{p}[\tau]/\tau^{p-1}} E^{2p^j+1}_{0,\bullet,*}
\end{align*}

In particular, we can identify a subalgebra:
\begin{multline*}
    \left ( \bigotimes_{h=0}^{j}  \Gamma_{\mathbb{F}_p[\tau]/\tau^{p-1}}(\mu_h) \right ) \otimes_{\mathbb{F}_p[\tau]/\tau^{p-1}} \left (  \bigotimes_{h=1}^{j}  \Lambda_{\mathbb{F}_p[\tau]/\tau^{p-1}}(\lambda_h) \right ) \\
    \hookrightarrow \pi_{*,*}(MHH(M\mathbb{Z}/p)/\tau^{p-1})
\end{multline*}
such that the inclusion is an isomorphism in degrees smaller than or equal to $2p^j$.
\begin{rmk}\label{rmk:rows.columns.in.2p^j+1}
    Observe that this description implies that in the $E^{2p^j+1}$-page, the columns with indices from 1 to $2p^j$ are empty. Because the spectral sequence is first quadrant, we know that the rows from 1 to $2p^j-1$ must be empty as well. However, our knowledge of $\pi_{*,*}(\mathcal{A}(p)/\tau^{p-1})$ allows us to look further: in the $E^{2p^j+1}$-page, the rows with indices from 1 to $2p^{j+1}-3$ are already empty; in row $2p^{j+1}-2$ the element $\xi_{j+1}$ waits to be killed at a higher page.
\end{rmk}

\item[$H_t$] For all the other $t$'s, nothing happens: all the $d^t$-differentials are trivial, hence $E^{t+1}\cong E^t$. In particular, if $\bar{j}$ is the integer with $2p^{\bar{j}}<t< 2p^{\bar{j}+1}-1$, the $E^{t+1}$-page presents the same tensor product decomposition as the $E^{2p^{\bar{j}}+1}$-page. Moreover, we can extend our knowledge on the inclusion:
\begin{multline*}
    \left ( \bigotimes_{h=0}^{\bar{j}}  \Gamma_{\mathbb{F}_p[\tau]/\tau^{p-1}}(\mu_h) \right ) \otimes_{\mathbb{F}_p[\tau]/\tau^{p-1}} \left (  \bigotimes_{h=1}^{\bar{j}}  \Lambda_{\mathbb{F}_p[\tau]/\tau^{p-1}}(\lambda_h) \right ) \\
    \hookrightarrow \pi_{*,*}(MHH(M\mathbb{Z}/p)/\tau^{p-1})
\end{multline*}
proving that it is an isomorphism up to degree $t$. In particular, all columns in the $E^{t+1}$-page indexed by 1 to $t$ are empty.
\end{itemize}

\subsubsection{Pages where nothing happens}

We begin with the last case, which is the easiest. Let $2p^{j}<t< 2p^{j+1}-1$ and assume $H_{t-1}$. As we are dealing with a first quadrant spectral sequence, due to the tensor product decomposition of the $E^t$-page, any differential in the spectral sequence must have a counterpart in the bottom horizontal line. But from remark \ref{rmk:rows.columns.in.2p^j+1} the first non-zero row it can encounter is the $2p^{j+1}-2> t-1$: it is too high. So all the $d^t$-differentials are trivial, and $E^{t+1} \cong E^t$.

From the convergence requirement of the spectral sequence, there cannot be permanent cycles outside of degree $(0,0)$; any element in $E^{t+1}_{t,0,*}$ would not be able to support any non-trivial differentials, creating a permanent cycle. So $E^{t+1}_{t,0,*}\cong 0$ and hence, by the tensor product decomposition of the $E^{t+1}$-page, $E^{t+1}_{t,*,*}\cong 0$. But then the inclusion:
\begin{multline*}
    \left ( \bigotimes_{h=0}^{j}  \Gamma_{\mathbb{F}_p[\tau]/\tau^{p-1}}(\mu_h) \right ) \otimes_{\mathbb{F}_p[\tau]/\tau^{p-1}} \left (  \bigotimes_{h=1}^{j}  \Lambda_{\mathbb{F}_p[\tau]/\tau^{p-1}}(\lambda_h) \right ) \\
    \hookrightarrow \pi_{*,*}(MHH(M\mathbb{Z}/p)/\tau^{p-1})
\end{multline*}
is an isomorphism also in degree $t$, proving the induction step $H_t$.
\subsubsection{\texorpdfstring{The $E^{2p^j-1}$ page}{The E\^{}(2p\^{}i-1) page}}

The argument to prove $H_{2p^j-1}$ goes in parallel with that for the $E^{2p-1}$ page described in \ref{sec:E^{2p-1}}, of which is in fact a generalisation. For the inductive proof, we assume $H_{2p^j-2}$, in particular we assume to know that:
\begin{equation}\label{eq:E^2p^j-1.page}
\begin{gathered}
    E^{2p^j-1}_{*, 0, *} \cong \pi_{*,*}(MHH(M\mathbb{Z}/p)/\tau^{p-1})/\langle \upgamma_i(\mu_0), \ldots, \upgamma_i(\mu_{j-1}), \lambda_1, \ldots, \lambda_{j-1} \rangle_{i\geq 1}\\
    E^{2p^j-1}_{0, *, *} \cong 
    \pi_{*,*}(\mathcal{A}(p)/\tau^{p-1})/\langle \tau_0, \tau_1, \ldots, \tau_{j-1}, \xi_1, \ldots, \xi_{j-1} \rangle\\
    E^{2p^j-1}_{\star, \bullet, *} \cong E^{2p^j-1}_{\star,0,*} \otimes_{\mathbb{F}_{p}[\tau]/\tau^{p-1}} E^{2p^j-1}_{0,\bullet,*}
\end{gathered}
\end{equation}
and that there is a subalgebra:
\begin{multline*}
    \left ( \bigotimes_{h=0}^{j-1}  \Gamma_{\mathbb{F}_p[\tau]/\tau^{p-1}}(\mu_h) \right ) \otimes_{\mathbb{F}_p[\tau]/\tau^{p-1}} \left (  \bigotimes_{h=1}^{j-1}  \Lambda_{\mathbb{F}_p[\tau]/\tau^{p-1}}(\lambda_h) \right ) \\
    \hookrightarrow \pi_{*,*}(MHH(M\mathbb{Z}/p)/\tau^{p-1})
\end{multline*}
such that the inclusion is an isomorphism in degrees smaller than or equal to $2p^j-2$. In particular, we have that the columns $E^{2p^j-1}_{a,*,*}$ for $1 \leq a \leq 2p^j-2$ and the rows $E^{2p^j-1}_{*,b,*}$ for $1 \leq b \leq 2p^j-3$ are empty.

The above description of the page and the convergence hypothesis that the $d^{2p^j-1}$-differential produces an isomorphism:
\[
    {d^{2p^j-1}}: E^{2p^j-1}_{2p^j-1,0,*}\xrightarrow{\sim} E^{2p^j-1}_{2p^j-1,0,*} \cong (\mathbb{F}_{p}[\tau]/\tau^{p-1})\{\xi_j\}
\]

Hence there exists an element $\lambda_j\in E^{2p^j-1}_{2p^j-1,0,p^j-1}$ such that $E^{2p^j-1}_{2p^j-1,0,*} \cong (\mathbb{F}_{p}[\tau]/\tau^{p-1})\{\lambda_j\}$.

Observe that because of the Leibniz rule, this $d^{2p^j-1}$-differential kills all the multiples of $\xi_j$ in the zeroth column: for any $x \in E^{2p^j-1}_{0,*,*}$, we have a differential:
\[
    \lambda_j x \xrightarrow{d^{2p^j-1}} \xi_j x
\]
For $x\neq 0$, this differential is non-trivial because of the algebra description of the $E^{2p^j-1}$ page: on the one hand, the tensor decomposition of the $E^{2p^j-1}$-page \ref{eq:E^2p^j-1.page} assures that $\lambda_j x \neq 0$; on the other hand, $\xi_j x \neq 0$ because of the algebra description of the zeroth column.

Next we show that $\lambda_j^2=0$. As $\lambda_j$ survives up to the $E^{2p^j-1}$-page and has an odd degree, by the Leibniz rule for graded rings:
\[
    d^k(\lambda_j^2)=0 \text{ for } k \leq 2p^j-1.
\]
As again nothing but zero can survive in degrees different from zero if $\lambda_j^2\neq 0$ it has to support a non-trivial differential, longer than a $d^{2p^j-1}$ by what we have just seen. But $E^{2p^j-1}_{*,b,*}$ for $1 \leq b \leq 2p^j-3$ are empty by induction hypothesis, so will be empty in the higher pages as well: the only possibility left is to have a transgressive differential, so $k=2(2p^j-1)$. Now:
\[
    d^{2(2p^j-1)}(\lambda_j^2)\in E^{2(2p^j-1)}_{0,4p^j-3,2p^j-2}.
\]
The lowest degree $\mathbb{F}_{p}[\tau]/\tau^{p-1}$-module generators of $E^{2p^j-1}_{0,*,*}$ are:
\begin{equation*}
    \begin{gathered} 
    1 \text{ with (degree, weight): } (0,0) \\
    \xi_j \text{ with (degree, weight): } (2p^j-2, p^j-1) \\
    \tau_j \text{ with (degree, weight): } (2p^j-1,p^j-1) \\
    \xi_j^2 \text{ with (degree, weight): } (4p^j-4, 2p^j-2) \\
    \xi_j\tau_j \text{ with (degree, weight): } (4p^j-3, 2p^j-2) \\
    \xi_j^3 \text{ with (degree, weight): } (6p^j-6, 3p^j-3) \\
    \cdots \\
    \xi_{j+1} \text{ with (degree, weight): } (2p^{j+1}-2, p^{j+1}-1)
\end{gathered}
\end{equation*}
All the other generators have higher degrees. As one can see, the only possible recipient for $d^{2(2p^j-1)}(\lambda_j^2)$ is in the module generated by $\xi_j\tau_j$, which however vanishes after the $E^{2p^j-1}$-page. Because of this degree argument, all differentials out of $\lambda_j^2$ are trivial, hence $\lambda_j^2=0$.

To conclude, we show that an element is involved in a non-trivial $d^{2p^j-1}$ differential if and only if it is a multiple of $\lambda_j$, up to $d^{2p^j-1}$-cycles, or is a multiple of $\xi_j$. More precisely, we first see that any element $x \in E^{2p^j-1}_{*,0,*}$ can be written as $x= a \lambda_j + b$, with $d^{2p^j-1}(a)=d^{2p^j-1}(b)=0$. 
In fact let $x\in E^{2p^j-1}_{*,0,*}$ have $d^{2p^j-1}(x)= a \xi_1$. Consider the element $\lambda_l a$; one has $d^{2p^j-1}(\lambda_l a)=a \xi_1$ since $0=d^{2p^j-1}(d^{2p^j-1}(x))=d^{2p^j-1}(a)\xi_1$. Hence $b=x-\lambda_l a$ is a $d^{2p^j-1}$-cycle. Observe that, by the tensor decomposition of the $E^{2p^j-1}$-page, this characterises tensors starting from other rows: an elementary tensor $xw$ with $x\in E^{2p^j-1}_{*,0,*}$ and $w\in E^{2p^j-1}_{0,*,*}$ has a non-trivial $d^{2p^j-1}$-differential if and only if $x$ has a non-trivial $d^{2p^j-1}$-differential, and $d^{2p^j-1}(xw)=d^{2p^j-1}(x)w= a \xi_1 w$ with $a\in E^{2p^j-1}_{*,0,*}$ a $d^{2p^j-1}$-cycle. We can actually draw the above conclusion for any element of the $E^{2p^j-1}$-page, independently of the row: it supports a non-trivial differential if and only if it is a multiple of $\lambda_l$ up to some $d^{2p^j-1}$-cycle.

This proves that the ideal generated by $\lambda_j$ disappears from the $E^{2p^j}$-page.

Symmetrically, all elements in the image of the $d^{2p^j-1}$-differential are multiples of $\xi_j$; not all multiples of $\xi_j$ are, though, in the image of the $d^{2p^j-1}$-differential. However, consider an elementary tensor that is multiple of $\xi_j$, say $y=x c \xi_l^d$, with $x \in E^{2p^j-1}_{*,0,*}$, $d>0$ and $c \in E^{2p^j-1}_{0,*,*}$ not a multiple of $\xi_j$. By the previous point we can write $x= a \lambda_j + b$, with $d^{2p^j-1}(a)=d^{2p^j-1}(b)=0$, so that $y=( a \lambda_j + b) c \xi_j^d$. If $a \neq 0$, then $d^{2p^j-1}(y)=ac\xi_j^{d+1} \neq 0$. If $a =0$, the element $\lambda_j bc \xi_j^{d-1}$ supports a differential $d^{2p^j-1}(\lambda_j bc \xi_j^{d-1})=bc\xi_j^d=y$. Hence the image of the $d^{2p^j-1}$ consists precisely of those multiples of $\xi_j$ that have trivial $d^{2p^j-1}$-differential. This proves that also the ideal generated by $\xi_j$ disappears from the $E^{2p^j}$-page.

This concludes this step of the induction, proving hypothesis $H_{2p^j-1}$.
\subsubsection{\texorpdfstring{The $E^{2p^j}$ page}{The E\^{}(2p\^{}j) page}} \label{sec:E^{2p^j}}

Last, we prove $H_{2p^j}$ for any positive integer $j$. Observe that we already proved this for $j=0$ (section \ref{sec:E2}) and for $j=1$ (section \ref{sec:E^{2p}}); this step is, in fact, a generalisation of what happens in the $E^{2p}$ page: the proof follows the same arguments, with the additional complication of the generic exponent $j$ here.

Our induction hypothesis at this step is given by $H_{2p^j-1}$; in particular, recall the decomposition:
\begin{equation*}
\begin{gathered}
    E^{2p^j-1}_{*, 0, *} \cong \pi_{*,*}(MHH(M\mathbb{Z}/p)/\tau^{p-1})/\langle \upgamma_i(\mu_0), \ldots, \upgamma_i(\mu_{j-1}), \lambda_1, \ldots, \lambda_{j} \rangle_{i\geq 1}\\
    E^{2p^j-1}_{0, *, *} \cong 
    \pi_{*,*}(\mathcal{A}(p)/\tau^{p-1})/\langle \tau_0, \tau_1, \ldots, \tau_{j-1}, \xi_1, \ldots, \xi_{j} \rangle\\
    E^{2p^j}_{\star, \bullet, *} \cong E^{2p^j}_{\star,0,*} \otimes_{\mathbb{F}_{p}[\tau]/\tau^{p-1}} E^{2p^j}_{0,\bullet,*}
\end{gathered}
\end{equation*}
and the injection:
\begin{multline*}
    \left ( \bigotimes_{h=0}^{j-1}  \Gamma_{\mathbb{F}_p[\tau]/\tau^{p-1}}(\mu_h) \right ) \otimes_{\mathbb{F}_p[\tau]/\tau^{p-1}} \left (  \bigotimes_{h=1}^{j}  \Lambda_{\mathbb{F}_p[\tau]/\tau^{p-1}}(\lambda_h) \right ) \\
    \hookrightarrow \pi_{*,*}(MHH(M\mathbb{Z}/p)/\tau^{p-1})
\end{multline*}
which is an isomorphism up to degree $2p^j-1$. In particular, the rows $E^{2p^j}_{*,a,*}$ for $1 \leq a \leq 2p^j-2$ and the columns $E^{2p^j}_{b,*,*}$ for $1 \leq b \leq 2p^j-1$ are empty. 

To prove $H_{2p^j}$ we adopt the same strategy as in section \ref{sec:E^{2p}}; in particular we make four statements, for each $s \geq 1$:
\begin{enumerate}[start=1, label={$\arabic*_s.$}]
    \item There exist elements $\mu_j=\upgamma_1 (\mu_j),\, \upgamma_2 (\mu_j), \ldots \upgamma_s (\mu_j)$, with $\upgamma_i (\mu_j) \in E^{2p^j}_{(2p^j i, 0, i(p^j-1))}$, respecting the multiplicative relations of divided powers in characteristic $p$ expressed by homogeneous equations of degree at most $2p^j s$ and  such that $d^{2p^j}(\upgamma_i (\mu_j))=\upgamma_{i-1} (\mu_j) \tau_j$. Recall that we use the convention $\upgamma_0 (\mu_j)=1$.
    
    \item The identity $(\upgamma_i (\mu_j))^p=0$ is known to hold for all integers $i \leq s$ but the largest power of $p$. 
    
    \item Up to degree $2p^js$, every element in $E^{2p^j}_{*,0,*}$ can be written as $\sum_{i=0}^s \alpha_i \cdot \upgamma_i(\mu_j)$, with $d^{2p^j}(\alpha_i)=0$ and $\upgamma_0(\mu_j)=1$. In particular, every non-trivial $d^{2p^j}$ differential originates in an element of the form $\sum_{i=1}^s \alpha_i \cdot \upgamma_i(\mu_j)$, with $d^{2p^j}(\alpha_i)=0$.\\
    Moreover $\sum_{i=0}^s \alpha_i \cdot \upgamma_i(\mu_j)=0$ if and only if $\alpha_i=0$ for all $i$.
    
    \item Every row starting with a multiple of $\tau_j$ vanishes after the $E^{2p^j}$-page, up to degree $2p^j(s-1)$. In particular, all such elements are targets of a non-trivial $d^{2p^j}$-differential.
    \end{enumerate}

The proof is again divided into three parts.

\begin{enumerate}[label={\Alph*.}]
    \item We first show that, for each $s \geq 1$,
    \[
        1_s, 3_{s-1}\Rightarrow 3_s \Rightarrow 4_s. 
    \]
    
    \item Then we show that it is enough to prove statements $1$ and $2$ on the powers of $p$, namely the implication:
    \[
        1_1,\ldots, 1_{p^k},2_1,\ldots, 2_{p^k} \Rightarrow 1_s,2_s \text{ for all } s \leq p^{k+1}-1.
    \]
    \item Finally, we prove $1_{p^{k}}$ and $2_{p^{k}}$ with an induction argument; the base step consists of showing $1_1$ and $2_1$(which is a trivial statement), while the induction one is:
    \[
        1_{p^{k}-1},2_{p^{k}-1},3_{p^{k}-1},4_{p^{k}-1} \Rightarrow 1_{p^{k}},2_{p^{k}} \text{ for all } k\geq 1.
    \]
\end{enumerate}

The proofs of $A.$ and $B.$ are the same as in section \ref{sec:E^{2p}}, up to relabelling of $\mu_1$ with $\mu_j$ and $\tau_1$ with $\tau_j$, and adjusting the corresponding degrees and weights, so we are not presenting them again.

To prove $C.$, we proceed by induction. Our strategy closely follows the one in \ref{sec:E^{2p}}.

Statement $1_1$ constitutes the base step: it reduces to the existence of an element $\mu_j=\upgamma_1 (\mu_j) \in E^{2p^j}_{2p^j,0,p^j-1}$ and a $2p^j$-differential: $d^{2p^j}(\mu_j)=\tau_j$. This immediately follows from the structure of the $E^{2p^j}$-page: the element $\tau_j$ is $\tau_j \neq 0 \in E^{2p^j}_{0, 2p^j-1, p^j-1}$, so the transgressive $d^{2p^j}$-differential:
\[
    d^{2p^j}_{2p^j,0,*}: E^{2p^j}_{2p^j,0,*} \to E^{2p^j}_{0, 2p^j-1,*} \cong (\mathbb{F}_{p}[\tau]/\tau^{p-1})\{\tau_j\} 
\]
has to be surjective. On the other hand, as we do not want to produce permanent cycles in degree $(2p^j,0,*)$ either, it has to be injective as well, hence:
\[
    E^{2p^j}_{2p^j,0,*} \cong (\mathbb{F}_{p}[\tau]/\tau^{p-1})\{\mu_j\}
\]
with $|\mu_j|=(2p^j,p^j-1)$.

We are now ready for the induction step. We begin by recalling what $E^{2p^j}_{0,*,*}$ looks like in low degrees; here the $\mathbb{F}_{p}[\tau]/\tau^{p-1}$-module generators with the smallest degrees are:
\begin{equation}\label{eq:degr.weights.E2p^j}
    \begin{gathered} 
    1 \text{ with (degree, weight): } (0,0) \\
    \tau_j \text{ with (degree, weight): } (2p^j-1, p^j-1) \\
    \xi_{j+1} \text{ with (degree, weight): } (2p^{j+1}-2,p^{j+1}-1) \\
    \tau_{j+1} \text{ with (degree, weight): } (2p^{j+1}-1,p^{j+1}-1) \\
\end{gathered}
\end{equation}

Observe that elements with higher degrees here have higher weights as well. Moreover, as weights in $\mathbb{F}_p[\tau]/\tau^{p-1}$ vary only from $0$ to $-p+2$, weights present the same span in each of the modules mentioned above.

To prove inductively $1_{p^k}$ for all larger $k$, as before we split the claim into three smaller statements regarding the element $\upgamma_{p^k-1}(\mu_j) \tau_j$:
\begin{enumerate}[label={(\roman*)}]
    \item No polynomial of elements of degree smaller then $p^k(2p^j)$ hits it with a $d^{2p^j}$ differential
    \item It has trivial $d^{2p^j}$
    \item It cannot support a longer non-trivial differential.
\end{enumerate}
These statements, together with the fact that $\upgamma_{p^k-1}(\mu_j) \tau_j$ is not a permanent cycle, imply the existence of a new independent class $\upgamma_{p^k}(\mu_j) \in E^{2p^j}_{p^k(2p^j),0, p^k(p^j-1)}$ with
\[
    d^{2p^j}(\upgamma_{p^k}(\mu_j))=\upgamma_{p^k-1}(\mu_j) \tau_j,
\]
in other words, statement $1_{p^k}$.

To prove $(i)$, one proceeds by contradiction, supposing such a polynomial exists. As the smallest non-zero first degree in $E^{2p^j}$ is $2p^{j}$, all factors in the monomials composing such polynomial have a degree smaller than or equal to $(p^k-1)2p^j$. Hence we can apply induction hypotheses $1_{p^k-1}$, $2_{p^k-1}$, $3_{p^k-1}$ and write it as:
\[
    \sum_{s=1}^{p^j-1}a_s \upgamma_s(\mu_j)+e(\upgamma_{p^{j-1}}(\mu_j))^p
\]
with $d^{2p^j}(a_s)=0$ and $e \in \mathbb{F}_p[\tau]/\tau^{p-1}$.
One then proceeds as in the relevant part of section \ref{sec:E^{2p}}, applying in this case the $d^{2p^j}$-differential, to reach a contradiction.

Statement $(ii)$ is just the observation that:
\[
    d^{2p^j}(\upgamma_{p^k-1}(\mu_j) \tau_j)=\upgamma_{p^k-2}(\mu_j) \tau_j \tau_j =0
\]
as $\tau_j^2=0$.

For the proof of $(iii)$  we assume by contradiction that $\upgamma_{p^k-1}(\mu_j) \tau_j$ supports a longer, non-trivial differential. Arguing as in section \ref{sec:E^{2p}}, one finds a positive integer $n$ and elements:
\begin{equation*}
\begin{gathered}
    \alpha_1 \in E^{2p^j}_{(p^k-1)2p^j- m_1,0, p^k(p^j-1)- w_1}\\
    \beta_1 \in E^{2p^j}_{0, 2p^j+ m_1-2, w_1}\\
    \alpha_h \in E^{2p^j}_{(p^k-1)2p^j-\sum_{t=1}^h m_k,0, p^k(p^j-1)-\sum_{t=1}^h w_k}\\
    \beta_h \in E^{2p^j}_{0, m_h-1, w_h}
\end{gathered}
\end{equation*}
for all $h=2, \ldots n$ with the constraints:
\begin{equation}\label{eq:inequalities.deg.wei.muj}
\begin{gathered}
    \alpha_n=1 \in E^{2p^j}_{0,0,0}\\
    m_1 \geq 2p^{j+1}-2p^j\\
    w_1 \geq p^{j+1}-p+1\\
    m_h \geq 2p^{j+1}-1\\
    w_h \geq p^{j+1}-p+1.
\end{gathered}
\end{equation}
for all $h=2,\ldots,n$. Moreover, observe that all the $\alpha_h \beta_h$ live in the region controlled by the induction hypothesis $4_{p^k-1}$: as all the $\beta_h$ survive after the $E^{2p^j}$-page, they are homogeneous $\mathbb{F}_p[\tau]/\tau^{p-1}$-polynomials in which $\tau_0,\ldots,\tau_j$ or $\xi_1,\ldots,\xi_j$ do not appear. Hence their degrees and weights are the same as those of monomials in $\pi_{*,*}(\mathcal{A}(p)/\tau^{p-1})$ where $\tau_0,\ldots,\tau_j$ and $\xi_1,\ldots,\xi_j$ do not appear; in particular, we can apply \ref{cor:rough.est.Q(i)} with $i=j+1$ and get the estimates:
\begin{equation*}
\begin{gathered}
    \frac{2p^{j+1}-2p+1}{4p^{j+1}-4} \leq \frac{w_1}{2p^j+ m_1-2} \leq \frac{1}{2} \\
    \frac{2p^{j+1}-2p+1}{4p^{j+1}-4} \leq \frac{w_h}{m_h-1} \leq \frac{1}{2}
\end{gathered}
\end{equation*}
for all $h=2,\ldots,n$. In particular, we are interested in the estimates provided by the left-hand sides:
\begin{equation*}
\begin{gathered}
    \frac{2p^{j+1}-2p+1}{4p^{j+1}-4} (2p^j+ m_1-2) \leq w_1\\
    \frac{2p^{j+1}-2p+1}{4p^{j+1}-4}(m_h-1) \leq w_h
\end{gathered}
\end{equation*}
Summing on all indices one gets:
\begin{equation}\label{eq:sum.on.all.muj}
    \frac{2p^{j+1}-2p+1}{4p^{j+1}-4}(2p^j-n+1+\sum_{h=1}^n m_h) \leq \sum_{h=1}^n w_h
\end{equation}

As $a_n=1$, it has bi-degree $(0,0)$, so we get equations:
\begin{equation}\label{eq:mk.wk.muj}
    \begin{cases}
        \sum_{k=1}^n m_k = (p^k-1)2p^j\\
        \sum_{k=1}^n w_k = p^k(p^j-1).
    \end{cases}
\end{equation}

Substitute in \ref{eq:sum.on.all.muj}:
\begin{multline*}
    \frac{2p^{j+1}-2p+1}{4p^{j+1}-4}(2p^j-n+1+(p^k-1)2p^j) \leq p^k(p^j-1)\\
    (2p^{j+1}-2p+1)(2p^j-n+1+(p^k-1)2p^j) \leq p^k(p^j-1)(4p^{j+1}-4)\\
    (2p^{j+1}-2p+1)(-n+1+p^{k+j}) \leq 4p^k(p^j-1)(p^{j+1}-1)\\
    -n+1+p^{k+j} \leq \frac{4p^k(p^j-1)(p^{j+1}-1)}{2p^{j+1}-2p+1}
\end{multline*}
As $2p^{j+1}-2p+1 > 2p^{j+1}-2p$:
\begin{multline*}
    -n+1+p^{k+j} \leq \frac{4p^k(p^j-1)(p^{j+1}-1)}{2p^{j+1}-2p}= 2p^{k-1}(p^{j+1}-1)\\
    -n \leq 2p^{k-1}(p^{j+1}-1)-1-2p^{k+j}=2p^{k+j}-2p^{k-1}-1-2p^{k+j}\\
    -n \leq -2p^{k-1}-1 \leq -2p^{k-1}
\end{multline*}
Hence $n \geq 2p^{k-1}$.

Now go back to the inequalities in \ref{eq:inequalities.deg.wei.muj} and combine them with \ref{eq:mk.wk.muj}:
\begin{equation*}
    \begin{aligned}
        p^{k+j}-p^k=p^k(p^j-1) &= \sum_{k=1}^n w_k \geq \sum_{k=1}^n (p^{j+1}-p+1) =n(p^{j+1}-p+1) \\
        &\geq 2p^{k-1}(p^{j+1}-p+1)=2p^{k+j}-2p^k+2p^{k-1}
    \end{aligned}
\end{equation*}
\[
    0 \geq p^{k+j}-p^{k}+2p^{k-1}>0
\]
for all $j>0$. This contradiction concludes this argument, proving the existence of $\upgamma_{p^k} (\mu_j)$.

We need now to show $2_{p^k}$, that is, $(\upgamma_{p^{k-1}} (\mu_j))^p=0$. We adopt the same strategy as in \ref{sec:E^{2p}} (or as in the previous point): if $(\upgamma_{p^{k-1}} (\mu_j))^p\neq 0$, then it survives up to the $E^{2p^j}$-page. Moreover, by the Leibniz rule:
\[
    d^{2p^j}((\upgamma_{p^{k-1}} (\mu_j))^p)=p (\upgamma_{p^{k-1}} (\mu_j))^{p-1}\upgamma_{p^{k-1}-1} (\mu_j) \tau_j=0
\]
because of the characteristic of the base ring.
Hence, if non-trivial, it has to support a differential longer than a $d^{2p^j}$; reasoning as in \ref{sec:E^{2p}}, we find a positive integer $n$ and, for all $1 \leq h \leq n$, elements:
\begin{equation*}
\begin{gathered}
    \alpha_h \in E^{2p^j}_{2p^{j+k}-\sum_{t=1}^h m_k,0, p^k(p^j-1)-\sum_{t=1}^h w_k}\\
    \beta_h \in E^{2p^j}_{0, m_h-1, w_h}
\end{gathered}
\end{equation*}
with the constraints:
\begin{equation}\label{eq:geq.mh.wh.muj^p}
\begin{gathered}
    \alpha_n=1 \in E^{2p^j}_{0,0,0}\\
    m_h \geq 2p^{j+1}-1\\
    w_h \geq p^{j+1}-p+1
\end{gathered}
\end{equation}
for all $h=1, \ldots n$. By the same argument used in the proof of $1_{p^j}$, we find the constraint:
\[
    \frac{2p^{j+1}-2p+1}{4p^{j+1}-4}(m_h-1) \leq w_h
\]
for all $1 \leq h \leq n$. By summing on all $h$ and rearranging the factors:
\begin{equation}\label{eq:sum.deg.wei.muj^p}
\begin{gathered}
    (2p^{j+1}-2p+1)\left ( \left (\sum_{h=1}^n m_h \right )-n \right) \leq \left (\sum_{h=1}^n w_h \right)(4p^{j+1}-4)
\end{gathered}
\end{equation}

As before, by the condition $\alpha_n=1$ we get equations on the degree and weight:
\begin{equation}\label{eq:deg.wei.muj^p}
    \begin{cases}
        2p^{j+k}=\sum_{t=1}^n m_k\\
        p^k(p^j-1)=\sum_{t=1}^n w_k.
    \end{cases}
\end{equation}
Substitute in \ref{eq:sum.deg.wei.muj^p}:
\[
    (2p^{j+1}-2p+1)(2p^{j+k}-n) \leq p^k(p^j-1)(4p^{j+1}-4)
\]
As $2p^{j+1}-2p+1 \geq 2p^{j+1}-2p=2p(p^j-1)$:
\begin{equation*}
\begin{gathered}
    2p^{j+k}-n \leq \frac{p^k(p^j-1)(4p^{j+1}-4)}{2p^{j+1}-2p}=2p^{k-1}(p^{j+1}-1)=2p^{j+k}-2p^{k-1}\\
    n \geq 2p^{k-1}
\end{gathered}
\end{equation*}

Now consider again the equality on degrees from \ref{eq:deg.wei.muj^p}, substitute the constraints from \ref{eq:geq.mh.wh.muj^p} and the newly found estimate on $n$:
\begin{equation*}
\begin{aligned}
    2p^{j+k}=\sum_{t=1}^n m_k &\geq \sum_{t=1}^n (2p^{j+1}-1)= (2p^{j+1}-1) n \\
    & \geq (2p^{j+1}-1)2p^{k-1}=4p^{j+k}-2p^{k-1}
\end{aligned}
\end{equation*}
Hence:
\[
2p^{j+k}\leq 2p^{k-1}
\]
which is a clear contradiction for all $j \geq 0$. This proves assertion $2_{p^k}$ and concludes our proof. 

By assertions $3_s$ and $4_s$, we deduce what happens in the $E^{2p^j}$-page: the $d^{2p^j}$-differential has image the ideal generated by $\tau_j$; moreover, all the non-trivial differentials are generated by the divided powers of $\mu_j$. Observe that, as the spectral sequence is first quadrant and the $E^{\infty}$-page is non-trivial only at the origin, there cannot be any other element in $\pi_{*,*}(MHH(M\mathbb{Z}/p)/\tau^{p-1})$ in degree less than or equal to $2p^j$ but those we have identified so far. Hence, $H_{2p^j}$ holds.

This concludes the proof regarding the behaviour of the spectral sequence and the structure of $\pi_{*,*}(MHH(M\mathbb{Z}/p)/\tau^{p-1})$.

\clearpage
\relax

\printbibliography[heading=bibintoc]

\end{document}